\documentclass[11pt,twoside,a4paper]{article}

\usepackage{amssymb}
\usepackage{amsmath}
\usepackage{amsthm}
\usepackage{color}
\usepackage{mathrsfs}

\usepackage{amsfonts}
\usepackage{amssymb}
\usepackage{amsmath}
\usepackage{amsthm}
\usepackage{bbm}
\usepackage{verbatim}
\usepackage{url}

\allowdisplaybreaks

\DeclareMathAlphabet{\mathpzc}{OT1}{pzc}{m}{it}

\newcommand{\Norm}[2]{\|#1\|_{#2}}

\def\Z{{\mathbb Z}}

\def\ch{{\mathcal H}}

\def\X{{\mathrm{X}}}

\def\fz{\infty}

\def\lz{\lambda}

\def\wz{\widetilde}

\def\ls{\lesssim}
\def\gs{\gtrsim}

\def\dint{\displaystyle\int}

\def\dfrac{\displaystyle\frac}

\def\r{\right}
\def\lf{\left}

\newtheorem{thm}{Theorem}[section]
\newtheorem{lem}[thm]{Lemma}
\newtheorem{prop}[thm]{Proposition}
\newtheorem{defn}[thm]{Definition}

\newcommand{\R}{\mathbb R}

\newcommand{\intav}{-\!\!\!\!\!\!\int}

\DeclareMathOperator*{\esssup}{ess\ sup}

\numberwithin{equation}{section}

\begin{document}

\arraycolsep=1pt

\title{\Large\bf  Two weight commutators
 on spaces of homogeneous type\\ and applications}
\author{Xuan Thinh Duong, Ruming Gong, Marie-Jose S. Kuffner, Ji Li, \\ Brett D. Wick and Dongyong Yang}

\date{}
\maketitle

\begin{center}
\begin{minipage}{13.5cm}\small

{\noindent  {\bf Abstract:}\
In this paper,
 we establish the two weight commutator of Calder\'on--Zygmund operators in the sense of Coifman--Weiss on spaces of homogeneous type, by studying the weighted Hardy and BMO space for $A_2$ weight and by proving the sparse operator domination of commutators.
 The main tool here is the Haar basis and the adjacent dyadic systems on spaces of homogeneous type, and the construction of a suitable version of a sparse operator on spaces of homogeneous type.
 As applications, we provide a two weight commutator theorem (including the high order commutator) for the following Calder\'on--Zygmund operators: Cauchy integral operator on $\mathbb R$, Cauchy--Szeg\"o projection operator on Heisenberg groups, Szeg\"o projection operators on a family of unbounded weakly pseudoconvex domains,   Riesz transform associated with the sub-Laplacian on stratified Lie groups, as well as the Bessel Riesz transforms (one-dimension and high dimension).

}

\end{minipage}
\end{center}

{\footnotesize
\tableofcontents
}%
%
\bigskip


\medskip


\section{Introduction and Statement of Main Results\label{s1}}

It is well-known that Coifman, Rochberg and Weiss \cite{crw} characterized the boundedness of the commutator $[b,R_i]$ acting on  Lebesgue spaces in terms of BMO with $R_j = \frac{\partial}{\partial x_j} \Delta^{-1/2}$ the $j$-th Riesz transform on the Euclidean space $\mathbb R^n$, which extended the work of Nehari \cite{Ne} about Hankel operators from complex setting to the real setting $\mathbb R^n$. Later, Bloom in \cite{B} established the characterisation of weighted BMO in terms of boundedness of commutators $[b,H]$ in the two weight setting, where $H$ is the Hilbert transform on $\R$.

Recent remarkable results were achieved by Holmes--Lacey--Wick \cite{HLW} giving the characterisation of weighted BMO space on $\R^n$ in terms of boundedness of commutators of Riesz transforms, and by Lerner--Ombrosi--Rivera-R\'ios \cite{LOR,LOR2} in terms of boundedness of commutators of Calder\'on--Zygmund operators with homogeneous kernels $\Omega({x\over |x|}) {1\over |x|^n}$,
and Hyt\"onen \cite{Hyt} in terms of boundedness of commutators of a more general version of Calder\'on--Zygmund operators
and weighted BMO functions on $\R^n$.
Meanwhile, two weight commutator
has also been studied extensively  in different settings, see for example \cite{DHLWY, DLLWW, GLW, IR}.

We note that to get the lower bound of the two weight commutator for Riesz transforms (for Hilbert in one dimension) in terms of weighted BMO space,
the first proofs used spherical harmonic expansion for the Riesz (Hilbert) kernels, which relies on properties of the Fourier transform of the Riesz (Hilbert)  kernels. A similar method of expansion of the Riesz transform associated with Neumann Laplacian  was used in \cite{DHLWY} for a larger class of $A_p$ weights and for the BMO space associated with Neumann Laplacian which is strictly larger than classical BMO. In \cite{LOR}, concerning the two weight commutator for Calder\'on--Zygmund operators associated with homogeneous kernel $\Omega({x\over |x|}){1\over|x|^{n}}$, the proof of the lower bound was obtained by assuming suitable conditions on the homogeneous function $\Omega$, see also \cite{GHWY, GLW}.  More recently, Hyt\"onen \cite{Hyt} studied the two weight commutator for Calder\'on--Zygmund operators and proposed a condition denoted by the ``non-degenerate Calder\'on--Zygmund kernel'', then the  the lower bound was obtained by a construction of factorisation. Also, in \cite{DLLW, DLLWW}, they established a version of a pointwise kernel lower bound for the Riesz transform associated to sub-Laplacian on stratified Lie groups, which covers the Heisenberg groups, and used this kernel lower bound to obtain the two weight commutator result following the idea in \cite{LOR}.

However, there are a few other important Calder\'on--Zygmund operators (not built on the Euclidean space setting) whose kernels do not have connection to the Fourier transform and are not of homogeneous type such as $\Omega({x\over |x|}){1\over|x|^{n}}$. Moreover,
whether the kernels fall into Hyt\"onen's ``non-degenerate Calder\'on--Zygmund kernel'' has not been studied before and hence the two weight commutator estimates and higher order commutator are unknown.

For example,  the Riesz transform from Muckenhoupt--Stein \cite{MS}: $R_\lambda:= -{d\over dx} (\Delta_\lambda)^{-{1\over2}}$, associated with the Bessel operator on $\mathbb R_+$:
\begin{align*}
\Delta_\lambda:= -{d^2\over dx^2} -{2\lambda\over x}{d\over dx}, \quad x>0, \lambda>-{1\over2},
\end{align*}
and the  Riesz transform $ R_{\lambda,j} = {d\over dx_j} (\Delta_\lambda^{(n+1)})^{-{1\over2}},\ j=1,\ldots,n+1$, associated with the Bessel operator $\Delta_\lambda^{(n+1)}$ on $\mathbb R^{n+1}_+$    studied in Huber \cite{Hu}:
\begin{align*}
\Delta_\lambda^{(n+1)} = -{d^2\over dx_1^2}\cdots-{d^2\over dx_n^2} -{d^2\over dx_{n+1}^2} -{2\lambda\over x_{n+1}} {d\over dx_{n+1}}.
\end{align*}

Another example is the  Cauchy--Szeg\"o projection operator $\mathcal C$ (for all the notation below we refer to Section 2 in Chapter XII in Stein \cite{s93}), which is the orthogonal projection from $L^2(b\mathcal U^n)$ to the subspace of functions
$\{F^b\}$ that are boundary values of functions $F\in \mathcal H^2(\mathcal U^n)$.  The associated Cauchy--Szeg\"o kernel is as follows.
$$ \mathcal C (f)(x) = \int_{\mathbb H^n} K(y^{-1}\circ x) f(y)dy, $$
where
$ K(x) = -{\partial\over \partial t} \Big({c\over n} [t+i|\zeta|^2]^{-n}\Big)\quad {\rm for}\ \ x=[\zeta,t]\in \mathbb H^n = \mathbb C^n\times\mathbb R$.

%
%

Then it is natural to study the following question:  is there a setting, by which the characterisation of two weight commutators and the related BMO space for Calder\'on--Zygmund operators $T$ can be obtained, that can be applied to Calder\'on--Zygmund operators such as the Bessel Riesz transform, the Cauchy--Szeg\"o projection operator on Heisenberg groups, and many other examples?

To address this question we work in a general setting:  spaces of homogeneous type introduced by
Coifman and Weiss in the early 1970s, in~\cite{CW1}, see also \cite{cw77}.
%
%
We say
that $(X,d,\mu)$ is a {space of homogeneous type} in the
sense of Coifman and Weiss if $d$ is a quasi-metric on~$X$
and $\mu$ is a nonzero measure satisfying the doubling
condition. A \emph{quasi-metric}~$d$ on a set~$X$ is a
function $d: X\times X\longrightarrow[0,\infty)$ satisfying
(i) $d(x,y) = d(y,x) \geq 0$ for all $x$, $y\in X$; (ii)
$d(x,y) = 0$ if and only if $x = y$; and (iii) the
\emph{quasi-triangle inequality}: there is a constant $A_0\in
[1,\infty)$ such that for all $x$, $y$, $z\in X$, 
\begin{eqnarray}\label{eqn:quasitriangleineq}
    d(x,y)
    \leq A_0 [d(x,z) + d(z,y)].
\end{eqnarray}
We say that a nonzero measure $\mu$ satisfies the
\emph{doubling condition} if there is a constant $C_\mu$ such
that for all $x\in X$ and $r > 0$,
\begin{eqnarray}\label{doubling condition}
   \mu(B(x,2r))
   \leq C_\mu \mu(B(x,r))
   < \infty,
\end{eqnarray}
where $B(x,r)$ is the quasi-metric ball by $B(x,r) := \{y\in X: d(x,y)
< r\}$ for $x\in X$ and $r > 0$.
We point out that the doubling condition (\ref{doubling
condition}) implies that there exists a positive constant
$n$ (the \emph{upper dimension} of~$\mu$)  such
that for all $x\in X$, $\lambda\geq 1$ and $r > 0$,
\begin{eqnarray}\label{upper dimension}
    \mu(B(x, \lambda r))
    \leq  C_\mu\lambda^{n} \mu(B(x,r)).
\end{eqnarray}
Throughout this paper we assume that $\mu(X)=\infty$ and that $\mu(\{x_0\})=0$ for every $x_0\in X$.

We now recall the
singular integral operator on spaces of homogeneous type in the sense of Coifman and Weiss.
\begin{defn}\label{def 1}
We say that $T$ is a Calder\'on--Zygmund operator on $(X,d,\mu)$ if $T$ is bounded on $L^2(X)$ and has the associated kernel $K(x,y)$ such that
$T(f)(x)=\int_{X} K(x,y)f(y) d\mu(y)$ for any $x\not\in {\rm supp}\,f$, and $K(x,y)$  satisfies the following estimates: for all $x\not= y$,
\begin{eqnarray}\label{size of C-Z-S-I-O}
    |K(x, y)|
    \leq {{C}\over {V(x, y)}},
\end{eqnarray}
and for $d(x, x')\leq (2A_0)^{-1} d(x, y)$,
\begin{eqnarray}\label{smooth of C-Z-S-I-O}
    |  K(x, y) - K(x', y) |+|  K(y,x) - K(y,x') |
    \leq {C\over V(x,y)}\omega\Big({d(x, x')\over d(x,y)}\Big),
\end{eqnarray}
where $V(x,y)=\mu(B(x,d(x,y)))$ and by the doubling condition we have that $V(x,y)\approx V(y,x)$,
$\omega:[0, 1]\to [0,\infty)$ is continuous, increasing, subadditive, $\omega(0) = 0$.
\end{defn}

We say that $\omega$ satisfies  the Dini condition if $\int_0^1 \omega(t){dt\over t}<\infty$.

Let $T$ be a Calder\'on--Zygmund operator on $X$.
Suppose $b\in L^1_{\rm loc}(X)$ and $f\in L^p(X)$. Let $[b, T]$ be the commutator defined by
\begin{equation*}
[b, T]f(x):= b(x)T( f)(x)-T(bf)(x).
\end{equation*}
The iterated commutators $T_b^m$, $m\in\mathbb N$, are defined inductively by
\begin{equation*}
T_b^mf(x):= [b, T^{m-1}_b]f(x),\quad T_b^1f(x):=[b, T]f(x).
\end{equation*}

Next we use $A_p$, $1\leq p\leq\infty$, to denote the Muckenhoupt weighted class on $X$ (see the precise definition of $A_p$ in Section 2), and the weighted BMO on $X$ is defined as
follows (the Euclidean version of weighted BMO was first introduced by Muckenhoupt and Wheeden \cite{MW76}).
\begin{defn}\label{MWbmo}
Suppose $w \in A_\infty$.
A function $b\in L^1_{\rm loc}(X)$ belongs to
${\rm BMO}_w(X)$ if
\begin{equation*}
\|b\|_{{\rm BMO}_w(X)}:=\sup_{B}{1\over w(B)}\dint_{B}
\lf|b(x)-b_{B}\r| \,d\mu(x)<\fz,
\end{equation*}
where $b_{B}:= {1\over \mu(B)}\int_B b(x)d\mu(x)$ and the supremum is taken over all balls $B\subset X$.
\end{defn}

Our first main result is the following theorem.
\begin{thm}\label{thm1}
Suppose $1<p<\infty$, $\lambda_1,\lambda_2\in A_p$, $\nu:= \lambda_1^{1\over p}\lambda_2^{-{1\over p}}$ and $m\in\mathbb N$. Suppose $b\in {\rm BMO}_{\nu^{1\over m}}(X)$. Then for any Calder\'on--Zygmund operator $T$ as in Definition \ref{def 1} with $\omega$ satisfying the Dini condition,
there exists a positive constant $C_1$ such that
\begin{equation*}
\| T_b^m : L^p_{\lambda_1}(X) \rightarrow L^p_{\lambda_2}(X) \|  \le C_1 \| b \|^m _{{\rm BMO}_{\nu^{1\over m}}(X)} \Big([\lambda_1]_{A_p}[\lambda_2]_{A_p}\Big)^{ {m+1\over2}\cdot \max\{1,{1\over p-1}\}}.
\end{equation*}
\end{thm}
To obtain the upper bound, we characterise the sparse system and then use the idea from \cite{LOR} to build a suitable version of a sparse operator on a space of homogeneous type. 
Here we apply the tool of adjacent dyadic systems from \cite{HK}, the explicit construction of Haar basis from \cite{KLPW}, and we have to allow suitable overlapping for the sparse sets due to the partition and  covering of the whole space via quasi-metric balls.

To consider the lower bound of the commutator, we assume that  the Calder\'on--Zygmund operator $T$ as in Definition \ref{def 1} with $\omega$ satisfying $\omega(t)\to0$ as $t\to0$, and that $T$ satisfies the following ``non-degenerate''
condition:

\noindent {\it There exist positive constant $c_0$ and $\overline C$ such that for every $x\in X$ and $r>0$, there exists $y\in B(x, \overline C r)\backslash B(x,r)$ satisfying
\begin{equation}\label{e-assump cz ker low bdd weak}
|K(x, y)|\geq \frac1{c_0\mu(B(x,r))}.
\end{equation}
}

Note that in $\R^n$, this ``non-degenrated'' condition was first proposed in \cite{Hy}, and
a similar assumption on the behaviour of the kernel  lower bound  was proposed in \cite{LOR}.
On stratified Lie groups, a similar condition of the Riesz transform kernel  lower bound  was verified in \cite{DLLW}.
 %
Then we have the following lower bound.
\begin{thm}\label{thm2}
Suppose $1<p<\infty$, $\lambda_1,\lambda_2\in A_p$, $\nu:= \lambda_1^{1\over p}\lambda_2^{-{1\over p}}$ and $m\in\mathbb N$. Suppose $b\in L^1_{\rm loc}(X)$ and that $T$ is a Calder\'on--Zygmund operator as in Definition \ref{def 1} and satisfies the non-degenrated condition \eqref{e-assump cz ker low bdd weak}.
Also suppose that  $T_b^m$ is bounded from
$ L^p_{\lambda_1}(X) $ to $ L^p_{\lambda_2}(X)$.
Then $b\in {\rm BMO}_{\nu^{1\over m}}(X)$, and there exists a positive constant $C_2$ such that
\begin{equation*}
 \| b \| _{{\rm BMO}_{\nu^{1\over m}}(X)}^m\leq C_2 \| T_b^m : L^p_{\lambda_1}(X) \rightarrow L^p_{\lambda_2}(X) \|.
\end{equation*}
\end{thm}


Based on the characterisation of ${\rm BMO}_\nu(X)$  via commutators $T^1_b=[b,T]$, we further have the weak factorisation for the 
weighted Hardy space $H_\nu^1(X)$ as follows.
\begin{thm}\label{thm5}
Suppose $1<p<\infty$, $\lambda_1,\lambda_2\in A_p$ and $\nu:= \lambda_1^{1\over p}\lambda_2^{-{1\over p}}$.  Let $p'$ be the conjugate of $p$
and $\lambda'_2:=\lambda_2^{-{1\over p-1}}$. For any $f\in H^1_\nu(X)$, there exist numbers $\{\alpha^k_j\}_{k,\,j}$, functions $\{g^k_j\}_{k,\,j}\subset L^p_{\lambda_1}(X)$ and $\{h^k_j\}_{k,\,j}\subset L^{p'}_{\lambda'_2}(X)$ such that
\begin{equation}\label{represent of H1}
f=\sum_{k=1}^\infty\sum_{j=1}^\infty \alpha^k_j\,\Pi\big(g^k_j,h^k_j\big)
\end{equation}
in $H^1_\nu(X)$, where the operator $\Pi$ is defined as follows: for $g\in L^p_{\lambda_1}(X)$ and $h\in L^{p'}_{\lambda'_2}(X)$,
\begin{equation*}
\Pi(g,h):=g T h- h T^* g,
\end{equation*}
where $T$ is a Calder\'on--Zygmund operator as in Definition \ref{def 1} and satisfies the non-degenrated condition \eqref{e-assump cz ker low bdd weak}.
Moreover, we have
\begin{eqnarray}\label{H1norm}
\|f\|_{H^1_\nu(X)}&\approx&\inf\lf\{\sum_{k=1}^\fz\sum_{j=1}^\infty\lf|\alpha^k_j\,\r|\lf\|g^k_j\r\|_{L^p_{\lambda_1}(X)}\lf\|h^k_j\r\|_{L^{p'}_{\lambda'_2}(X)}: f=\sum_{k=1}^\infty\sum_{j=1}^\infty \alpha^k_j\,\Pi\lf(g^k_j,h^k_j\r)\r\},
\end{eqnarray}
where the implicit constants are independent of $f$.
\end{thm}

As applications, besides the classical Hilbert transform, Riesz transform and the Calder\'on--Zygmund operators with homogeneous kernels $\Omega({x\over |x|}){1\over |x|^n}$ on $\R^n$ (studied in \cite{HLW,LOR}),  we use our main theorems  to obtain the two weight commutator result of the following operators:

\begin{enumerate}

\item the Cauchy integral operator $C_A$ along a Lipschitz curve  $z:= x+iA(x)$, $x\in(-\infty,\infty)$ and $A'\in L^\infty(\mathbb R)$;


\item The Cauchy--Szeg\"o projection operator on Heisenberg group $\mathbb H^n$;

\item The Szeg\"o projection operator on a family of weakly pseudoconvex domains;


\item Riesz transforms associated with Sub-Laplacian on stratified Lie groups $\mathcal G$;

\item Riesz transforms associated with Bessel operator $\Delta_\lambda$ on $\R_+$ for $\lambda>-1/2$;

\item Riesz transforms associated with higher order Bessel operator $\Delta_{n,\lambda}$ on $\R^{n+1}_+$ for $\lambda>-1/2$.

\end{enumerate}
The definitions of the above operators  will be given in Section \ref{s:Application}. We have the following result.


\begin{thm}\label{thm3}
Let $T$ be one of the operators listed above and let $(X,d,\mu)$ be the underlying space adapted to $T$.  Suppose $1<p<\infty$, $\lambda_1,\lambda_2\in A_p$, $\nu:= \lambda_1^{1\over p}\lambda_2^{-{1\over p}}$ and $m\in\mathbb N$. Suppose that $b\in L^1_{\rm loc}(X)$. Then we have
\begin{equation}\label{equiv}
 \| b \| _{{\rm BMO}_{\nu^{1\over m}}(X)}^m\approx \| T_b^m : L^p_{\lambda_1}(X) \rightarrow L^p_{\lambda_2}(X) \|.
\end{equation}
Moreover, based on the result above for $m=1$ and on the duality,  the corresponding weighted Hardy space $H^1_\nu(X)$ has a weak factorisation as in \eqref{represent of H1}.
\end{thm}
To prove this theorem, the key step is to verify that all these operators listed above satisfy the conditions as in Definition \ref{def 1}
and the homogeneous condition as in \eqref{e-assump cz ker low bdd weak}. We point out that such verification for Cauchy integral operator $C_A$ is direct. The verifications of Cauchy--Szeg\"o projection operator on Heisenberg group, the Szeg\"o projection operator on a family of weakly pseudoconvex domains and the Riesz transforms associated with Sub-Laplacian on stratified Lie groups
can be derived based on the results in \cite[Chapter XII]{s93}, \cite{GrSt} and \cite{DLLW}, respectively. The verification for Riesz transforms associated with Bessel operator $\Delta_\lambda$ on $\R_+$ for $\lambda>0$ can be derived from the result in \cite{MS}, while for $\lambda\in (-1/2,0)$ is new here.  The verification for Riesz transforms associated with higher order Bessel operator
is totally new, especially the pointwise kernel lower bound of this Riesz transform.

\medskip
We now address our result Theorem \ref{thm3} with respect to  the 6 examples above, respectively:
\begin{enumerate}

\item The unweighted result was obtained in \cite{LTWW} when $m=1$, and the two weight result is new here for $m\geq1$;


\item This result is new, even the unweighted version is unknown;

\item This result is new, even the unweighted version is unknown;


\item This result was obtained in \cite{DLLWW} when $m=1$ and is new here when $m>1$.

\item The unweighted result was obtained in \cite{DLWY} when $\lambda>0$ and $m=1$, the two weight result is new here for $m\geq1$ and for all $\lambda>-1/2$;

\item This result is new, even the unweighted version is unknown;

\end{enumerate}

This paper is organised as follows. In Section 2 we recall the necessary preliminaries on spaces of homogeneous type. 
In Section \ref{s4}, we first characterise the sparse system equivalently via the $\Lambda$-Carleson packing condition and the $\eta$-sparse condition,  and then borrowing the idea from \cite{LOR}, we study the sparse operators and its domination of commutator on spaces of homogeneous type, and by using this as a main tool, in Section \ref{s5} we obtain the upper bound of two weight commutator, i.e., Theorem
\ref{thm1}. In Section \ref{s6} we provide the lower bound of two weight commutator, i.e., Theorem \ref{thm2}, by combining the ideas in \cite{LOR2} and \cite{Hyt}. In Section \ref{s3}
we provide a study of weighted Hardy space and its duality on spaces of homogeneous type, and provide the proof of Theorem \ref{thm5}. In Section 7 we provide the applications where we address the new points in this paper. In the last section we also provide a new proof of the lower bound of two weight commutator in the product setting for little bmo space on spaces of homogeneous type. Note that in $\R^n\times \R^m$, this was first studied by \cite{HPW} by using the Fourier transform for the Riesz transform kernel.

\medskip
Throughout the paper,
we denote by $C$ and $\widetilde{C}$ {\it positive constants} which
are independent of the main parameters, but they may vary from line to
line. For every $p\in(1, \fz)$, we denote by $p'$ the conjugate of $p$, i.e., $\frac{1}{p'}+\frac{1}{p}=1$.  If $f\le Cg$ or $f\ge Cg$, we then write $f\ls g$ or $f\gs g$;
and if $f \ls g\ls f$, we  write $f\approx g.$

\section{Preliminaries on Spaces of Homogeneous Type}
\label{s2}
\noindent


%
%

Let $(X,d,\mu)$ be a space of homogeneous type as mentioned in Section 1.

\subsection{A System of Dyadic Cubes}\label{sec:dyadic_cubes}
In $(X,d,\mu)$, a
countable family
$
    \mathscr{D}
    := \cup_{k\in\Z}\mathscr{D}_k, \
    \mathscr{D}_k
    :=\{Q^k_\alpha\colon \alpha\in \mathscr{A}_k\},
$
of Borel sets $Q^k_\alpha\subseteq X$ is called \textit{a
system of dyadic cubes with parameters} $\delta\in (0,1)$ and
$0<a_1\leq A_1<\infty$ if it has the following properties:
\begin{equation*} 
    X
    = \bigcup_{\alpha\in \mathscr{A}_k} Q^k_{\alpha}
    \quad\text{(disjoint union) for all}~k\in\Z;
\end{equation*}
\begin{equation*}
    \text{if }\ell\geq k\text{, then either }
        Q^{\ell}_{\beta}\subseteq Q^k_{\alpha}\text{ or }
        Q^k_{\alpha}\cap Q^{\ell}_{\beta}=\emptyset;
\end{equation*}
\begin{equation*}
    \text{for each }(k,\alpha)\text{ and each } \ell\leq k,
    \text{ there exists a unique } \beta
    \text{ such that }Q^{k}_{\alpha}\subseteq Q^\ell_{\beta};
\end{equation*}
\begin{equation*}
\begin{split}
    & \text{for each $(k,\alpha)$ there exists at most $M$
        (a fixed geometric constant)  $\beta$ such that }  \\
    & Q^{k+1}_{\beta}\subseteq Q^k_{\alpha}, \text{ and }
        Q^k_\alpha =\bigcup_{{Q\in\mathscr{D}_{k+1},
    Q\subseteq Q^k_{\alpha}}}Q;
\end{split}
\end{equation*}
\begin{equation}\label{eq:contain}
    B(x^k_{\alpha},a_1\delta^k)
    \subseteq Q^k_{\alpha}\subseteq B(x^k_{\alpha},A_1\delta^k)
    =: B(Q^k_{\alpha});
\end{equation}
\begin{equation*}
   \text{if }\ell\geq k\text{ and }
   Q^{\ell}_{\beta}\subseteq Q^k_{\alpha}\text{, then }
   B(Q^{\ell}_{\beta})\subseteq B(Q^k_{\alpha}).
\end{equation*}
The set $Q^k_\alpha$ is called a \textit{dyadic cube of
generation} $k$ with centre point $x^k_\alpha\in Q^k_\alpha$
and sidelength~$\delta^k$. 


From the properties of the dyadic system above and from the doubling measure, we can deduce that there exists a constant
$C_{\mu,0}$ depending only on $C_\mu$ as in \eqref{doubling condition} and $a_1, A_1$ as above, such that for any $Q^k_\alpha$ and $Q^{k+1}_\beta$  with $Q^{k+1}_\beta\subset Q^k_\alpha$,
\begin{align}\label{Cmu0}
\mu(Q^{k+1}_\beta)\leq \mu(Q^k_\alpha)\leq C_{\mu,0}\mu(Q^{k+1}_\beta).
\end{align}

We recall from \cite{HK} the following construction, which is a
slight elaboration of seminal work by M.~Christ \cite{Chr}, as
well as Sawyer--Wheeden~\cite{SawW}.

\begin{thm}\label{theorem dyadic cubes}
On $(X,d,\mu)$, there exists a system of dyadic cubes with parameters
$0<\delta\leq (12A_0^3)^{-1}$ and $a_1:=(3A_0^2)^{-1},
A_1:=2A_0$. The construction only depends on some fixed set of
countably many centre points $x^k_\alpha$, having the
properties that
   $ d(x_{\alpha}^k,x_{\beta}^k)
        \geq \delta^k$ with $\alpha\neq\beta$,
    $\min_{\alpha}d(x,x^k_{\alpha})
        < \delta^k$ for all $x\in X,$
   and a certain partial order ``$\leq$'' among their index pairs
$(k,\alpha)$. In fact, this system can be constructed in such a
way that
$$
    \overline{Q}^k_\alpha
    =\overline{\{x^{\ell}_\beta:(\ell,\beta)\leq(k,\alpha)\}}, \quad\quad
    \widetilde{Q}^k_\alpha:=\operatorname{int}\overline{Q}^k_\alpha=
    \Big(\bigcup_{\gamma\neq\alpha}\overline{Q}^k_\gamma\Big)^c,
\quad \quad   \widetilde{Q}^k_\alpha\subseteq Q^k_\alpha\subseteq 
    \overline{Q}^k_\alpha,
$$
where $Q^k_\alpha$ are obtained from the closed sets
$\overline{Q}^k_\alpha$ and the open sets $\widetilde{Q}^k_\alpha$ by
finitely many set operations.
\end{thm}

We also recall the following remark from \cite[Section 2.3]{KLPW}.
The construction of dyadic cubes requires their centre points
and an associated partial order be fixed \textit{a priori}.
However, if either the centre points or the partial order is
not given, their existence already follows from the
assumptions; any given system of points and partial order can
be used as a starting point. Moreover, if we are allowed to
choose the centre points for the cubes, the collection can be
chosen to satisfy the additional property that a fixed point
becomes a centre point at \textit{all levels}:
\begin{equation}\label{eq:fixedpoint}
\begin{split}
    &\text{given a fixed point } x_0\in X, \text{ for every } k\in \Z,
        \text{ there exists }\alpha \text{ such that } \\
    & x_0
        = x^k_\alpha,\text{ the centre point of }
        Q^k_\alpha\in\mathscr{D}_k.
\end{split}
\end{equation}

\subsection{Adjacent Systems of Dyadic Cubes}

On $(X,d,\mu)$, a
finite collection $\{\mathscr{D}^t\colon t=1,2,\ldots ,\mathpzc T\}$ of the dyadic 
families  is called a \textit{collection of
adjacent systems of dyadic cubes with parameters} $\delta\in
(0,1), 0<a_1\leq A_1<\infty$ and $1\leq C_{adj}<\infty$ if it has the
following properties: individually, each $\mathscr{D}^t$ is a
system of dyadic cubes with parameters $\delta\in (0,1)$ and $0
< a_1 \leq A_1 < \infty$; collectively, for each ball
$B(x,r)\subseteq X$ with $\delta^{k+3}<r\leq\delta^{k+2},
k\in\Z$, there exist $t \in \{1, 2, \ldots, \mathpzc T\}$ and
$Q\in\mathscr{D}^t$ of generation $k$ and with centre point
$^tx^k_\alpha$ such that $d(x,{}^tx_\alpha^k) <
2A_0\delta^{k}$ and
\begin{equation}\label{eq:ball;included}
    B(x,r)\subseteq Q\subseteq B(x,C_{adj}r).
\end{equation}

We recall from \cite{HK} the following construction.

\begin{thm}\label{thm:existence2}
    Let $(X,d,\mu)$ be a  space of homogeneous type.
    Then there exists a collection $\{\mathscr{D}^t\colon
    t = 1,2,\ldots ,\mathpzc T\}$ of adjacent systems of dyadic cubes with
    parameters $\delta\in (0, (96A_0^6)^{-1}), a_1 := (12A_0^4)^{-1},
    A_1 := 4A_0^2$ and $C := 8A_0^3\delta^{-3}$. The centre points
    $^tx^k_\alpha$ of the cubes $Q\in\mathscr{D}^t_k$ have, for each
    $t\in\{1,2,\ldots,\mathpzc T\}$, the two properties
    \begin{equation*}
        d(^tx_{\alpha}^k, {}^tx_{\beta}^k)
        \geq (4A_0^2)^{-1}\delta^k\quad(\alpha\neq\beta),\qquad
        \min_{\alpha}d(x,{}^tx^k_{\alpha})
        < 2A_0\delta^k\quad \text{for all}~x\in X.
    \end{equation*}
    Moreover, these adjacent systems can be constructed in such a
    way that each $\mathscr{D}^t$ satisfies the distinguished
    centre point property \eqref{eq:fixedpoint}.
\end{thm}

We recall from \cite[Remark 2.8]{KLPW}
that the number $\mathpzc T$ of the adjacent systems of dyadic
    cubes as in the theorem above satisfies the estimate
    \begin{equation*}
        \mathpzc T
        = \mathpzc T(A_0,\widetilde A_1,\delta)
        \leq \widetilde A_1^6(A_0^4/\delta)^{\log_2\widetilde A_1},
    \end{equation*}
where $\widetilde A_1$ is the geometrically doubling constant, see \cite[Section 2]{KLPW}.

%

\subsection{An Explicit Haar Basis on Spaces of Homogeneous Type}

Next we recall the explicit construction in \cite{KLPW} of a Haar basis
$\{h_{Q}^{\epsilon}: Q\in \mathscr{D}, \epsilon = 1,\dots,M_Q - 1\}$ for $L^p(X,\mu)$,
$1 < p < \infty$, associated to the dyadic cubes
$Q\in\mathscr{D}$ as follows. Here $M_Q := \#\ch(Q) = \# \{R\in
\mathscr{D}_{k+1}\colon R\subseteq Q\}$ denotes the number of
dyadic sub-cubes (``children'') the cube $Q\in \mathscr{D}_k$
has; namely $\mathcal{H}(Q)$ is the collection of dyadic children of $Q$.

\begin{thm}[\cite{KLPW}]\label{thm:convergence}
    Let $(X,d,\mu)$ be a
    space of homogeneous type and suppose $\mu$ is a positive Borel measure on $X$
    with the property that $\mu(B) < \infty$ for all balls
    $B\subseteq X$. For $1 < p < \infty$, for each $f\in
    L^p(X,\mu)$, we have
    \[
        f(x)
        =  \sum_{Q\in\mathscr{D}}\sum_{\epsilon=1}^{M_Q-1}
            \langle f,h^\epsilon_Q\rangle h^\epsilon_Q(x), 
    \]
    where the sum converges (unconditionally) both in the
    $L^p(X,\mu)$-norm and pointwise $\mu$-almost everywhere.
 \end{thm}

The following theorem collects several basic properties of the
functions $h_{Q}^{\epsilon}$.

\begin{thm}[\cite{KLPW}]\label{prop:HaarFuncProp}
    The Haar functions $h_{Q}^{\epsilon}$, $Q\in\mathscr{D}$,
    $\epsilon = 1,\ldots,M_Q - 1$, have the following properties:
    \begin{itemize}
        \item[(i)] $h_{Q}^{\epsilon}$ is a simple Borel-measurable
            real function on $X$;
        \item[(ii)] $h_{Q}^{\epsilon}$ is supported on $Q$;
        \item[(iii)] $h_{Q}^{\epsilon}$ is constant on each
            $R\in\ch(Q)$;
        \item[(iv)] $\int h_{Q}^{\epsilon}\, d\mu = 0$ (cancellation);
        \item[(v)] $\langle h_{Q}^{\epsilon},h_Q^{\epsilon'}\rangle = 0$ for
            $\epsilon \neq \epsilon'$, $\epsilon$, $\epsilon'\in\{1, \ldots, M_Q - 1\}$;
        \item[(vi)] the collection
            $
                \big\{\mu(Q)^{-1/2}1_Q\big\}
                \cup \{h_{Q}^{\epsilon} : \epsilon = 1, \ldots, M_Q - 1\}
            $
            is an orthogonal basis for the vector
            space~$V(Q)$ of all functions on $Q$ that are
            constant on each sub-cube $R\in\ch(Q)$;
        \item[(vii)] 
        if $h_{Q}^{\epsilon}\not\equiv 0$ then
            $
                \Norm{h_{Q}^{\epsilon}}{L^p(X,\mu)}
                \approx \mu(Q_{\epsilon})^{\frac{1}{p} - \frac{1}{2}}
                \quad \text{for}~1 \leq p \leq \infty;
            $
        \item[(viii)] 
                \hspace{4cm}
                $\Norm{h_{Q}^{\epsilon}}{L^1(X,\mu)}\cdot
                \Norm{h_{Q}^{\epsilon}}{L^\infty(X,\mu)} \approx 1$.
    \end{itemize}
\end{thm}
As stated in \cite{KLPW}, we also have $h_Q^0:= \mu(Q)^{-1/2}1_Q$ which is a non-cancellative Haar function.
Moreover, the martingale associated with the Haar functions are as follows: for $Q \in \mathscr{D}_k$,
$$ \mathbb{E}_Qf = \langle f,h_Q^0\rangle h_Q^0
\quad {\rm and}\quad
\mathbb{D}_Qf =\sum_{\epsilon=1}^{M_Q-1} \mathbb{D}_{Q}^{\epsilon}f, $$
where $\mathbb{D}_{Q}^{\epsilon}=\langle f,h_{Q}^{\epsilon}\rangle h_{Q}^{\epsilon}$ is the martingale operator associated with the $\epsilon$-th subcube of $Q$. Also we have
$$
   \mathbb{E}_kf =\sum_{Q\in \mathscr{D}_k}\mathbb{E}_Qf \quad{\rm and} \quad \mathbb{D}_kf = \mathbb{E}_{k+1}f- \mathbb{E}_kf.
$$
Hence, based on the construction of Haar system $\{h_Q^{\epsilon}\}$ in \cite{KLPW} we obtain that for each $R\in\mathscr D$,
\begin{align*}
\sum_{Q:\ R\subset Q} \sum_{\epsilon=1}^{M_{Q}-1} \langle f, h_Q^{\epsilon}\rangle h_Q^{\epsilon} h_R^\eta= \sum_{Q:\ R\subset Q} \mathbb{D}_Qf\cdot h_R^{\eta}
= \mathbb{E}_Rf\cdot h_R^{\eta} = \langle f,h_R^0\rangle h_R^0h_R^{\eta}.
\end{align*}

%
%

\subsection{Muckenhoupt $A_p$ Weights}

\begin{defn}
  \label{def:Ap}
  Let $w(x)$ be a nonnegative locally integrable function
  on~$X$. For $1 < p < \infty$, we
  say $w$ is an $A_p$ \emph{weight}, written $w\in
  A_p$, if
  \[
    [w]_{A_p}
    := \sup_B \left(\intav_B w\right)
    \left(\intav_B
      \left(\dfrac{1}{w}\right)^{1/(p-1)}\right)^{p-1}
    < \infty.
  \]
  Here the supremum is taken over all balls~$B\subset X$. The quantity $[w]_{A_p}$ is called the \emph{$A_p$~constant
  of~$w$}.
  For $p = 1$, we say $w$ is an $A_1$ \emph{weight},
  written $w\in A_1$, if $M(w)(x)\leq w(x)$ for a.e. $x\in X$, and let $A_\infty := \cup_{1\leq p<\infty} A_p$ and we have
    \(
    [w]_{A_\infty}
    := \sup_B \left(\intav_B w\right)
    \exp\left(\intav_B \log \left(\frac{1}{w}\right) \right)
    < \infty.
  \)

\end{defn}

Next we note that for $w\in A_p$ the measure $w(x)d\mu(x)$ is a doubling measure on $X$. To be more precise, we have
that for all $\lambda>1$ and all balls $B\subset X$,
\begin{align*}
w(\lambda B)\leq \lambda^{np}[w]_{A_p}w(B),
\end{align*}
where $n$ is the upper dimension of the measure $\mu$, as in \eqref{upper dimension}.

We also point out that for $w\in A_\infty$, there exists $\gamma>0$ such that for every ball $B$,
$$ \mu\Big( \Big\{x\in B: \ w(x)\geq\gamma \intav_B w\Big\} \Big)\geq {1\over 2}\mu(B). $$
And this implies that for every ball $B$ and for all $\delta\in(0,1)$,
\begin{align}\label{e-reverse holder}
\intav_B w\le C \left(\intav_B w^\delta\right)^{1/\delta};
\end{align}
see  also \cite{LOR2}.

\section{Sparse Operators and Domination of Commutators on Spaces of Homogeneous Type}\label{s4}

Let $\mathcal D$ be a system of dyadic cubes on $X$ as in Section 2.1.  As in the Euclidean setting, we have two competing versions of sparsity for a collection of sets, one geometric and the other a Carleson measure condition.

\begin{defn} \label{D:Sparse}
Given $0<\eta<1$, a collection $\mathcal S \subset \mathcal D$ of dyadic cubes is said to be $\eta$-sparse provided that for every $Q\in\mathcal S$, there is a measurable subset $E_Q \subset Q$ such that
$\mu(E_Q) \geq \eta \mu(Q)$ and the sets $\{E_Q\}_{Q\in\mathcal S}$ have only finite overlap.
\end{defn}

\begin{defn} \label{D:Carleson}
Given 
$\Lambda>1$, a collection $\mathcal S \subset \mathcal D$ of dyadic cubes is said to be $\Lambda$-Carleson if for
every cube $Q\in\mathcal D$,
$$ \sum_{P\in\mathcal S, P\subseteq Q}\mu(P)\leq \Lambda \mu(Q). $$
\end{defn}

We first show that the above two definitions are equivalent in a space of homogeneous type.  The proof closely follows the original idea in \cite{LN} with some minor modifications.
\begin{thm}\label{thm sparse}
Given $0<\eta<1$ and a collection $\mathcal S \subset \mathcal D$ of dyadic cubes, the following statements hold:
\begin{itemize}
\item If $\mathcal{S}$ is $\eta$-sparse, then $\mathcal{S}$ is ${c\over\eta}$-Carleson.
\item If $\mathcal{S}$ is ${1\over\eta}$-Carleson, then $\mathcal{S}$ is $\eta$-sparse.
\end{itemize}
\end{thm}
The reason for the extra constant $c$ in the above, is that for later parts of our argument, to control the commutator, we need to allow the sets $E_Q$ to have finite overlap.  If the sets $E_Q$ were exactly disjoint then one could take $c=1$ in the above and the statement would be cleaner and more in line with that in \cite{LN}.
\begin{proof}
Note that if a collection $\mathcal S \subset \mathcal D$  of dyadic cubes is $\eta$-sparse, that is for every $Q\in\mathcal S$, there is a measurable subset $E_Q \subset Q$ such that $\mu(E_Q) \geq \eta \mu(Q)$ and the sets $\{E_Q\}_{Q\in\mathcal S}$ have only finite overlap, we will have that $\mathcal S$ is $c\eta^{-1}$-Carleson according to Definition \ref{D:Carleson} (following from the standard computation and the constant $c$ denoting the amount of overlap of the sets $E_Q$).

Thus, it suffices to show that for $\Lambda>1$, if a collection $\mathcal S \subset \mathcal D$  of dyadic cubes is $\Lambda$-Carleson, then it is $\Lambda^{-1}$-sparse.  Here we will follow the proof in \cite{LN} but with some minor modifications.  To see this, we first point out that if the collection $\mathcal S \subset \mathcal D$  of dyadic cubes $\{Q\}$ has a bottom layer $\mathscr D_K$ for some fixed integer $K$, then it is direct to construct the set $E_Q$. We  begin with considering all dyadic cubes $\{Q\}\subset \mathcal S \cap \mathscr D_K$
and choose any measurable set $E_Q\subset Q$ of measure $\Lambda^{-1}\mu(Q)$ for them. We now just repeat this choice for each dyadic cube in upper layers one by one. To be more specific, for each $Q\in\mathcal S\cap \mathscr D_k$ with $k\leq K$, choose a set 
$$E_Q\subset Q\Big\backslash \bigcup_{R\in\mathcal S, R\subsetneq Q} E_R $$
such that $\mu(E_Q) = \Lambda^{-1}\mu(Q)$. We now show that such choice of $E_Q$ is possible. In fact, note that for every $Q\in \mathcal S$, we have
\begin{align*}
\mu\Big( \bigcup_{R\in\mathcal S, R\subsetneq Q} E_R \Big)\leq \Lambda^{-1} \sum_{R\in\mathcal S, R\subsetneq Q} \mu(R) \leq  \Lambda^{-1}  (\Lambda-1)\mu(Q)= (1-\Lambda^{-1})\mu(Q),
\end{align*}
where the last inequality follows from the $\Lambda$-Carleson condition and from the fact that $R\subsetneq Q$. This shows that 
$$ \mu\Big( Q\Big\backslash \bigcup_{R\in\mathcal S, R\subsetneq Q} E_R \Big)\geq \mu(Q) - (1-\Lambda^{-1})\mu(Q)= \Lambda^{-1}\mu(Q)$$
and so the choice of the top set $E_Q$ is always possible.

Next, we consider the case that there is no fixed bottom layer. We run the above construction with a particular choice for each $K=0,1,2,\ldots$ and then pass to the limit.  To begin with, fix $K\geq0$. For each $Q\in  \mathcal S \cap ( \cup_{k\leq K}\mathscr D_k )$, we define the sets
$\widehat{E}_Q^{(K)}$ inductively as follows. 

First, for each $Q\in  \mathcal S \cap \mathscr D_k $ with $k\leq K$, 
we  consider the auxiliary set 
$$  \mathcal Q(t,Q):=  B(x_Q, t\delta^k)\cap Q, \quad t\in (0,A_1), $$
where $x_Q$ is the centre point of $Q$ and $A_1$, $\delta$ are the constants as introduced in Section 2.1. From property \eqref{eq:contain}, it is clear that when $0<t<a_1$, then $B(x_Q, t\delta^k)\subset Q$  and when 
$t>A_1$, then $ Q\subset B(x_Q, t\delta^k)$; moreover, we have $\mu(B(x_Q, t\delta^k))\to0$ as $t\to 0^+$.

Now for $Q\in  \mathcal S \cap \mathscr D_K$,
from the above observations together with the continuity and monotonicity of the function $t\mapsto \mathcal Q(t,Q) =\mu(B(x_Q, t\delta^K)) \cap Q$, we conclude that  there must be some $t_{\Lambda,K,K}\in (0,A_1)$ such that
$\mu(B(x_Q, t_{\Lambda,K,K}\delta^K)\cap Q )= \Lambda^{-1}\mu(Q)$. Here and in what follows, we use the triple $(\Lambda,k,K)$ for the subscript of $t$, where $\Lambda$ denotes that the value of such $t$ depends on $\Lambda$,  $k$ denotes that $Q$ is in the layer $k$ and the last $K$ denotes that we start at the layer $K$.
We set 
$$\widehat{E}_Q^{(K)}:=\mathcal Q(t_{\Lambda,K,K},Q)=B(x_Q, t_{\Lambda,K,K}\delta^K)\cap Q.$$

Suppose now $\widehat{E}_R^{(K)}$ are already defined for every $R\in \mathcal S \cap ( \cup_{k+1\leq i\leq K}\mathscr D_i )$. We now define $\widehat{E}_Q^{(K)}$ for $Q\in \mathcal S \cap \mathscr D_k $ in the following manner.  We set
$$   \widehat{E}_Q^{(K)} : =   \mathcal Q(t_{\Lambda,k,K},Q) \bigcup F_Q^{(K)},$$
where
$$  F_Q^{(K)} := \bigcup_{  R\in \mathcal S \cap ( \cup_{k+1\leq i\leq K}\mathscr D_i ), R\subsetneq Q }  \widehat{E}_R^{(K)}  $$
and $t_{\Lambda,k,K}\in (0,A_1)$ is chosen such that the set 
$$ E_Q^{(K)}:=  \mathcal Q(t_{\Lambda,k,K},Q) \backslash F_Q^{(K)} $$
satisfies 
$\mu(  E_Q^{(K)} ) = \Lambda^{-1}\mu(Q) $.

Now we claim that $ \widehat{E}_Q^{(K)} \subset \widehat{E}_Q^{(K+1)}  $ for every  $Q\in \mathcal S \cap ( \cup_{k\leq K}\mathscr D_k )$. To see this, we note that for each $Q\in \mathcal S \cap \mathscr D_K$, $ \widehat{E}_Q^{(K)}$ is just the set $\mathcal Q(t_{\Lambda,K,K},Q)$.
On the other hand, $ \widehat{E}_Q^{(K+1)}$ contains the set $\mathcal Q(t_{\Lambda,K,K+1},Q)$ which has the same centre point as $\mathcal Q(t_{\Lambda,K,K},Q)$, but with $t_{\Lambda,K,K+1} \geq t_{\Lambda,K,K}$ since 
$$\mu\Big( \mathcal Q(t_{\Lambda,K,K+1},Q)   \backslash  \bigcup_{  R\in \mathcal S\cap \mathscr D_{K+1}, R\subsetneq Q }  \widehat{E}_R^{(K+1)}\Big)=\Lambda^{-1}\mu(Q) = \mu(   \mathcal Q(t_{\Lambda,K,K},Q)  ).$$
Hence, we see that for each $Q\in \mathcal S \cap \mathscr D_K$, we have $ \widehat{E}_Q^{(K)} \subseteq \widehat{E}_Q^{(K+1)}$.  Then, we proceed via backward induction. Assume that $ \widehat{E}_Q^{(K)} \subseteq \widehat{E}_Q^{(K+1)}$ for every $Q\in \mathcal S \cap ( \cup_{k<i\leq K}\mathscr D_i )$. Take any $Q\in \mathcal S\cap \mathscr D_k$. Then the inductive hypothesis implies that $F_Q^{(K)}\subseteq F_Q^{(K+1)}  $. Let $Q(t_{\Lambda,k,K},Q)$ be the set added to $F_Q^{(K)}$ when constructing $\widehat{E}_Q^{(K)} $. Then we have
$$ \mu\Big(  Q(t_{\Lambda,k,K},Q) \backslash F_Q^{(K+1)}   \Big) \leq  \mu\Big(  Q(t_{\Lambda,k,K},Q) \backslash F_Q^{(K)}   \Big)=\Lambda^{-1}\mu(Q), $$
which implies that
$  t_{\Lambda,k,K+1} \geq t_{\Lambda,k,K}  $. Thus, we have $Q(t_{\Lambda,k,K},Q)\subset Q(t_{\Lambda,k,K+1},Q)$, which yields $ \widehat{E}_Q^{(K)} \subseteq \widehat{E}_Q^{(K+1)}$, and hence the claim follows.

Now for $Q\in\mathcal S\cap \mathscr D_k$, we define
$$   \widehat E_Q:=\lim_{K\to\infty} \widehat{E}_Q^{(K)},   $$
which, by using the claim above, equals
$$ \bigcup_{K=k}^\infty  \widehat{E}_Q^{(K)} \subset Q. $$
Moreover, for each $K$ we have
$$ \mu( {E}_Q^{(K)}) =\mu\big(  \widehat{E}_Q^{(K)}\backslash  {F}_Q^{(K)}  \big)=\Lambda^{-1}\mu(Q).  $$
Note that the sets $ {F}_Q^{(K)}$ also form an increasing sequence (with respect to $K$), so for each $Q\in\mathcal S$, the limit set
$$E_Q:= \lim_{K\to\infty} E_Q^{(K)} = \widehat E_Q \backslash \big( \lim_{K\to\infty} F_Q^{(K)} \big) = \widehat E_Q \big\backslash \Big( \bigcup_{R\in\mathcal S, R\subsetneq Q} \widehat E_R \Big) $$
exists, and is contained in $Q$ and has the required measure. 
Moreover, all $E_Q$ are disjoint.  The proof of Theorem \ref{thm sparse} is complete.
\end{proof}

We now recall the well-known definition for sparse operators.
\begin{defn} \label{D:Sparse Operator}
Given $0<\eta<1$ and an $\eta$-sparse family $\mathcal S \subset \mathcal D$ of dyadic cubes. The
sparse operator $\mathcal A_{\mathcal S}$ is defined by
$$ \mathcal A_{\mathcal S}f(x):= \sum_{Q\in \mathcal S} f_Q \chi_Q(x).$$
\end{defn}
Following the proof of \cite[Theorem 3.1]{Moen}, we obtain that
\begin{align*}
\|\mathcal A_{\mathcal S}f\|_{L^p_w(X)}\leq C_{\eta,n,p}[w]_{A_p}^{\max\{1,{1\over p-1}\}} \|f\|_{L^p_w(X)}, \quad 1<p<\infty.
\end{align*}
Denote by $\Omega(b,B)$ the standard mean oscillation
\begin{align}\label{Omega oscillation}
\Omega(b,B):={1\over\mu(B)} \int_B|b(x)-b_B|d\mu(x).
\end{align}
\begin{lem}\label{lem b-bQ}
Given $0<\gamma<1$. Let $\mathcal D$ be a dyadic system in $X$ and let $\mathcal S\subset \mathcal D$ be a $\gamma$-sparse
family. Assume that $b\in L^1_{loc}(X)$. Then there exists a ${\gamma\over 2(\gamma+1)}$-sparse family
$\tilde{\mathcal S}\subset \mathcal D$ such that $\mathcal S\subset \tilde{\mathcal S}$ and for every cube $Q\in \tilde{\mathcal S}$,
\begin{align}\label{b-bQ eee0}
|b(x)-b_Q|\leq C\sum_{R\in \tilde{\mathcal S}, R\subset Q} \Omega(b,R)\chi_R(x)
\end{align}
for a.e. $x\in Q$.
\end{lem}
\begin{proof}
Fix a dyadic cube $Q\in\mathcal D$. We now show that there exists a family of
pairwise disjoint cubes $\{P_j\}\subset\mathcal D(Q)$ such that
$\sum_j\mu(P_j)\leq {1\over 2}\mu(Q)$ and for a.e. $x\in Q$,
\begin{align}\label{b-bQ eee1}
|b(x)-b_Q|\leq C\cdot C_{\mu,0}\Omega(b,Q)+ \sum_{j} |b(x)-b_{P_j}| \chi_{P_j}(x).
\end{align}
%

Let $M_Q^d$ be the standard dyadic local maximal operator restricted to $\mathcal D(Q)$ and $C_{M_Q^d}$ be the weak type $(1,1)$-norm of $M_Q^d$.
Then one can choose a constant $C$ depending on $C_{M_Q^d}$ such that the set
$E:=\{x\in Q: M_Q^d(b-b_Q)(x)> 4 C_{\mu,0}\cdot C\cdot \Omega(b,Q) \}$
satisfies  that $\mu(E)\leq {1\over 4 C_{\mu,0}}\mu(Q)$, where $C_{\mu,0}$ is the constant as in \eqref{Cmu0}.

If  $\mu(E)=0$, then \eqref{b-bQ eee1} holds trivially with the empty family
$\{P_j\}_j$. If $\mu(E)>0$ , then we now apply the Calder\'on--Zygmund decomposition
to the function $h(x):=\chi_E(x)$ on $Q$ at height $\lambda:={1\over 2 C_{\mu,0}}$ as follows: \ 
we begin by considering the descendants of $Q$ in $\mathcal D(Q)$ since
$$ \int_Q |h(x)|d\mu(x) < \lambda \mu(Q). $$
Let $\{Q_j^{(1)}\}\subset \mathcal D(Q)$ be the children of $Q$. If
\begin{align}\label{criteria1}
 \int_{Q_j^{(1)}} |h(x)|d\mu(x) >\lambda \mu(Q_j^{(1)})
\end{align}
then we select it as our candidate cube. If
$$ \int_{Q_j^{(1)}} |h(x)|d\mu(x) \leq \lambda \mu(Q_j^{(1)}) $$
then we keep looking at the children of $Q_j^{(1)}$ in $\mathcal D(Q)$ and then repeat
the above selection criteria and we will stop only when we find some descendant of $Q_j^{(1)}$ in $\mathcal D(Q)$
such that it meets the criteria \eqref{criteria1}.

Then it is direct to see that this produces pairwise disjoint
cubes $\{P_j\}\subset\mathcal D(Q)$ such that
$$ {1\over 2 C_{\mu,0}} \mu(P_j)< \mu(P_j\cap E) \leq {1\over 2}\mu(P_j) $$
and
$ \mu(E\backslash \cup_jP_j)=0 $. It follows that $\sum_j\mu(P_j)\leq {1\over 2}\mu(Q)$
and $P_j\cap E^c \not=\emptyset$.

Therefore, we get
\begin{align}\label{b-bQ eee2}
|b_{P_j}-b_Q|\leq {1\over \mu(P_j)}\int_{P_j} |b(x)-b_{Q}|d\mu(x) \leq 4 C_{\mu,0}\cdot C\cdot \Omega(b,Q)
\end{align}
and for a.e. $x\in Q$, $  |b(x)-b_Q|\chi_{Q\backslash \cup_jP_j}\leq  4 C_{\mu,0}\cdot C\Omega(b,Q)$.

Then, we have
\begin{align*}
|b(x)-b_Q|\chi_Q(x)&\leq  |b(x)-b_Q|\chi_{Q\backslash \cup_jP_j}(x) + \sum_{j} |b_{P_j}-b_{Q}| \chi_{P_j}(x)
+  \sum_{j} |b(x)-b_{P_j}| \chi_{P_j}(x)\\
&\leq 4 C_{\mu,0}\cdot C\Omega(b,Q) +  \sum_{j} |b(x)-b_{P_j}| \chi_{P_j}(x),
\end{align*}
which gives \eqref{b-bQ eee1}.

We observe that if $P_j\subset R$, where $R\in\mathcal D(Q)$, then $R\cap E^c\not=\emptyset$. Hence $P_j$
in \eqref{b-bQ eee2} can be replaced by $R$, namely, we have
$ |b_R-b_Q|\leq 4 C_{\mu,0}\cdot C\Omega(b,Q)$. Therefore, if $\cup_j P_j\subset \cup_i R_i$, where
$R_i\in\mathcal D(Q)$, and the cubes $\{R_i\}$ are pairwise disjoint, then we have
\begin{align}\label{b-bQ eee3}
|b(x)-b_Q|\leq 4 C_{\mu,0}\cdot C\Omega(b,Q)+ \sum_{i} |b(x)-b_{R_i}| \chi_{R_i}(x).
\end{align}

Iterating \eqref{b-bQ eee1}, from the selection of $\{P_j\}$ and from Definition \ref{D:Carleson}, we obtain that there exists a ${1\over 2}$-sparse family $\mathcal F(Q)\subset \mathcal D(Q)$
such that for a.e. $x\in Q$,
\begin{align*}
|b(x)-b_Q|\chi_Q(x)\leq 4 C_{\mu,0}\cdot C \sum_{P\in\mathcal F(Q)} \Omega(b,P)  \chi_{P}(x).
\end{align*}

Now for each $\mathcal F(Q)$, let $\tilde {\mathcal F}(Q)$ be the family that consists of all cubes
$\{P\}\subset\mathcal F(Q)$ that are not contained in any cube $R\in\mathcal S$ with $R\subsetneq Q$.
Then we define
$$\tilde {\mathcal S} := \bigcup_{Q\in\mathcal S} \tilde {\mathcal F}(Q). $$
It is clear, by construction, that the augmented family $\tilde{\mathcal S}$ contains the original family $\mathcal S$. Furthermore, if $\mathcal S$ and each $\mathcal F(Q)$ are sparse families, then the augmented family $\tilde{\mathcal S}$ is also sparse.

To be specific, we have that if $\mathcal S\subset \mathcal D$ is an $\gamma$-sparse family then the augmented family
$\tilde{\mathcal S}$ built upon ${1\over2}$-sparse family $\mathcal F(Q)$, $Q\in\mathcal S$, is an ${\gamma\over 2(\gamma+1)}$-sparse family.

We now show \eqref{b-bQ eee0}. Take an arbitrary cube $Q\in\mathcal S$. Let $P_j$ be the cubes appearing in
\eqref{b-bQ eee1}. Denote by $\mathcal M(Q)$ the family of the
maximal pairwise disjoint cubes from $\tilde{\mathcal S}$ which are strictly  contained in $Q$. Then by the augmentation
process, $\cup_jP_j \subset \cup_{P\in\mathcal M(Q)} P$. Therefore, by
\eqref{b-bQ eee3}, we have
\begin{align}\label{b-bQ eee4}
|b(x)-b_Q|\chi_Q(x)\leq 4 C_{\mu,0}\cdot C\Omega(b,Q)+ \sum_{P\in\mathcal M(Q)} |b(x)-b_{P}| \chi_{P}(x).
\end{align}
Now split $\tilde{\mathcal S}(Q) :=\{ P\in\mathcal S:\ P\subset Q\}$ into the layers $\tilde{\mathcal S}(Q)=\cup_{k=0}^\infty\mathcal M_k$, where $\mathcal M_0:=\{Q\}$, $\mathcal M_1 := \mathcal M(Q)$ and $\mathcal M_k$ is the family of the maximal elements
of $\mathcal M_{k-1}$. Iterating \eqref{b-bQ eee4} $k$ times, we get that
\begin{align}\label{b-bQ eee5}
|b(x)-b_Q|\chi_Q(x)\leq 4 C_{\mu,0}\cdot C\sum_{P\in\tilde{\mathcal S}(Q)}\Omega(b,P)\chi_P(x)+ \sum_{P\in\mathcal M_k} |b(x)-b_{P}| \chi_{P}(x).
\end{align}
Now we observe that since $\tilde{\mathcal S}$ is ${\gamma\over 2(\gamma+1)}$-sparse,
\begin{align*}
\sum_{P\in\mathcal M_k}\mu(P) \leq {1\over k+1}\sum_{i=0}^k \sum_{P\in\mathcal M_i}\mu(P)\leq
{1\over k+1}   \sum_{P\in\tilde{\mathcal S}(Q)}\mu(P)\leq {2(\gamma+1)\over\gamma(k+1)}\mu(Q).
\end{align*}
By letting $k\to\infty$ in \eqref{b-bQ eee5}, we obtain \eqref{b-bQ eee0}.
\end{proof}


Let $T$ be a Calder\'on--Zygmund operator as in Definition \ref{def 1}. We now have the maximal truncated operator
$T^*$ defined by
\begin{align*}
T^*f(x) := \sup_{\epsilon>0}\bigg| \int_{d(x,y)>\epsilon} K(x,y)f(y)d\mu(y) \bigg|.
\end{align*}
We recall the standard Hardy--Littlewood maximal function $\mathcal Mf(x)$ on $X$, defined as 
$$ \mathcal Mf(x):=\sup_{B \ni x} {1\over \mu(B)}\int_B |f(y)|\,d\mu(y), $$
where the supremum is taken over all balls $B\subset X$.
%
%
We now have the grand maximal truncated operator $\mathcal M_T$ defined by
\begin{align*}
\mathcal M_Tf(x) := \sup_{B\ni x} \esssup_{\xi\in B} \Big| T\big(f\chi_{X\backslash C_{\widetilde j_0}B}\big)(\xi) \Big|,
\end{align*}
where the supremum is taken over all balls $B\subset X$ containing $x$, $\widetilde j_0$ is the smallest integer such that 
\begin{align}\label{widetildej0}
 2^{\widetilde j_0}>\max\{3A_0, 2A_0\cdot C_{adj}\}
\end{align}
and
$C_{\widetilde j_0} := 2^{\widetilde j_0+2} A_0$, where $C_{adj}$ is an absolute constant as mentioned in Section 2.2.
Given a ball $B_0\subset X$, for $x\in B_0$ we define a local grand maximal truncated operator $\mathcal M_{T,B_0}$ as follows:
\begin{align*}
\mathcal M_{T,B_0}f(x) := \sup_{B\ni x,\,  B\subset B_0} \esssup_{\xi\in B} \Big| T\big(f\chi_{C_{\widetilde j_0}B_0\backslash C_{\widetilde j_0}B}\big)(\xi) \Big|.
\end{align*}

Then we first claim that the following lemma holds.
\begin{lem} \label{lem TM}
The following pointwise estimates hold:
\begin{itemize}
\item[{\rm (i)}] for a.e. $x\in B_0$,
$   |T(f\chi_{C_{\widetilde j_0}B_0})(x)|\leq C\|T\|_{L^1\to L^{1,\infty}} |f(x)| + \mathcal M_{T,B_0}f(x).   $

\item[{\rm(ii)}] for all $x\in X$,
$ \mathcal M_Tf(x) \leq C \mathcal Mf(x) + T^*f(x).   $

\end{itemize}
\end{lem}
\begin{proof}
The result in the Euclidean setting is from  \cite[Lemma 3.2]{Ler1}.
Here we can adapt the proof in \cite{Ler1} to our setting of spaces of homogeneous type.
\end{proof}

Next we have the sparse domination for the higher order commutator.
\begin{thm}\label{thm sparse2}
Let $T$ be the Calder\'on--Zygmund operator as in Definition \ref{def 1} and let $b\in L^1_{loc}(X)$. For every  $f\in L^\infty(X)$ with bounded support,
there exist $\mathpzc T$ dyadic systems $\mathcal D^t, t=1,2,\ldots,\mathpzc T$ and $\eta$-sparse families $\mathcal S_t \subset \mathcal D^t$ such that
for a.e. $x\in X$,
\begin{align}\label{sparse domination high}
|T_b^m(f)(x)|\leq C \sum_{t=1}^{\mathpzc T}\sum_{k=0}^m C_m^k \sum_{Q\in \mathcal S_t} |b(x)-b_Q|^{m-k}
\bigg( {1\over \mu(Q)}\int_Q |b(z)-b_Q|^k |f(z)|dz \bigg)\chi_Q(x),
\end{align}
where $C_m^k:= {  m! \over (m-k)! \cdot k! }$.
\end{thm}
\begin{proof}
We follow the idea as in \cite{LOR2,IR} for the domination, and adapt it to our setting of space of homogeneous type.

{
Suppose $f$ is supported in a ball $B_0:= B(x_0,r)\subset X$. 
 We now consider a decomposition of $X$ with respect to this ball $B_0$.
We define the annuli $U_j:= 2^{j+1}B_0\backslash 2^j B_0$, $j\geq0$ and we choose $j_0$ to be the smallest integer such that
\begin{align}\label{j0}
j_0>\widetilde j_0\quad{\rm and}\quad 2^{j_0}> 4A_0.
\end{align}
Next, for each $U_j$, we choose the balls 
\begin{align}\label{Bjl}
\{ \widetilde B_{j,\ell} \}_{\ell=1}^{L_j}
\end{align}
 centred in $U_j$ and with radius $2^{j-\widetilde j_0}r$ to cover $U_j$. From the geometric doubling property \cite[p. 67]{CW1}, it is direct to see that 
\begin{align}\label{CA0mu}
 \sup _j L_j \leq C_{A_0,\mu,\widetilde j_0},
\end{align}  
 where $C_{A_0,\mu,\widetilde j_0}$ is an absolute constant depending only on $A_0$, $\widetilde j_0$ and $C_\mu$.

We now first study the properties of these $\widetilde B_{j,\ell}$. Denote $\widetilde B_{j,\ell}:= B(x_{j,\ell}, 2^{j-\widetilde j_0}r)$, where $\widetilde j_0$ is defined as in \eqref{widetildej0}. Then we have $C_{adj}\widetilde B_{j,\ell}:= B(x_{j,\ell}, C_{adj}2^{j-\widetilde j_0}r)$, where $C_{adj}$ is an absolute constant as mentioned in Section 2.2.
We claim that 
\begin{align}\label{claim1}
C_{adj}\widetilde B_{j,\ell}\cap U_{j+j_0}=\emptyset, \quad \forall j\geq0 \quad{\rm and}\quad \forall \ell=1,2,\ldots, L_j;
\end{align}
and that
\begin{align}\label{claim2}
C_{adj}\widetilde B_{j,\ell}\cap U_{j-j_0}=\emptyset, \quad \forall j\geq j_0 \quad{\rm and}\quad \forall \ell=1,2,\ldots, L_j.
\end{align}

%

Assume  \eqref{claim1} and \eqref{claim2} at the moment.
Now combining the properties as in \eqref{claim1} and \eqref{claim2}, we see that each $C_{adj}\widetilde B_{j,\ell}$ only intersects with at most $2j_0+1$ annuli $U_j$s. Moreover, 
for every $j$ and $\ell$, 
$C_{\widetilde j_0}\widetilde B_{j,\ell} $ covers $B_0$.  


}


Now for the given ball $B_0$ as above, we point out that from \eqref{eq:ball;included} we have that there exist an integer $t_0 \in \{1, 2, \ldots, \mathpzc T\}$ and $Q_0\in\mathscr{D}^{t_0}$ such that
$B_0\subseteq Q_0\subseteq C_{adj} B_0$.
Moreover, for this $Q_0$, as in \eqref{eq:contain} we use $B(Q_0)$ to denote 
the ball that contains $Q_0$ and has measure comparable to $Q_0$. Then it is easy to see that $B(Q_0)$ covers $B_0$ and
$\mu(B(Q_0))\ls \mu(B_0)$, where the implicit constant depends only on $C_{adj}$,  $C_\mu$ and $A_1$ as in \eqref{eq:contain}.

We show that there exists a {${1\over 2}$}-sparse family $\mathcal F^{t_0}\subset \mathcal D^{t_0}(Q_0)$, { the set of all dyadic cubes in $t_0$-th dyadic system that are contained in $Q_0$}, such that for a.e.
$x\in B_0$,
\begin{align}\label{ee a high}
|T_b^m(f\chi_{C_{\widetilde j_0}B(Q_0)})(x)|\leq C\sum_{k=0}^m C_m^k \sum_{Q\in \mathcal F^{t_0}}   \Big(|b(x)-b_{R_Q}|^{m-k} \Big| |f|\, |b-b_{R_Q}|^k\Big|_{C_{\widetilde j_0}B(Q)} \Big)\chi_Q(x).
\end{align}
Here, $R_Q$ is the dyadic cube in $\mathscr{D}^t$ for some $t \in \{1, 2, \ldots, \mathpzc T\}$ such that
$C_{\widetilde j_0} B(Q)\subset R_Q \subset C_{adj}\cdot C_{\widetilde j_0} B(Q)$, where $B(Q)$ is defined as in \eqref{eq:contain}, $j_0$ defined as in \eqref{j0} and $\widetilde j_0$  defined as in \eqref{widetildej0}. 

To prove the claim it suffices to prove the following recursive estimate:
there exist pairwise disjoint cubes $P_j \in \mathscr{D}^{t_0}(Q_0)$ such that $\sum_j\mu(P_j) \leq {1\over2}\mu(Q_0)$ and
\begin{align}\label{ee a claim high}
|T_b^m(f\chi_{C_{\widetilde j_0}B(Q_0)})(x)|\chi_{Q_0}(x)&\leq C\sum_{k=0}^m C_m^k     \Big(|b(x)-b_{R_{Q_0}}|^{m-k} \Big| |f|\, |b-b_{R_{Q_0}}|^k\Big|_{C_{\widetilde j_0}B_0} \Big)\chi_{Q_0}(x)\\
&\quad + \sum_j \big|  T_b^m(  f\chi_{C_{\widetilde j_0}B(P_j)}  ) (x)\big|\chi_{P_j}(x)\nonumber
\end{align}
for a.e. $x\in B_0$.

Iterating this estimate we obtain \eqref{ee a high} with $\mathcal F^{t_0}$ being the union of all the families $\{P_j^k\}$
where $\{P_j^0\} = \{Q_0\}$, $\{P_j^1\} = \{Q_j\}$ as mentioned above, and $\{P_j^k\}$ are the cubes obtained at the $k$-th stage of the iterative
process. It is also clear that $\mathcal F^{t_0}$ is a $1/2$-sparse family. 

Let us prove then the recursive estimate. We observe that for any arbitrary family of disjoint cubes $\{P_j\}\subset \mathcal D^{t_0}(Q_0)$, we have that
\begin{align*}
&|T_b^m(f\chi_{C_{\widetilde j_0}B(Q_0)})(x)|\chi_{Q_0}(x)\\
&\leq |T_b^m(f\chi_{C_{\widetilde j_0}B(Q_0)})(x)|\chi_{Q_0\backslash \cup_j P_j}(x) +\sum_j |T_b^m(f\chi_{C_{\widetilde j_0}B(Q_0)})(x)|\chi_{P_j}(x)\\
&\leq |T_b^m(f\chi_{C_{\widetilde j_0}B(Q_0)})(x)|\chi_{Q_0\backslash \cup_j P_j}(x) +\sum_j |T_b^m(f\chi_{C_{\widetilde j_0}B(Q_0) \backslash C_{\widetilde j_0}B(P_j) })(x)|\chi_{P_j}(x) \\
&\quad\quad+\sum_j |T_b^m(f\chi_{ C_{\widetilde j_0}B(P_j) })(x)|\chi_{P_j}(x).
\end{align*}
So it suffices to show that we can choose a family of pairwise disjoint cubes $\{P_j\}\subset\mathcal D^{t_0}(Q_0)$ with
$\sum_j\mu(P_j)\leq {1\over 2}\mu(Q_0)$ and such that for a.e. $x\in B_0$,
\begin{align*}
 &|T_b^m(f\chi_{C_{\widetilde j_0}B(Q_0)})(x)|\chi_{Q_0\backslash \cup_j P_j}(x) +\sum_j |T_b^m(f\chi_{C_{\widetilde j_0}B(Q_0) \backslash C_{\widetilde j_0}B(P_j) })(x)|\chi_{P_j}(x) \\
&\leq C\sum_{k=0}^m C_m^k |b(x)-b_{R_{Q_0}}|^{m-k} \Big| |f|\, |b-b_{R_{Q_0}}|^k\Big|_{C_{\widetilde j_0}B(Q_0)}.
\end{align*}

To see this, use the fact that 
$$ T_b^mf= T_{b-b_{R_{Q_0}}}^mf = \sum_{k=0}^m (-1)^k C_m^k T\big( (b-b_{R_{Q_0}})^k f\big) (b-b_{R_{Q_0}})^{m-k},   $$
we obtain that
\begin{align*}
 &|T_b^m(f\chi_{C_{\widetilde j_0}B(Q_0)})(x)|\chi_{Q_0\backslash \cup_j P_j}(x) +\sum_j |T_b^m(f\chi_{C_{\widetilde j_0}B(Q_0) \backslash C_{\widetilde j_0}B(P_j) })(x)|\chi_{P_j}(x) \\
& \leq \sum_{k=0}^m  C_m^k |T\big( (b-b_{R_{Q_0}})^k f\chi_{C_{\widetilde j_0}B(Q_0)}\big)(x)| |b(x)-b_{R_{Q_0}}|^{m-k}\chi_{Q_0\backslash \cup_j P_j}(x)\\
&\quad\quad+  \sum_{k=0}^m  C_m^k |T\big( (b-b_{R_{Q_0}})^k f\chi_{C_{\widetilde j_0}B(Q_0) \backslash C_{\widetilde j_0}B(P_j) }\big)(x)| |b(x)-b_{R_{Q_0}}|^{m-k}\chi_{P_j}(x)\\
&=: I_1+I_2.
\end{align*}

Now for $k=0,1,\ldots,m$, we define the set $E_k$ as
\begin{align*}
E_k &:= \Big\{ x\in B_0: |b(x)-b_{R_{Q_0}}|^k|f(x)| > \alpha \Big|  |b-b_{R_{Q_0}}|^k \,|f|\Big|_{C_{\widetilde j_0}B(Q_0)}\Big\}\\
&\quad \bigcup  \Big\{ x\in B_0: \mathcal M_{T,B_0}\big((b-b_{R_{Q_0}})^kf\big)(x) > \alpha C_T \Big|  |b-b_{R_{Q_0}}|^k\,|f|\Big|_{C_{\widetilde j_0}B(Q_0)}\Big\}
\end{align*}
and  $E:=\cup_{k=0}^m E_k$. Then, choosing $\alpha$ big enough (depending on $C_{\widetilde j_0}$, $C_{adj}$,  $C_\mu$ and $A_1$ as in \eqref{eq:contain}), we have that
$$ \mu(E)\leq {1\over 4 C_{\mu,0}} \mu(B_0), $$
 where $C_{\mu,0}$ is the constant in \eqref{Cmu0}.
We now apply the Calder\'on--Zygmund decomposition to the function $\chi_E$ on $B_0$ at the height
$\lambda:= {1\over 2 C_{\mu,0}}$, to obtain pairwise disjoint cubes $\{P_j\}\subset \mathcal D^{t_0}(Q_0)$ such that
$$ {1\over 2 C_{\mu,0}}\mu(P_j)\leq \mu(P_j\cap E)\leq {1\over2}\mu(P_j) $$
and
$ \mu(E\backslash \cup_j P_j)=0$. It follows that
$$ \sum_j\mu(P_j)\leq {1\over2}\mu(B_0)\quad{\rm and}\quad P_j\cap E^c\not=\emptyset.$$

Then we have
$$
\esssup_{\xi\in P_j}\bigg|T\Big(|b-b_{R_{Q_0}}|^k|f|\chi_{C_{\widetilde j_0}B(Q_0)\backslash C_{\widetilde j_0} B(P_j)}\Big)(\xi)\bigg|\leq C   \Big| |f|\, |b-b_{R_{Q_0}}|^k\Big|_{C_{\widetilde j_0}B(Q_0)},
$$
which allows us to control the summation in the term $I_2$ above.

Now from (i) in Lemma \ref{lem TM}, we obtain that for a.e. $x\in B_0$,
\begin{align*}
\Big|T\big( (b-b_{R_{Q_0}})^k f\chi_{C_{\widetilde j_0}B(Q_0)}\big)(x)\Big|\leq C |b(x)-b_{R_{Q_0}}|^k|f(x)| +  \mathcal M_{T,B_0}\big( (b-b_{R_{Q_0}})^k f\chi_{C_{\widetilde j_0}B(Q_0)}\big)(x).
\end{align*}
Since $\mu(E\backslash \cup_j P_j)=0$, we have that from the definition of the set $E$, the following estimate
$$   |b(x)-b_{R_{Q_0}}|^k|f(x)| \leq \alpha \Big| |f|\, |b-b_{R_{Q_0}}|^k\Big|_{C_{\widetilde j_0}B(Q_0)}   $$
  holds for a.e. $x\in B_0\backslash \cup_j P_j$, and also
$$ \mathcal M_{T,B_0}\big( (b-b_{R_{Q_0}})^k f\chi_{C_{\widetilde j_0}B_0}\big)(x) \leq \alpha C_T \Big| |f|\, |b-b_{R_{Q_0}}|^k\Big|_{C_{\widetilde j_0}B(Q_0)}  $$
  holds for a.e. $x\in B_0\backslash \cup_j P_j$.  These estimates allow us to control the summation in the term $I_1$ above.
Thus, we obtain that \eqref{ee a claim high} holds, which yields that  \eqref{ee a high} holds.


We now consider the partition of the space as follows. Suppose $f$ is supported in a ball $B_0\subset X$.  We have
$$ X = \bigcup_{j=0}^{\infty} 2^jB_0. $$

We now 
consider the annuli $U_j:= 2^{j+1}B_0\backslash 2^{j}B_0$ for $j\geq0$ and the covering $\{\widetilde B_{j,\ell}\}_{\ell=1}^{L_j}$ of $U_j$
as in \eqref{Bjl}.
We note that for each $\tilde B_{j,\ell}$, there exist $t_{j,\ell} \in \{1, 2, \ldots, \mathpzc T\}$ and $\tilde Q_{j,\ell}\in\mathscr{D}^{t_{j,\ell}}$ such that
$    \tilde B_{j,\ell}\subseteq \tilde Q_{j,\ell}\subseteq C_{adj}\tilde B_{j,\ell}$.
Moreover, we note that
for each such $\tilde B_{j,\ell}$, the enlargement $C_{\widetilde j_0} B(\tilde Q_{j,\ell})$ covers $B_0$ since $C_{\widetilde j_0}\widetilde B_{j,\ell} $ covers $B_0$.

We now apply \eqref{ee a high} to each $\widetilde B_{j,\ell}$, then
we obtain a ${1\over 2}$-sparse  family $\mathcal{\tilde F}_{j,\ell} \subset \mathscr{D}^{t_{j,\ell}}(\tilde Q_{j,\ell})$ such that
\eqref{ee a high} holds for a.e. $x\in \widetilde B_{j,\ell}$.  

Now we set $\mathcal F := \cup_{j,\ell} \mathcal{\tilde F}_{j,\ell}$.
Note that the balls $C_{adj}\widetilde B_{j,\ell}$ are overlapping at most $C_{A_0,\mu,\widetilde j_0}(2j_0+1)$ times, where $C_{A_0,\mu,\widetilde j_0}$ is the constant in \eqref{CA0mu}.   Then we obtain that
$\mathcal F $ is a ${1\over 2C_{A_0,\mu,\widetilde j_0}(2j_0+1)} $-sparse family and for a.e. $x\in X$,
\begin{align*}
|T_b^m(f)(x)|\leq C\sum_{k=0}^m C_m^k \sum_{Q\in \mathcal F}   \Big(|b(x)-b_{R_Q}|^{m-k} \Big| |f|\, |b-b_{R_Q}|^k\Big|_{C_{\widetilde j_0}B(Q)} \Big)\chi_Q(x).
\end{align*}

Since $C_{\widetilde j_0}B(Q)\subset R_Q$, and it is clear that $\mu(R_Q)\leq \overline C\mu(C_{\widetilde j_0}B(Q))$ ($\overline C$ depends only on $C_\mu$ and $C_{adj}$), we obtain that
$|f|_{C_{\widetilde j_0}B(Q)} \leq \overline C |f|_{R_Q} $. Next, we further set
$ \mathcal S_{t}:=\{R_Q\in \mathcal D^t:\ Q\in\mathcal F\}, \ \ t\in\{1,2,\ldots,\mathpzc T\}, $
 and from the fact that $\mathcal F$ is ${1\over 2C_{A_0,\mu,\widetilde j_0}(2j_0+1)}$-sparse, we can obtain that
 each family $\mathcal S_{t}$ is ${1\over 2C_{A_0,\mu,\widetilde j_0}(2j_0+1)\overline c}$-sparse.
 Now we let
 $$\eta:={1\over 2C_{A_0,\mu,\widetilde j_0}(2j_0+1)\overline c}, $$
 where $\overline c$ is a constant depending only on $\overline C$, $C_{\widetilde j_0}$ and the doubling constant $C_\mu$.
 Then it follows that
\eqref{sparse domination high}
holds, which finishes the proof.

In the end, we show \eqref{claim1} and \eqref{claim2}.

We first show \eqref{claim1} by contradiction. Suppose there exists some $\widetilde B_{j,\ell}= B(x_{j,\ell}, 2^{j-\widetilde j_0}r)$ such that $C_{adj}\widetilde B_{j,\ell}\cap U_{j+j_0}\not=\emptyset$. Then there exists at least one $y_0\in C_{adj}\widetilde B_{j,\ell}\cap U_{j+j_0}$. Then from the definition of $U_{j+j_0}$ we see that
$$ d(x_0, y_0)\geq 2^{j+j_0}r. $$
Moreover, from the definition of $x_{j,\ell}$ and the quasi triangular inequality \eqref{eqn:quasitriangleineq} we get that
\begin{align*}
d(x_0, y_0)\leq A_0\big(  d(x_0,x_{j,\ell}) + d(x_{j,\ell},y_0) \big)< A_0\big( 2^{j+1}r + C_{adj}2^{j-\widetilde j_0}r \big),
\end{align*}
which, together with the previous inequality, show that
$  2^{j+j_0}r \leq A_0\big( 2^{j+1}r + C_{adj}2^{j-\widetilde j_0}r \big) $, and hence we have
$$ 2^{j_0} \leq A_0\big( 2 + C_{adj}2^{-\widetilde j_0} \big) < 3A_0,$$
which contradicts to \eqref{j0}. Hence, we see that  \eqref{claim1} holds.

We now show \eqref{claim2}, and again we will prove it by contradiction. Suppose there exists some $\widetilde B_{j,\ell}= B(x_{j,\ell}, 2^{j-\widetilde j_0}r)$ such that $C_{adj}\widetilde B_{j,\ell}\cap U_{j-j_0}\not=\emptyset$, where $j\geq j_0$.
Then there exists at least one $y_0\in C_{adj}\widetilde B_{j,\ell}\cap U_{j-j_0}$. From the definition of $x_{j,\ell}$ and the quasi triangular inequality \eqref{eqn:quasitriangleineq},  we see that
\begin{align*}
2^j r \leq d(x_0, x_{j,\ell})\leq A_0\big(  d(x_0,y_{0}) + d(y_0,x_{j,\ell}) \big)< A_0\big( 2^{j- j_0+1}r +C_{adj}2^{j-\widetilde j_0}r   \big),
\end{align*}
which implies that
$$ 1\leq A_0(2^{- j_0+1}+ C_{adj}2^{-\widetilde j_0}  ), $$
which contradicts to \eqref{j0} and \eqref{widetildej0}. Hence, we see that  \eqref{claim2} holds.
\end{proof}

\section{Upper Bound of the Commutator $T_b^m $: Proof of Theorem \ref{thm1}}\label{s5}

In this section we provide the proof of Theorem \ref{thm1} following the idea in \cite{LOR2}.

Let $\mathcal D$ be a dyadic system in $(X,d,\mu)$ and let
$\mathcal S$ be a sparse family from $\mathcal D$.
We now define
$$ A_b^{m,k}f(x) := \sum_{Q\in\mathcal S} |b(x)-b_Q|^{m-k} \bigg( {1\over \mu(Q)}\int_Q |b(z)-b_Q|^k |f(z)|d\mu(z) \bigg)\chi_Q(x).$$
By duality, we have that
\begin{align}\label{A Lpnorm}
&\|A_b^{m,k}f\|_{L^p_\lambda(X)} \\
&\leq \sup_{g: \|g\|_{L^{p'}_\lambda(X)}=1}\sum_{Q\in\mathcal S}  \bigg(\int_Q |g(x)\lambda(x)||b(x)-b_Q|^{m-k}d\mu(x)\bigg) \bigg( {1\over \mu(Q)}\int_Q |b(z)-b_Q|^k |f(z)|d\mu(z) \bigg).\nonumber
\end{align}

Now by Lemma \ref{lem b-bQ}, there exists a sparse family
$\tilde{\mathcal S}\subset \mathcal D$ such that $\mathcal S\subset \tilde{\mathcal S}$ and for every cube $Q\in \tilde{\mathcal S}$, for a.e. $x\in Q$,
\begin{align*}
|b(x)-b_Q|\leq C\sum_{P\in \tilde{\mathcal S}, P\subset Q} \Omega(b,P)\chi_P(x).
\end{align*}

Since $b$ is in ${\rm BMO}_{\nu^{1\over m}}(X)$, then we have for a.e. $x\in Q$
\begin{align*}
|b(x)-b_Q|\leq C\|b\|_{{\rm BMO}_{\nu^{1\over m}}(X)} \sum_{P\in \tilde{\mathcal S}, P\subset Q}  {\nu^{1\over m}(P)\over \mu(P)} \chi_P(x).
\end{align*}
Then combining this estimate and inequality \eqref{A Lpnorm}, we further have
\begin{align*}
&\|A_b^{m,k}f\|_{L^p_\lambda(X)} \\
&\leq C \|b\|_{{\rm BMO}_{\nu^{1\over m}}(X)}^m \sup_{g: \|g\|_{L^{p'}_{\lambda}(X)}=1}\sum_{Q\in\mathcal S}  \Bigg({1\over \mu(Q)}\int_Q |g(x)\lambda(x)|   \bigg(\sum_{P\in \tilde{\mathcal S}, P\subset Q}  {\nu^{1\over m}(P)\over \mu(P)} \chi_P(x)\bigg)^{m-k}   d\mu(x)\Bigg) \\
&\quad\quad\times\bigg( {1\over \mu(Q)}\int_Q \bigg(\sum_{P\in \tilde{\mathcal S}, P\subset Q}  {\nu^{1\over m}(P)\over \mu(P)} \chi_P(z)\bigg)^k |f(z)|d\mu(z) \bigg)\, \mu(Q).
\end{align*}

Next, note that for each $\ell\in\mathbb N$,
we have
\begin{align*}
   \bigg(\sum_{P\in \tilde{\mathcal S}, P\subset Q}  {\nu^{1\over m}(P)\over \mu(P)} \chi_P(x)\bigg)^{\ell}
   &= \sum_{P_1,P_2,\ldots,P_\ell\in \tilde{\mathcal S}, P_1,P_2,\ldots,P_\ell\subset Q}  {\nu^{1\over m}(P_1)\over \mu(P_1)} \cdots {\nu^{1\over m}(P_\ell)\over \mu(P_\ell)} \chi_{\{P_1\cap\cdots\cap P_\ell\}}(x)\\
   &\leq \ell!   \sum_{P_1,\ldots,P_\ell\in \tilde{\mathcal S}, P_\ell\subset P_{\ell-1}\cdots \subset P_1\subset Q}  {\nu^{1\over m}(P_1)\over \mu(P_1)} \cdots {\nu^{1\over m}(P_\ell)\over \mu(P_\ell)} \chi_{ P_\ell}(x).
\end{align*}
Therefore, for an arbitrary function $h$, we have
\begin{align*}
   &\int_Q |h(x)| \bigg(\sum_{P\in \tilde{\mathcal S}, P\subset Q}  {\nu^{1\over m}(P)\over \mu(P)} \chi_P(x)\bigg)^{\ell} d\mu(x)  \\
   &\leq \ell!   \sum_{P_1,\ldots,P_\ell\in \tilde{\mathcal S}, P_\ell\subset P_{\ell-1}\cdots \subset P_1\subset Q}  {\nu^{1\over m}(P_1)\over \mu(P_1)} \cdots {\nu^{1\over m}(P_\ell)\over \mu(P_\ell)} |h|_{ P_\ell}\mu(P_\ell)\\
  &\leq C   \sum_{P_1,\ldots,P_{\ell-1}\in \tilde{\mathcal S}, P_{\ell-1}\subset P_{\ell-2}\cdots \subset P_1\subset Q}  {\nu^{1\over m}(P_1)\over \mu(P_1)} \cdots {\nu^{1\over m}(P_{\ell-1})\over \mu(P_{\ell-1})}    \sum_{ P_\ell\subset P_{\ell-1}, P_\ell\in\tilde{ \mathcal S} } |h|_{ P_\ell}\nu^{1\over m}(P_\ell) \\
  &\leq C   \sum_{P_1,\ldots,P_{\ell-1}\in \tilde{\mathcal S}, P_{\ell-1}\subset P_{\ell-2}\cdots \subset P_1\subset Q}  {\nu^{1\over m}(P_1)\over \mu(P_1)} \cdots {\nu^{1\over m}(P_{\ell-1})\over \mu(P_{\ell-1})}    \int_{P_{\ell-1}} A_{\tilde{\mathcal S}}( |h|)(x)\nu^{1\over m}(x) d\mu(x) \\
  &= C   \sum_{P_1,\ldots,P_{\ell-1}\in \tilde{\mathcal S}, P_{\ell-1}\subset P_{\ell-2}\cdots \subset P_1\subset Q}  {\nu^{1\over m}(P_1)\over \mu(P_1)} \cdots {\nu^{1\over m}(P_{\ell-1})\over \mu(P_{\ell-1})}    \Big( A_{\tilde{\mathcal S},\nu^{1\over m}}( |h|) \Big)_{P_{\ell-1}} \mu(P_{\ell-1}),
  \end{align*}
where $A_{\tilde{\mathcal S},\nu^{1\over m}}( |h|)(x)  := A_{\tilde{\mathcal S}}( |h|)(x)\nu^{1\over m}(x)$
and $ A_{\tilde{\mathcal S}}(h):= \sum_{Q\in\tilde{\mathcal S}} h_Q \chi_Q.$

By iteration, we obtain that
\begin{align*}
   \int_Q |h(x)| \bigg(\sum_{P\in \tilde{\mathcal S}, P\subset Q}  {\nu^{1\over m}(P)\over \mu(P)} \chi_P(x)\bigg)^{\ell} d\mu(x)
   &\leq C \int_Q A^\ell_{\tilde{\mathcal S},\nu^{1\over m}}( |h|)(x)\, d\mu(x),
   \end{align*}
where $A^\ell_{\tilde{\mathcal S},\nu^{1\over m}}$ denotes the $\ell$-fold iteration of $A_{\tilde{\mathcal S},\nu^{1\over m}}$.
Then we have
\begin{align*}
\|A_b^{m,k}f\|_{L^p_\lambda(X)}
&\leq C \|b\|_{{\rm BMO}_{\nu^{1\over m}}(X)}^m \sup_{g: \|g\|_{L^{p'}_{\lambda}(X)}=1}\sum_{Q\in\mathcal S}  \Bigg({1\over \mu(Q)}\int_Q  A^{m-k}_{\tilde{\mathcal S},\nu^{1\over m}}( |g|\,\lambda)(x) d\mu(x)\Bigg) \\
&\quad\quad\times\bigg( {1\over \mu(Q)}\int_Q   A^{k}_{\tilde{\mathcal S},\nu^{1\over m}}( |f|)(z) d\mu(z) \bigg)\, \mu(Q)\\
&\leq C \|b\|_{{\rm BMO}_{\nu^{1\over m}}(X)}^m \sup_{g: \|g\|_{L^{p'}_{\lambda}(X)}=1}   \Bigg(\int_X A_{\tilde{\mathcal S}}\Big( A^{k}_{\tilde{\mathcal S},\nu^{1\over m}}( |f|)\Big)(x)   A^{m-k}_{\tilde{\mathcal S},\nu^{1\over m}}( |g|\,\lambda)(x) d\mu(x)\Bigg).
\end{align*}
Observe that $A_{\tilde{\mathcal S}}$ is self-adjoint. We have
\begin{align*}
&\int_X A_{\tilde{\mathcal S}}\Big( A^{k}_{\tilde{\mathcal S},\nu^{1\over m}}( |f|)\Big)(x)   A^{m-k}_{\tilde{\mathcal S},\nu^{1\over m}}( |g|\,\lambda)(x) d\mu(x)\\
&=\int_X A_{\tilde{\mathcal S}}\bigg(A_{\tilde{\mathcal S}}\Big( A^{k}_{\tilde{\mathcal S},\nu^{1\over m}}( |f|)\Big)\bigg)(x)   A^{m-k-1}_{\tilde{\mathcal S},\nu^{1\over m}}( |g|\,\lambda)(x) d\mu(x)\\
&=\cdots\\
&=\int_X A_{\tilde{\mathcal S}}\Big( A^{m}_{\tilde{\mathcal S},\nu^{1\over m}}( |f|)\Big)(x)    |g(x)|\,\lambda(x) d\mu(x).
\end{align*}
Then from H\"older's inequality, we further have
\begin{align*}
&\|A_b^{m,k}f\|_{L^p_\lambda(X)}\\
&\leq C \|b\|_{{\rm BMO}_{\nu^{1\over m}}(X)}^m \sup_{g: \|g\|_{L^{p'}_{\lambda}(X)}=1}
\Bigg(\int_X \bigg[A_{\tilde{\mathcal S}}\Big( A^{m}_{\tilde{\mathcal S},\nu^{1\over m}}( |f|)\Big)(x) \bigg]^p  \,\lambda(x) d\mu(x)\Bigg)^{1\over p}\|g\|_{L^{p'}_\lambda(X)}\\
&\leq C \|b\|_{{\rm BMO}_{\nu^{1\over m}}(X)}^m [\lambda]_{A_p}^{\max\{1,{1\over p-1}\}}
\big\| A^{m}_{\tilde{\mathcal S},\nu^{1\over m}}( |f|)\big\|_{L^p_\lambda(X)}\\
&= C \|b\|_{{\rm BMO}_{\nu^{1\over m}}(X)}^m [\lambda]_{A_p}^{\max\{1,{1\over p-1}\}}
\Big\| A_{\tilde{\mathcal S}}\Big(A^{m-1}_{\tilde{\mathcal S},\nu^{1\over m}}( |f|)\Big)\Big\|_{L^p_{\lambda\cdot \nu^{p\over m}}(X)}\\
&\leq C \|b\|_{{\rm BMO}_{\nu^{1\over m}}(X)}^m \Big([\lambda]_{A_p} [\lambda\cdot \nu^{p\over m}]_{A_p}\Big)^{\max\{1,{1\over p-1}\}}
\Big\| A^{m-1}_{\tilde{\mathcal S},\nu^{1\over m}}( |f|)\Big\|_{L^p_{\lambda\cdot \nu^{p\over m}}(X)}.
\end{align*}
Then by iteration we have that
\begin{align*}
\|A_b^{m,k}f\|_{L^p_\lambda(X)}
&\leq C \|b\|_{{\rm BMO}_{\nu^{1\over m}}(X)}^m \Big([\lambda]_{A_p} {[\lambda\cdot \nu^{p\over m}]_{A_p}\cdots   [\lambda\cdot\nu^{mp\over m} ]_{A_p}  }\Big)^{\max\{1,{1\over p-1}\}}
\|  f\|_{L^p_{\lambda\cdot \nu^{mp\over m}}(X)}\\
&\leq C \|b\|_{{\rm BMO}_{\nu^{1\over m}}(X)}^m \Big([\lambda]_{A_p}[\mu]_{A_p} \prod_{i=1}^{m-1}[\lambda^{1-{i\over m}} \cdot\nu^{i\over m} ]_{A_p}\Big)^{\max\{1,{1\over p-1}\}}
\|  f\|_{L^p_{\nu}(X)}.
\end{align*}
We note that by H\"older's inequality we have
$$\prod_{i=1}^{m-1}[\lambda^{1-{i\over m}} \cdot\nu^{i\over m} ]_{A_p} \leq  \Big( [\lambda]_{A_p}[\mu]_{A_p}\Big)^{m-1\over2}.$$
As a consequence, we have that
\begin{align*}
\|A_b^{m,k}f\|_{L^p_\lambda(X)}
&\leq C \|b\|_{{\rm BMO}_{\nu^{1\over m}}(X)}^m \Big([\lambda]_{A_p}[\mu]_{A_p} \Big)^{{m+1\over2}\cdot\max\{1,{1\over p-1}\}}
\|  f\|_{L^p_{\nu}(X)}.
\end{align*}

\section{Lower Bound of the Commutator $T_b^m$: Proof of Theorem \ref{thm2}}\label{s6}

In this section, we use some ideas from \cite{Hyt,LOR,LOR2} and adapt them to our general setting
with the aim to prove Theorem \ref{thm2}.
To begin with, let $T$ be the Calder\'on--Zygmund operator as in Definition \ref{def 1} with the kernel $K$ and $\omega$ satisfying $\omega(t)\to0$ as $t\to0$, and satisfy the homogeneous condition as in
\eqref{e-assump cz ker low bdd weak}.

We first introduce  another version of the homogeneous condition: There exist positive constants $3\le A_1\le A_2$ such that for any ball $B:=B(x_0, r)\subset X$, there exist balls $\widetilde B:=B(y_0, r)$
such that $A_1 r\le d(x_0, y_0)\le A_2 r$, 
and  for all $(x,y)\in ( B\times \widetilde{B})$, $K(x, y)$ does not change sign and
\begin{equation}\label{e-assump cz ker low bdd}
|K(x, y)|\gs \frac1{\mu(B)}.
\end{equation}
If the kernel $K(x, y):=K_1(x,y)+i K_2(x,y)$ is complex-valued, where $i^2=-1$,
then at least one of $K_i$ satisfies \eqref{e-assump cz ker low bdd}.

Then we first point out that the homogeneous condition \eqref{e-assump cz ker low bdd weak} implies
\eqref{e-assump cz ker low bdd}.

\begin{prop}\label{prop homogeneous}
Let $T$ be the Calder\'on--Zygmund operator as in Definition \ref{def 1} with the kernel $K$ and $\omega$ satisfying $\omega(t)\to0$ as $t\to0$, and satisfy the homogeneous condition as in
\eqref{e-assump cz ker low bdd weak}. Then $T$ satisfies \eqref{e-assump cz ker low bdd}.
\end{prop}
\begin{proof}
Let $T$ be the Calder\'on--Zygmund operator as in Definition \ref{def 1} with the kernel $K$ and $\omega$ satisfying $\omega(t)\to0$ as $t\to0$, and satisfy the homogeneous condition as in
\eqref{e-assump cz ker low bdd weak}.
Since $\omega(t)\to0$ as $t\to0$, there exists $\delta\in(0,1)$ such that when $0<t<\delta$,
$$ \omega(t)< {1\over 20\cdot 3^n\cdot C\cdot C_\mu \cdot c_0}, $$
where $c_0$ is from \eqref{e-assump cz ker low bdd weak}, $C$ is from Definition \ref{def 1} and $C_\mu$ is from \eqref{upper dimension}.

For all  numbers $A$ with
\begin{align}\label{A}
A>\max\Big\{3, 2+{1\over\delta}, 2A_0\Big\},
\end{align}
and for any ball $B:=B(x_0, r)\subset X$, according to the homogeneous condition  \eqref{e-assump cz ker low bdd weak}, there exists a point $y_0\in B(x_0,\overline C Ar)\backslash B(x_0,Ar)$
such that
\begin{align}\label{Kx0y0}
 |K(x_0,y_0)|\geq {1\over c_0 \mu(B(x_0,Ar))}.
\end{align}

Next, from the smoothness condition \eqref{smooth of C-Z-S-I-O},
we have that for every $x\in B(x_0,r)$ and $y\in B(y_0,r)$,
\begin{align*}
|K(x,y)-K(x_0,y_0)|
&\leq |K(x,y)-K(x,y_0)|+|K(x,y_0)-K(x_0,y_0)|\\
&\leq {C\over V(x,y)}\omega\Big({d(y, y_0)\over d(x,y)}\Big)+{C\over V(x_0,y_0)}\omega\Big({d(x, x_0)\over d(x_0,y_0)}\Big)\\
&\leq {C\over \mu(B(x_0,(A-2)r))}\omega\Big({r\over (A-2)r}\Big)+ {C\over \mu(B(x_0,Ar))}\omega\Big({r\over Ar}\Big)\\
&\leq {2C\over \mu(B(x_0,(A-2)r))}\omega\Big({1\over A-2}\Big),
\end{align*}
where we use the fact that $\omega(t)$ is increasing. Next, by \eqref{upper dimension}, we obtain that
\begin{align*}
|K(x,y)-K(x_0,y_0)|
\leq 2C C_\mu \Big({A\over A-2}\Big)^n\omega\Big({1\over A-2}\Big)\  {1\over \mu(B(x_0,Ar))}\leq {1\over 10c_0 \mu(B(x_0,Ar))},
\end{align*}
where the last inequality follows from the choice of $A$ as in \eqref{A}.

We now fix a positive number $A_1$ satisfying \eqref{A} and set $A_2 := \overline C A_1$.

We first consider the kernel $K(x,y)$ to be a real-valued function. If $K(x_0,y_0)>0$, then for every $x\in B(x_0,r)$ and $y\in B(y_0,r)$
we have that
\begin{align*}
K(x,y) &=K(x_0,y_0)- (K(x_0,y_0)-K(x,y))\geq K(x_0,y_0)- |K(x,y)-K(x_0,y_0)|\\
&\geq {1\over c_0 \mu(B(x_0,Ar))} - {1\over 10c_0 \mu(B(x_0,Ar))}> {1\over 2c_0 \mu(B(x_0,Ar))}.
\end{align*}
Similarly if $K(x_0,y_0)<0$, then  every $x\in B(x_0,r)$ and $y\in B(y_0,r)$
we have that $$K(x,y) < -{1\over 2c_0 \mu(B(x_0,Ar))}.$$
Thus, combining these two cases we obtain that \eqref{e-assump cz ker low bdd} holds.

Next we consider the kernel $K(x,y)$ to be a complex function. We write
$K(x,y)= K_1(x,y)+ i K_2(x,y)$, with $i^2=-1$. Then \eqref{Kx0y0} implies that
$${\rm either}\ \  |K_1(x_0,y_0)|\geq {\sqrt 2\over 2c_0 \mu(B(x_0,Ar))}  {\rm\ \ \ \  or\ \ \ \ }  |K_2(x_0,y_0)|\geq {\sqrt 2\over 2 c_0 \mu(B(x_0,Ar))}.$$

Suppose $|K_j(x_0,y_0)|\geq {\sqrt 2\over 2c_0 \mu(B(x_0,Ar))}$ for some $j\in\{1,2\}$. If $K_j(x_0,y_0)>0$, then every $x\in B(x_0,r)$ and $y\in B(y_0,r)$
we have that
\begin{align*}
K_j(x,y) &=K_j(x_0,y_0)- (K_j(x_0,y_0)-K_j(x,y))\geq K_j(x_0,y_0)- |K(x,y)-K(x_0,y_0)|\\
&\geq {\sqrt2\over 2c_0 \mu(B(x_0,Ar))} - {1\over 10c_0 \mu(B(x_0,Ar))}> {1\over 2c_0 \mu(B(x_0,Ar))}.
\end{align*}
Similarly if $K_j(x_0,y_0)<0$ for some $j\in\{1,2\}$, then  every $x\in B(x_0,r)$ and $y\in B(y_0,r)$
we have that $$K_j(x,y) < -{1\over 2c_0 \mu(B(x_0,Ar))}.$$
Thus, \eqref{e-assump cz ker low bdd} holds for $K_j(x,y)$.
%
%
%
%
%

The proof of Proposition \ref{prop homogeneous} is complete.
\end{proof}



\begin{defn}\label{d-median value}
  By a median value of a real-valued measurable function $f$ over $B$ we mean a possibly non-unique, real number $\alpha_B(f)$ such that
$$\mu(\{x\in B: f(x)>\alpha_B(f)\})\leq \frac12\mu(B)\,\, \mbox{and}\,\,\mu(\{x\in B: f(x)<\alpha_B(f)\})\leq \frac12\mu(B). $$
\end{defn}
It is known that for a given function $f$ and ball $B$, the median value exists and may not be unique; see, for example, \cite{J}.

\begin{lem}\label{l-bmo decomp low bdd}
  Let $b$ be a real-valued measurable function.
For any ball $B$, let $\widetilde B$ be as in \eqref{e-assump cz ker low bdd}. Then
there exist measurable sets $E_1, E_2\subset  B$, and $F_1, F_2\subset \widetilde B$, such that
\begin{itemize}
  \item[{\rm (i)}] $ B=E_1\cup E_2$,  $\widetilde B=F_1\cup F_2$ and $\mu(F_i)\ge\frac{1}{2}\mu(\widetilde B)$, $i=1,2$;
  \item[{\rm (ii)}]  $b(x)-b(y)$ does not change sign for all $(x, y)$ in $E_i\times F_i$, $i=1,2$;
  \item[{\rm (iii)}]  $|b(x)-\alpha_{\widetilde B}(b)|\leq |b(x)-b(y)|$ for all $(x, y)$ in $E_i\times F_i$, $i=1,2$.
\end{itemize}
\end{lem}
\begin{proof}
For the given balls $B$ and $\widetilde B$, following the idea in \cite[Proposition 3.1]{LOR2} we set
$$
F_1:= \{y\in \widetilde B: b(y)\leq \alpha_{\widetilde B}(b)\}\ \ \ \text{and}\ \ \ \ \ F_2:= \{y\in \widetilde B: b(y)\geq \alpha_{\widetilde B}(b)\}.
$$
Moreover, we define
$$ E_1:=\{x\in  B: b(x)\geq \alpha_{\widetilde B}(b)\}\ \ \ \mbox{and}\ \ \ \ E_2:=\{x\in  B: b(x)\leq \alpha_{\widetilde B}(b)\}
$$
Then by Definition \ref{d-median value}, we see that $\mu(F_i)\ge\frac{1}{2}\mu(\widetilde B)$, $i=1,2$.
Moreover, for $(x,y)\in E_i\times F_i$, $i=1,2$,
\begin{align*}
|b(x)-b(y)|
&= \lf|b(x)-\alpha_{\widetilde B}(b)+\alpha_{\widetilde B}(b)-b(y)\r|\\
&=
\lf|b(x)-\alpha_{\widetilde B}(b)\r|+\lf|\alpha_{\widetilde B}(b)-b(y)\r|\geq \lf|b(x)-\alpha_{\widetilde B}(b)\r|.
\end{align*}
This finishes the proof of Lemma \ref{l-bmo decomp low bdd}.
\end{proof}

We now return to the proof of Theorem \ref{thm2}, following the approach and method in \cite{LOR2}.
\begin{proof}[Proof of Theorem \ref{thm2}]
For given $b\in L^1_{\rm loc}(X)$ and for any ball $B$, let $\Omega(b,B)$ be the oscillation
as in \eqref{Omega oscillation}.
Under the assumptions of Theorem \ref{thm2},
we will show that for any ball $B$,
\begin{align}\label{e-mean osci weigh upp bdd}
\Omega(b, B)\ls \frac{{\nu^{1\over m}}(B)}{\mu(B)}.
\end{align}

Without loss of generality, we assume that $K(x,y)$ is real-valued.
Let $B$ be a ball. We apply the assumption \eqref{e-assump cz ker low bdd} and Lemma \ref{l-bmo decomp low bdd}
to get sets $E_i, F_i,\,i=1,2$.

On the one hand, by Lemma \ref{l-bmo decomp low bdd} and \eqref{e-assump cz ker low bdd}, we have that for $f_i:=\chi_{F_i}$, $i=1,2$,
\begin{align*}
\frac1{\mu(B)}\sum_{i=1}^2\int_{ B}\lf|T^m_bf_i(x)\right|\,d\mu(x)
&\ge\frac1{\mu(B)}\sum_{i=1}^2\int_{E_i}\lf|T^m_bf_i(x)\right|\,d\mu(x)\\
&=\frac1{\mu(B)}\sum_{i=1}^2\int_{E_i}\int_{F_i}|b(x)-b(y)|^m|K(x,y)|\,d\mu(y)\,d\mu(x)\\
&\gs\frac1{\mu(B)}\sum_{i=1}^2\int_{E_i}\int_{F_i}\frac{|b(x)-\alpha_{\widetilde B}(b)|^m}{\mu(B)}\,d\mu(y)\,d\mu(x)\\
&\gs \frac1{\mu(B)}\int_B\lf|b(x)-\alpha_{\widetilde B}(b)\r|^m\,d\mu(x)\\
&\gs  \Omega(b; B)^m.
\end{align*}

On the other hand, from H\"older's inequality and the boundedness of $T^m_b$, we deduce that
\begin{align*}
&\frac1{\mu(B)}\sum_{i=1}^2\int_{ B}\lf|T^m_bf_i(x)\right|\,d\mu(x)\\
&\quad\le \frac1{\mu(B)}\sum_{i=1}^2\lf[\int_{ B}\lf|T^m_bf_i(x)\right|^p\lambda_2(x)\,d\mu(x)\r]^{1/p}\lf(\int_{ B}\lambda_2(x)^{-\frac1{p-1}}d\mu(x)\right)^{1/p'}\\
&\quad\ls \frac1{\mu(B)}\sum_{i=1}^2[\lambda_1(F_i)]^{1/p}\lf(\int_{ B}\lambda_2(x)^{-\frac1{p-1}}d\mu(x)\right)^{1/p'}\\
&\quad\ls \frac1{\mu(B)}[\lambda_1(\widetilde B)]^{1/p}\lf(\int_{ B}\lambda_2(x)^{-\frac1{p-1}}d\mu(x)\right)^{1/p'}\\
&\quad\ls \frac1{\mu(B)}[\lambda_1(B)]^{1/p}\lf(\int_{B}\lambda_2(x)^{-\frac1{p-1}}d\mu(x)\right)^{1/p'},
\end{align*}
where in the last inequality, we use the fact that 
$K_1 r_B\le d(x_B, x_{\widetilde B})\le K_2 r_B$ and  $\lambda_1\in A_p$. 

Combining the two inequalities above and invoking $\lambda_i\in A_p$, we conclude that
\begin{align*}
\Omega(b, B)^m\ls\frac1{\mu(B)}[\lambda_1(B)]^{1/p}\lf(\int_{B}\lambda_2(x)^{-\frac1{p-1}}d\mu(x)\right)^{1/p'}\ls
\bigg(\frac{\nu^{1\over m}(B)}{\mu(B)}\bigg)^m,
\end{align*}
where the last inequality follows from the argument as in the proof of Theorem 1.1 in \cite{LOR2}, by using \eqref{e-reverse holder}.
Thus, \eqref{e-mean osci weigh upp bdd} holds and hence, the proof of Theorem \ref{thm2} is complete.
\end{proof}



\section{Weighted Hardy space, Duality and Weak Factorisation: Proof of Theorem \ref{thm5}}\label{s3}
\noindent
In this section we study the weighted Hardy, BMO spaces and duality, as well as their dyadic versions on spaces of homogeneous type.

\subsection{Dyadic Littlewood--Paley Square Function}
Following the form in \cite{HLW} we now introduce the dyadic Littlewood--Paley square function on spaces of homogeneous type.
\begin{defn}\label{def:dyadic_discrete_square_function}
Given a dyadic grid $\mathscr{D}$ on $X$, the dyadic square function $S_{\mathscr{D}}$ is defined by:
\begin{align*} 
S_{\mathscr{D}}f:=\bigg[\sum_{Q\in \mathscr{D}}\sum_{\epsilon=1}^{M_{Q}-1}|\langle f,h_{Q}^{\epsilon}\rangle |^2{\chi_{Q}\over \mu(Q)}\bigg]^{1\over 2}.
\end{align*}
\end{defn}

Our main result in this subsection is:
\begin{thm}\label{thm square}
Suppose $1<p<\infty$ and $w\in A_p$. Then we have
$$  \|S_\mathscr{D}f\|_{L^p_w(X)}\leq  C_p [w]_{A_p}^{\max\{1, {1\over p-1}\}}\|f\|_{L^p_{w}(X)}.   $$
\end{thm}

We prove this theorem by following the idea in \cite[Theorem 3.1 and Corollary 3.2]{PP}. To begin with, we first introduce an auxiliary lemma.

\begin{lem}\label{lem A2}
Let $w$ be an $A_{2}$ weight in $(X, \rho, \mu)$. Then
\[
\sum_{Q\in \mathscr{D}} \sum_{\epsilon =1}^{M_{Q}-1} | \langle f, h_{Q}^{\epsilon} \rangle|^{2} \frac{1}{\langle w\rangle_{Q}} \lesssim [w]_{A_{2}} \|f\|_{L^{2}_{w^{-1}}(X)}^{2} \quad\text{ for all } f\in L^{2}_{w^{-1}}(X),
\]
where $\langle w\rangle_Q:= {1\over \mu(Q)}\int_Q w(x)d\mu(x)$.
\end{lem}
\begin{proof}
Recall from \cite{KLPW}, we have
$
h_{Q}^{\epsilon} = a_{\epsilon}\chi_{Q_{\epsilon}} - b_{\epsilon}  \chi_{E_{\epsilon +1}},
$
where
\[
a_{\epsilon}:=\sqrt{\frac{\mu(E_{\epsilon +1})}{\mu(Q_{\epsilon})\mu(E_{\epsilon})}} \text{, }\ \  b_{\epsilon}:=\sqrt{\frac{\mu(Q_{\epsilon})}{\mu(E_{\epsilon})\mu(E_{\epsilon+1})}} \quad\text{ and } \quad E_{\epsilon} = Q_{\epsilon} \cup E_{\epsilon+1},
\]
where $Q_{\epsilon}$ and $E_{\epsilon +1}$ are disjoint.
Now we introduce the weighted Haar system $\{h^{w,\epsilon}_{Q}\}_{1\leq \epsilon \leq M_{Q} -1, Q \in \mathscr{D}}$ in $L^{2}_{w} (X)$,
where
\[
h^{w,\epsilon}_{Q} := \frac{1}{\sqrt{w(E_{\epsilon})}} \Big( \frac{\sqrt{w(E_{\epsilon+1})}}{\sqrt{w(Q_{\epsilon})}} \chi_{Q_{\epsilon}} - \frac{\sqrt{w(Q_{\epsilon})}}{\sqrt{w(E_{\epsilon+1})}} \chi_{E_{\epsilon+1}} \Big).
\]
Note that when $w=1$, we have
\[
h^{1,\epsilon}_{Q} := h^{\epsilon}_{Q} = \frac{1}{\sqrt{\mu(E_{\epsilon})}} \Big( \frac{\sqrt{\mu(E_{\epsilon+1})}}{\sqrt{\mu(Q_{\epsilon})}} \chi_{Q_{\epsilon}} - \frac{\sqrt{\mu(Q_{\epsilon})}}{\sqrt{\mu(E_{\epsilon+1})}} \chi_{E_{\epsilon+1}} \Big).
\]
We set
\[{
h^{1}_{E_{\epsilon}} := \frac{\chi_{E_{\epsilon}}}{{\mu(E_{\epsilon})}}   },
\]
and write
$ h_{Q}^{\epsilon} = C_{Q}(w,\epsilon) h_{Q}^{w,\epsilon} + D_{Q}(w,\epsilon) h_{E_{\epsilon}}^{1}.
$

It is easy to see that $\int_{Q} h_{Q}^{w,\epsilon} dw = 0 \text{ and } \int_{Q} (h_{Q}^{w,\epsilon})^2 dw = 1$. This implies
\[ D_{Q}(w,\epsilon)  
= \frac{\hat{w} (Q,\epsilon)}{\langle w \rangle _{E_{\epsilon}}}, \quad{\rm where\ } { \hat{w} (Q,\epsilon):= \langle w, h_{Q}^{\epsilon} \rangle }
\]
and, after some computation,
\begin{align*}
C_{Q}(w,\epsilon) ^2 = &\frac{\mu(E_{\epsilon +1})}{\mu(E_{\epsilon})} \langle w \rangle_{Q_{\epsilon}} + \frac{\mu(Q_{\epsilon})}{\mu(E_{\epsilon})} \langle w \rangle _{E_{\epsilon +1}} - \frac{\mu(E_{\epsilon +1})}{\mu(E_{\epsilon})} \frac{w(Q_{\epsilon})}{w(E_{\epsilon})} \langle w \rangle_{Q_{\epsilon}}\\
& - \frac{\mu(Q_{\epsilon})}{\mu(E_{\epsilon})} \frac{w(E_{\epsilon +1})}{w(E_{\epsilon})} \langle w \rangle_{E_{\epsilon +1}}
+ 2 \frac{w(E_{\epsilon +1})}{w(E_{\epsilon})} \frac{w(Q_{\epsilon})}{\mu(E_{\epsilon})}.
\end{align*}
Note that it doesn't really matter what $C_{Q}(w,\epsilon)$ really is as long as we have some nice bound for it. In fact,
{from Lemma 4.6 in \cite{KLPW},} we have that
\begin{align*}
C_{Q}(w,\epsilon) ^2 &\leq \frac{\mu(E_{\epsilon +1})}{\mu(E_{\epsilon})} \langle w \rangle_{Q_{\epsilon}} + \frac{\mu(Q_{\epsilon})}{\mu(E_{\epsilon})} \langle w \rangle _{E_{\epsilon +1}}
\lesssim \frac{w(Q_{\epsilon})+w(E_{\epsilon +1})}{\mu(E_{\epsilon})}= \langle w \rangle_{E_{\epsilon}}\lesssim \langle w \rangle_{Q},
\end{align*}
which implies that
$
C_{Q}(w,\epsilon)^{2} \langle w \rangle_{Q}^{-1} \lesssim 1.
$

Now,
\begin{align*}
\sum_{Q\in \mathscr{D}} \sum_{\epsilon=1}^{M_{Q} -1} \langle w \rangle_{Q}^{-1} |\langle f, h_{Q}^{\epsilon} \rangle |^2 &= \sum_{Q\in \mathscr{D}} \sum_{\epsilon=1}^{M_{Q} -1} \langle w \rangle_{Q}^{-1} |\langle f, C_{Q}(w,\epsilon) h_{Q}^{w,\epsilon} + D_{Q}(w,\epsilon) h_{E_{\epsilon}}^{1} \rangle|^2\\
&= \sum_{Q\in \mathscr{D}} \sum_{\epsilon=1}^{M_{Q} -1} \langle w \rangle_{Q}^{-1} |C_{Q}(w,\epsilon) \langle f, h_{Q}^{w,\epsilon} \rangle + D_{Q}(w,\epsilon) \langle f, h_{E_{\epsilon}}^{1} \rangle|^2\\
&= \sum_{Q\in \mathscr{D}} \sum_{\epsilon=1}^{M_{Q} -1} \langle w \rangle_{Q}^{-1} C_{Q}(w,\epsilon)^{2} | \langle f, h_{Q}^{w,\epsilon} \rangle|^{2}\\
&\quad+ 2 \sum_{Q\in \mathscr{D}} \sum_{\epsilon=1}^{M_{Q} -1} \langle w \rangle_{Q}^{-1} C_{Q}(w,\epsilon)\, D_{Q}(w,\epsilon)\, \langle f, h_{Q}^{w,\epsilon} \rangle \,\langle f,h_{E_{\epsilon}}^{1}\rangle\\
&\quad+\sum_{Q\in \mathscr{D}} \sum_{\epsilon=1}^{M_{Q} -1} \langle w \rangle_{Q}^{-1} D_{Q}(w,\epsilon)^{2} |\langle f,h_{E_{\epsilon}}^{1} \rangle |^{2}\\
&=: S_{1} + S_{2} + S_{3}.
\end{align*}
$S_{2}$ can be bounded by $\sqrt{S_{1}}\sqrt{S_{3}}$, so it suffices to bound $S_{1}$ and $S_{3}$. By using the bound on $C_{Q}(w,\epsilon)$, we have
\begin{align*}
S_{1} 
&\lesssim \sum_{Q\in \mathscr{D}} \sum_{\epsilon=1}^{M_{Q} -1} |\langle f, h_{Q}^{w,\epsilon} \rangle |^2= \sum_{Q\in \mathscr{D}} \sum_{\epsilon=1}^{M_{Q} -1} |\langle w^{-1} f, h_{Q}^{w,\epsilon} \rangle_{L^{2}_{w}(X)} |^2
\leq \|f\|_{L^{2}_{w^{-1}}(X)}^{2}.
\end{align*}
On the other hand,
\begin{align*}
S_{3} &= \sum_{Q\in \mathscr{D}} \sum_{\epsilon=1}^{M_{Q} -1} \langle w \rangle_{Q}^{-1} D_{Q}(w,\epsilon)^{2} |\langle f,h_{E_{\epsilon}}^{1} \rangle |^{2}= \sum_{Q\in \mathscr{D}} \sum_{\epsilon=1}^{M_{Q} -1} D_{Q}(w,\epsilon)^{2} \langle f \rangle_{E_{\epsilon}}^{2} \langle w \rangle _{Q}^{-1}.
\end{align*}
Now,
\begin{align*}
\frac{1}{\mu(E_{\epsilon})} \sum_{R\subseteq Q} \sum_{\eta: E_{\eta}\subseteq E_{\epsilon}} D_{R}(w,\eta)^2 \langle w \rangle _{R}^{-1} \langle w \rangle_{E_{\eta}}^{2} &= \frac{1}{\mu(E_{\epsilon})} \sum_{R\subseteq Q} \sum_{\eta: E_{\eta}\subseteq E_{\epsilon}} \frac{\hat{w}(R,\eta)^{2}}{\langle w \rangle_{E_{\eta}}^{2}} \langle w \rangle _{R}^{-1} \langle w \rangle_{E_{\eta}}^{2}\\
&\lesssim \frac{1}{\mu(E_{\epsilon})} \sum_{R\subseteq Q} \sum_{\eta: E_{\eta}\subseteq E_{\epsilon}} \hat{w}(R,\eta)^{2} { \langle w \rangle_{R}^{-1}  }\\
&\lesssim [w]_{A_{2}} \langle w \rangle_{E_{\epsilon}},
\end{align*}
where the last inequality follows from a Bellman function technique that can be found in \cite{C}. Thus, by adopting the remark of Treil in \cite[Section 5]{T} on the dyadic Carleson Embedding Theorem on a general space of homogeneous type, we get:
\begin{align*}
S_{3} &= \sum_{Q\in \mathscr{D}} \sum_{\epsilon=1}^{M_{Q} -1} D_{Q}(w,\epsilon)^{2} \langle f \rangle_{E_{\epsilon}}^{2} \langle w \rangle _{Q}^{-1}\lesssim [w]_{A_{2}} \|fw^{-\frac{1}{2}}\|_{L^{2}(X)}^{2}.
\end{align*}
The proof of Lemma \ref{lem A2} is complete.
\end{proof}

\begin{proof}[Proof of Theorem \ref{thm square}]
Suppose  $w\in A_2$. Following the argument in the proof of Theorem 3.1 in \cite{PP}, we obtain that the Lemma \ref{lem A2} above implies that
\begin{align*}
\|f\|_{L^{2}_{w}(X)}\ls [w]_{A_2}^{1\over2}\|S_{\mathscr{D}} f\|_{L^{2}_{w}(X)},
\end{align*}
where the implicit constant is independent of $f$ and $w$.

Then following the argument in the proof of Corollary 3.2 in \cite{PP}, we obtain that
\begin{align*}
\|S_{\mathscr{D}} f\|_{L^{2}_{w}(X)}
 &\lesssim [w]_{A_{2}} \|f\|_{L^{2}_{w}(X)}.
\end{align*}
Next, by the sharp form of Rubio de Francia's extrapolation theorem  (due to Dragi\^cevi\'c, Grafakos, Pereyra and Petermichl \cite{DGPP} in the Euclidean space and due to Anderson and  Dami\'an \cite{AD} on spaces of homogeneous type), this implies the corresponding weighted $L^p$ bound
$$  \|S_\mathscr{D}f\|_{L^p_w(X)}\leq  C_p [w]_{A_p}^{\max\{1, {1\over p-1}\}}\|f\|_{L^p_{w}(X)}.   $$
The proof of Theorem \ref{thm square} is complete.
\end{proof}

\subsection{Weighted Hardy Spaces, Duality and Weak Factorisation}

We now introduce the atoms for the weighted Hardy space.
\begin{defn}\label{def: weighted atom}
Suppose $w\in A_2$. 
A function $a$ is called a $(1,2)$-atom, if
there exists a ball $B\subset X$ such that
\begin{align*}
(1)\ \ {\rm {supp}}(a)\subset B;\quad(2)\ \ \int_B a(x)\, d\mu(x)=0;
\quad(3)\ \ \|a\|_{L^2_{w}(B)}\le \left[w(B)\right]^{-{1\over 2}}.
\end{align*}
\end{defn}

\begin{defn}\label{def: weighted atom Hardy}
Suppose $w\in A_2$. 
A function $f$ is said to belong to the Hardy space $H^{1}_{w,2}(X)$, if
$f=\sum_{j=1}^\infty \lambda_j a_j$
with $\sum_{j=1}^\infty|\lambda_j|<\infty$ and $a_j$ is a $(1,2)$-atom
for each $j$. Moreover, the norm of $f$ on $H^{1}_{w,2}(X)$ is defined by
$\|f\|_{H^{1}_{w,2}(X)}=\inf\big\{\sum_{j=1}^\infty|\lambda_j|\big\},$
where the infimum is taken over all possible decompositions of $f$ as above.

\end{defn}

We then have the duality between the weighted Hardy space and weighted BMO.
We point out that for the sake of simplicity,  we obtain this result for $p=2$.  For the Euclidean version of duality of weighted Hardy and BMO spaces, we refer to
\cite[Section 4]{DHLWY} for the full range of $p\in[1,2]$.
\begin{thm}\label{HardyBMODuality}
Suppose $w\in A_2$. 
Then we have
$  \big(H^{1}_{w,2}(X)\big)' =  {\rm BMO}_{w}(X). $
\end{thm}
\begin{proof}
To prove $ {\rm BMO}_{w}(X) \subset \big(H^{1}_{w,p}(X)\big)',$
for any $g\in {\rm BMO}_{w}(X)$,
let
$$\ell_g(a)=\int_{X} a(x)g(x)d\mu(x),$$
where  $a$ is an   atom as in Definition \ref{def: weighted atom}.

Assume that $a$ is supported in a ball $B\subset X$.
Then by H\"older's inequality and $w\in A_2$, we see that
\begin{align*}
\left|\int_{X} g(x)a(x)\,d\mu(x)\right|
&=\left|\int_B [g(x)-g_B]a(x)\,d\mu(x)\right|\\
&\le \left[\int_B |g(x)-g_B|^2w^{-1}(x)\,d\mu(x)\right]^{\frac12}\left[\int_B |a(x)|^2w(x)\,d\mu(x)\right]^{\frac12}\\
&\le \left[\frac1{w(B)}\int_B |g(x)-g_B|^2w^{-1}(x)\,d\mu(x)\right]^{\frac12}\\
&\le C\|g\|_{{\rm BMO}_{w}(X)}.
\end{align*}
Thus $\ell_g$ can be extended to a bounded linear functional on $H^{1}_{w,2}(X)$.

Conversely, assume that $\ell\in  \big(H^{1}_{w,p}(X)\big)' $.
For any ball $B\subset X$, let
$$L^2_{0,\,w}(B)=\bigg\{f\in L^2_{w}(B): {\rm supp}(f)\subset B,\,\,\int_B f(x)\,d\mu(x)=0\bigg\}.$$
Then we see that for any $f\in L^2_{0,\,w}(B)$, $a:=\frac1{[w(B)]^\frac12\|f\|_{L^2_{w}(B)}}f$
is an atom  as in Definition \ref{def: weighted atom}.
This implies that
$$|\ell(a)|\le \|\ell\|\|a\|_{H^{1}_{w,2}(X)}\le \|\ell\|.$$
Moreover, we see that
$$|\ell(f)|\le \|\ell\|[w(B)]^\frac12\|f\|_{L^2_{w}(B)}.$$
From the Riesz Representation theorem, there exists $[\varphi]\in [L^2_{0,\,w}(B)]^\ast=L^2_{w^{-1}}(B)/\mathbb C$,
and $\varphi\in [\varphi]$, such that for any $f\in L^2_{0,\,w}(B)$,
$$\ell(f)=\int_B f(x)\varphi(x)\,d\mu(x)$$
and
$$\|[\varphi]\|=\inf_c\|\varphi+c\|_{L^2_{w^{-1}}(B)}\le \|\ell\|[w(B)]^\frac12.$$

Now for a fixed ball $B$, we define $B_j=2^jB$, $j\in\mathbb N$. And for $B_0$, we mean the ball $B$ itself. Then we have that for all $f\in L^2_{0,\,w}(B)$
and $j\in\mathbb N$,
$$\int_B f(x)\varphi_{B}(x)\,d\mu(x)=\int_B f(x)\varphi_{B_j}(x)\,d\mu(x).$$
It follows that for almost every $x\in B$, $\varphi_{B_j}(x)-\varphi_{B_0}(x)=C_j$
for some constant $C_j$. From this we further deduce that for all $j,\,l\in\mathbb N$,
$j\le l$ and almost every $x\in B_j$,
$$\varphi_{B_j}(x)-C_j=\varphi_{B_0}(x)=\varphi_{B_l}(x)-C_l.$$
Define
$\varphi(x)=\varphi_j(x)-C_j$
on $B_j$ for $j\in\mathbb N$. Then $\varphi$ is well defined. Moreover,
since $X=\cup_j B_j$,
by H\"older's inequality and $w\in A_p$,
we see that for any $c$ and any ball $B\subset X$,
\begin{align*}
&\left[\int_B |\varphi(x)-\varphi_B|^2w^{-1}(x)\,d\mu(x)\right]^{\frac12}=
\sup_{\|f\|_{L^2_{w}(B)}\le1}\left|\langle f, \varphi-\varphi_B\rangle\right|\\
&\quad=
\sup_{\|f\|_{L^2_{w}(B)}\le1}\left|\int_B f(x)[\varphi(x)-\varphi_B]\,d\mu(x)\right|\\
&\quad=\sup_{\|f\|_{L^2_{w}(B)}\le1}\left|\int_B [f(x)-f_B][\varphi(x)+c]\,d\mu(x)\right|\\
&\quad\le\sup_{\|f\|_{L^2_{w}(B)}\le1}\left[\|f\|_{L^2_{w}(B)}
+|f_B|[w(B)]^\frac12\right]\|[\varphi(x)+c]\chi_B\|_{L^2_{w^{-1}}(B)}\\
&\quad\le \|[\varphi(x)+c]\chi_B\|_{L^2_{w^{-1}}(B)}.
\end{align*}
Taking the infimum over $c$, we have that
$\varphi\in {\rm BMO}_{w}(X)$ and
$\|\varphi\|_{{\rm BMO}_{w}(X)}\le C \|\ell\|$.
\end{proof}

We now provide a sketch of the proof of Theorem \ref{thm5},
the details are similar to the weak factorisation result obtained in \cite{crw}.

\begin{proof}[Proof of Theorem \ref{thm5}]

Similar to \cite{crw} (see also Corollary 1.4 in \cite{HLW} and its proof), as a consequence of  the duality of weighted Hardy space $H^1_\nu(X)$ and ${\rm BMO}_{\nu}(X)$
  (Theorem \ref{HardyBMODuality} above), we see that if $f$ is of the form \eqref{represent of H1}, then $f$ is in $H^1_\nu(X)$ with the $H^1_\nu(X)$-norm control by the right-hand side of \eqref{H1norm}.
Conversely, 
based the characterisation of ${\rm BMO}_\nu(X)$ via the commutator $[b,T]$ as in Theorem \ref{thm2} (the case of $m=1$) and the linear functional analysis argument as in \cite{CLMS}, we get that every 
$f\in H^1_\nu(X)$ admits an factorisation as in \eqref{represent of H1}. Hence, 
 we obtain that Theorem \ref{thm5} holds.
\end{proof}

\section{Applications}
\label{s:Application}

The aim of this section is to show that Theorem \ref{thm3} holds for each of
the six operators listed in the introduction.

%
%
%

\subsection{Cauchy's Integral Operator}

Let $A(x)$ be a Lipschitz function on $\mathbb R$.
Consider the Lipschitz curve as $z= x+iA(x)$, $x\in(-\infty,\infty)$. 
Recall that the Cauchy integral adapted to this Lipschitz curve is:
\begin{align*}
C_{A}(f)(x)&= p.v. {1\over \pi} \int_{-\infty}^\infty {f(y)dy\over (x-y) + i(A(x)-A(y))}.
\end{align*}

The (unweighted version) commutator result was obtained in \cite{LTWW}. Here we point out that the two weight commutator and high order commutator results also hold for Cauchy's Integral Operator.
\begin{prop}
Theorem \ref{thm3} holds for the Cauchy integral operator $C_{A}$ with the underlying setting $(\R,|\cdot|,dx)$.
\end{prop}
\begin{proof}
To see this, we point out that this operator has the associated kernel
$$C_A(x,y)= {1\over \pi} {1\over (x-y) + i(A(x)-A(y))},$$
which satisfies the size condition
$$ |C_A(x,y)|\leq {1\over |x-y|}  $$
and the smoothness condition
$$ |C_A(x,y)-C_A(x,y')| + |C_A(y,x)-C_A(y',x)|  \leq 2(\|A'\|_\infty+1) {|y-y'|\over |x-y|^2} $$
for every $x,y,y'$
such that $|y-y'| \leq |x-y|/2$.
Moreover, for any interval $I:=I(x_0, r)$, we take  $y_0=x_0+4r$  
Then we see that ${\rm Re}\, C_A(x_0, y_0)$, the real part of $C_A(x_0, y_0)$,
satisfies that 
${\rm Re}\, C_A(x_0, y_0)<0$ and
$$|{\rm Re}\, C_A(x_0, y_0)|={1\over \pi}\frac{y_0-x_0}{(x_0-y_0)^2+(A(x_0)-A(y_0))^2}\gs \frac{y_0-x_0}{(\|A'\|^2_\infty+1)(x_0-y_0)^2}\gs\frac1{|I|}.$$
Therefore, \eqref{e-assump cz ker low bdd weak} holds.
As a consequence of this fact and Theorems \ref{thm1} and \ref{thm2}, we see that 
 Theorem \ref{thm3} holds.
\end{proof}

%
%
%

\subsection{The Cauchy--Szeg\"o Projection Operator on the Heisenberg Group $\mathbb H^n$}

We recall all the related definitions for the Heisenberg group in \cite[Chapter XII]{s93}. Recall that $\mathbb{H}^{n}$ is
the Lie group with underlying manifold $\mathbb{C}^{n} \times \mathbb{R}= \{[z,t]: z=(z_{1},\cdots,z_{n}) \in \mathbb{C}^{n}, t \in \mathbb{R}\}$
and multiplication law
\begin{align*}
[z,t]\circ [z', t']
&: =\Big[z_{1}+z_{1}',\cdots, z_{n}+z_{n}',
t+t'+2 \mbox{Im} \big(\sum_{j=1}^{n}z_{j} \bar{z}'_{j}\big) \Big].
\end{align*}
The identity of $\mathbb{H}^{n}$ is the origin and the inverse is given
by $[z, t]^{-1} = [-z,-t]$.
Hereafter, we  identify $\mathbb{C}^{n}$  with $\mathbb{R}^{2n}$ and  use the following notation to denote the
points  of $\mathbb{C}^{n} \times \mathbb{R} \equiv \mathbb{R}^{2n+1}$:
$g=[z,t] \equiv [x,y,t] = [x_{1},\cdots, x_{n}, y_{1},\cdots, y_{n},t]
$
with $z = [z_{1},\cdots, z_{n}]$, $z_{j} =x_{j}+iy_{j}$ and $x_{j}, y_{j}, t \in \mathbb{R}$ for $j=1,\ldots,n$. Then, the
composition law $\circ$ can be explicitly written as
$g\circ g'=[x,y,t] \circ [x',y',t'] = [x+x', y+y', t+t' + 2\langle y, x'\rangle -2\langle x, y'\rangle],
$ 
where $\langle \cdot, \cdot\rangle$ denotes the usual inner product in $\mathbb{R}^{n}$.


We recall the upper half-space $\mathcal U^n$ and its boundary $b\mathcal U^n$ as follows:
$$\mathcal U^n=\Big\{ z\in \mathbb C^{n+1}: \ {\rm Im}(z_{n+1})>\sum_{j=1}^n |z_j|^2\Big\},\ \ \ b\mathcal U^n=\Big\{ z\in \mathbb C^{n+1}: \ {\rm Im}(z_{n+1})=\sum_{j=1}^n |z_j|^2\Big\}.$$

For any function $F$ defined on $\mathcal U^n$, we write $F_\epsilon$ for its vertical translate:
$F_\epsilon(z)=F(z+\epsilon {\bf i} )$ with ${\bf i}=(0,\ldots,0,i)$.
We also recall the Hardy space $\mathcal H^2(\mathcal U^n)$, which consists of all functions $F$
holomorphic on $\mathcal U^n$ for which
$ \|F\|_{\mathcal H^2(\mathcal U^n)} = \Big( \sup_{\epsilon>0} \int_{b\mathcal U^n} |F_\epsilon(z)|^2 d\beta(z) \Big)^{1\over2}<\infty$,
where $d\beta(z)$ is the surface measure on $b\mathcal U^n$.

The Cauchy--Szeg\"o projection operator $\mathcal C$ is the orthogonal projection from $L^2(b\mathcal U^n)$ to the subspace of functions
$\{F^b\}$ that are boundary values of functions $F\in \mathcal H^2(\mathcal U^n)$. According to \cite[Section 2.3, Section 2.4, Chapter XII]{s93}, we get that for $f\in L^2(\mathbb H^n)$,
$$ \mathcal C (f)(x) = \int_{\mathbb H^n} K(x,y)f(y)dy, $$
where $K(x,y) = K(y^{-1}\circ x) $ for $x\not=y$ and
$$ K(x) = -{\partial\over \partial t} \bigg({c\over n} [t+i|\zeta|^2]^{-n}\bigg)\quad {\rm for}\ \ x=[\zeta,t]\in \mathbb H^n = \mathbb C^n\times\mathbb R, $$
and $c = 2^{n-1} i^{n+1} n! \pi^{-n-1}$.

\begin{prop}
Theorem \ref{thm3} holds for the Cauchy--Szeg\"o projection operator $\mathcal C$  with the underlying setting $(\mathbb H^n,\rho, dx)$,
where $dx$ is the usual Lebesgue measure on $\mathbb C^n\times\R$ and $\rho$ is the norm  on $\mathbb H^n$
 defined by
$\rho(x) := \max\{ |\zeta|, |t|^{1\over2} \} \quad {\rm for}\ \ x=[\zeta,t]\in \mathbb H^n = \mathbb C^n\times\mathbb R.$
\end{prop}
\begin{proof}

We begin by recalling that with this norm $\rho(x)$ as above, and we set $\rho(x,y):=\rho(y^{-1}\circ x)$. From \cite[Section 2.5, Chapter XII]{s93} we obtain that
the Cauchy--Szeg\"o kernel $K(x,y)$ satisfies the following conditions:
\begin{align*}
 &|K(x,y)|\approx \rho(x,y)^{-2n-2}\\
 &|K(x,y)-K(x,y_0)|\lesssim {\rho(y,y_0)\over \rho(x,y)} {1\over\rho(x,y)^{2n+2}}\quad {\rm whenever\ \ } \rho(x,y)\geq c\rho(y,y_0)\\
 &|K(x,y)-K(x_0,y)|\lesssim {\rho(x,x_0)\over \rho(x,y)} {1\over\rho(x,y)^{2n+2}}\quad {\rm whenever\ \ } \rho(x,y)\geq c\rho(x,x_0).
\end{align*}

Thus, it is straightforward to see that $K(x,y)$ satisfies \eqref{e-assump cz ker low bdd weak}.
Hence, we see that
Theorem \ref{thm3} holds for the Cauchy--Szeg\"o projection operator $\mathcal C$.
\end{proof}

\subsection{The Szeg\"o Projection Operator on a Family of Unbounded Weakly Pseudoconvex
Domains}

We now recall the weakly pseudoconvex domains $ \Omega_k$ and their boundary $\partial \Omega_k$, $k\in \mathbb Z_+$,  from Greiner and Stein \cite{GrSt}:
$$  \Omega_k=\Big\{ (z_1,z_2)\in \mathbb C^2: {\rm Im}(z_2) > |z_1|^{2k}\Big\},\quad     \partial\Omega_k=\Big\{ (z_1,z_2)\in \mathbb C^2: {\rm Im}(z_2) =|z_1|^{2k}\Big\}.$$
Recall that $\partial \Omega_k$ is naturally parameterized by $z_1$ and ${\rm Re} z_2$. We use the following notation. Points in $\partial \Omega_k$ are denoted by $\zeta,\omega,\nu$ etc.
\begin{align*}
&\zeta =(z_1,z_2) \sim (z,t), \quad z=z_1\in \mathbb C,\ t= {\rm Re}(z_2) \in\mathbb R;\\
&\omega =(w_1,w_2) \sim (w,s), \quad w=w_1\in \mathbb C,\ s= {\rm Re}(w_2) \in\mathbb R;\\
&\nu =(u_1,u_2) \sim (u,r), \quad u=u_1\in \mathbb C, r= {\rm Re}(u_2) \in\mathbb R.
\end{align*}

The Szeg\"o projection $S$ on $\Omega_k$ is the orthogonal projection from $L^2(\partial \Omega_k)$ to the Hardy space
$H^2(\Omega_k)$ of holomorphic functions on $\Omega_k$ with $L^2$ boundary values. The Szeg\"o kernel $S(\zeta,\omega)$
is the kernel for which
$$ S(f)(\zeta) = \int_{  \partial \Omega_k } f(\omega) S(\zeta,\omega) dV(\omega), $$
where $dV(\omega) = dV(x, y, s) = dxdyds$ with $\omega=(w_1,s) = (x+iy,s)$, which is Lebesgue measure on the parameter space $\R^3$.
Greiner and Stein \cite{GrSt} have computed the Szeg\"o kernel  with Lebesgue measure on the parameter space with the formula
\begin{align*}
S(\zeta,\omega)&={1\over 4\pi^2}\bigg[   \bigg(   \Big({i\over2}[s-t] + { |z_1|^{2k}+|w_1|^{2k} \over 2} + {\mu+\eta\over 2} \Big)^{1\over k} - z_1\bar w_1 \bigg)^2 \\
&\ \quad \times\bigg(   {i\over2}[s-t]   + { |z_1|^{2k}+|w_1|^{2k} \over 2} + {\mu+\eta\over 2}    \bigg)^{k-1\over k}  \bigg]^{-1},
\end{align*}
where $\mu= {\rm Im}(z_2)-|z_1|^{2k}$ and $\eta= {\rm Im}(w_2)-|w_1|^{2k}$.

In \cite{Diaz}, Diaz defined and analyzed a pseudometric $d(\zeta,\omega)$ globally suited to the complex geometry of $\partial \Omega_k$, which
was arrived at by the study of the Szeg\"o kernel. This allows the treatment of the Szeg\"o kernel as a singular integral kernel:
$$ d(\zeta,\omega) = \bigg|  \bigg({i\over2}[s-t] + { |z_1|^{2k}+|w_1|^{2k} \over 2}\bigg)^{1\over k}   - z\bar w  \bigg|^{1\over2}. $$

Then the pseudometric balls are defined as
$  B_\zeta(\delta)= B_\zeta^d(\delta) = \{ \omega\in\partial\Omega_k:\  d(\zeta,\omega)<\delta\}  $
and the volume of the balls is
$$ V(B_\zeta(\delta)) = 4\pi\delta^2 \bigg(    {  (\sin(\pi/ k))^{2k-2} \over4} |z|^{2k-2} \delta^2 + {1\over2} \delta^{2k} \bigg), $$
and it is shown that this measure is doubling.

\begin{prop}
Theorem \ref{thm3} holds for the Szeg\"o projection $S$  with the underlying space of homogenous type $(\partial\Omega_k,d, dV)$,
where $d$ and $dV$ are as introduced above.
\end{prop}
\begin{proof}
We point out that
it is proved in \cite{Diaz} that $S(\zeta,\omega)$ satisfies the following size and smoothness conditions:
\begin{align}
|S(\zeta,\omega)|&\approx {1\over V(B_\zeta( d(\zeta,\omega)))   }, \label{size of S in Omegak}\\
|S(\zeta,\omega) -S(\zeta',\omega)|&\lesssim  \bigg( {d(\zeta,\zeta') \over d(\zeta,\omega)} \bigg) {1\over V(B_\zeta( d(\zeta,\omega)))   },
\quad {\rm for\ \ }  cd(\zeta,\zeta') \leq d(\zeta,\omega), \nonumber\\  
|S(\zeta,\omega) -S(\zeta,\omega')|&\lesssim  \bigg( {d(\omega,\omega') \over d(\zeta,\omega)} \bigg) {1\over V(B_\zeta( d(\zeta,\omega)))   },
\quad {\rm for\ \ }  cd(\omega,\omega') \leq d(\zeta,\omega).\nonumber 
\end{align}

Thus, from \eqref{size of S in Omegak} it is direct to see that $S(\zeta,\omega)$ satisfies \eqref{e-assump cz ker low bdd weak}.
Hence, we see that
Theorem \ref{thm3} holds for the Szeg\"o projection operator $S$ on $\Omega_k$ for $k\in\Z_+$.
\end{proof}

\subsection{Riesz Transforms Associated with sub-Laplacian on Stratified Nilpotent Lie Groups}

Recall that a connected, simply connected nilpotent Lie group $\mathcal{G}$ is said to be stratified if its left-invariant Lie algebra $\mathfrak{g}$ (assumed real and of finite dimension) admits a direct sum decomposition
\begin{align*}
\mathfrak{g} = \bigoplus_{i = 1}^k V_i \  \mbox{where $[V_1, V_i] = V_{i + 1}$ for $i \leq k - 1$.}
\end{align*}
One identifies $\mathfrak{g}$ and $\mathcal{G}$ via the exponential map
$
\exp: \mathfrak{g} \longrightarrow \mathcal{G},
$
which is a diffeomorphism.
We fix once and for all a (bi-invariant) Haar measure $dg$ on $\mathcal G$ (which is just the lift of Lebesgue measure on $\frak g$ via $\exp$).
There is a natural family of dilations on $\frak g$ defined for $r > 0$ as follows:
\begin{align*}
\delta_r \left( \sum_{i = 1}^k v_i \right) = \sum_{i = 1}^k r^i v_i, \quad \mbox{with $v_i \in V_i$}.
\end{align*}
This allows the definition of dilation on $\mathcal{G}$, which we still denote by $\delta_r$.
We choose once and for all a basis $\{\X_1, \cdots, \X_n\}$ for $V_1$ and consider the sub-Laplacian $\Delta = \sum_{j=1 }^n \X_j^2 $. Observe that $\X_j$ ($1 \leq j \leq n$) is homogeneous of degree $1$ and $\Delta$ of degree $2$ with respect to the dilations in the sense that:
$
\X_j \left( f \circ \delta_r \right) = r \, \left( \X_j f \right) \circ \delta_r, \  1 \leq j \leq  n, \  r > 0, \ f \in C^1$ and that
$\delta_{\frac{1}{r}} \circ \Delta \circ \delta_r = r^2 \, \Delta, \quad \forall r > 0.
$

Let $Q$ denote the homogeneous dimension of $\mathcal{G}$, namely,
$Q= \sum_{i=1}^k i\, {\rm dim} V_i.
$
And let $p_h$ ($h > 0$)  be the heat kernel (that is, the integral kernel of $e^{h \Delta}$)
on $\mathcal G$. For convenience, we set $p_h(g) = p_h(g, o)$ (that is, in this article, for a convolution operator, we will identify the integral kernel with the convolution kernel) and $p(g) = p_1(g)$.

Recall that (c.f. for example \cite{FoSt})
$
p_h(g) = h^{-\frac{Q}{2}} p(\delta_{\frac{1}{\sqrt{h}}}(g)), \  \forall h > 0, \  g \in \mathcal G.
$
The kernel of the $j^{\mathrm{th}}$  Riesz transform $\X_j (-\Delta)^{-\frac{1}{2}}$ ($1 \leq j \leq  n$) is written simply as $K_j(g, g') = K_j(g'^{-1} \circ g)$, where
\begin{align*}
K_j(g) = \frac{1}{\sqrt{\pi}} \int_0^{+\infty} h^{-\frac{1}{2}} \X_j p_h(g) \, dh  = \frac{1}{\sqrt{\pi}} \int_0^{+\infty} h^{- \frac{Q}{2} - 1} \left( \X_j p \right)(\delta_{\frac{1}{\sqrt{h}}}(g)) \, dh.
\end{align*}

\begin{prop}
Theorem \ref{thm3} holds for the Riesz transform $\X_j (-\Delta)^{-\frac{1}{2}}$ ($1 \leq j \leq  n$)  with the underlying setting $(\mathcal G,\rho, dg)$, where $\rho$ is the homogeneous norm on $\mathcal G$ (see \cite{DLLW}).
\end{prop}

\begin{proof}

For the Riesz transform kernel, we have the following lower bound estimate, obtained in \cite{DLLW}:

 {\it
Fix $j=1,\ldots,n$. There exist $ 0<\varepsilon_o\ll1$ and $C>0$ such that  for any $0 <\eta< \varepsilon_o$ and for all  $g \in \mathcal G$ and $r > 0$, we can find $g_* = g_*(j, g,r) \in \mathcal G$ satisfying
\begin{align}
\rho(g, g_*) = r, \quad |K_j(g_1, g_2)| \geq C r^{- Q}, \quad \forall g_1 \in B(g, \eta r), \  g_2 \in B(g_*, \eta r)
\end{align}
and all $K_j(g_1,g_2)$  have the same sign.}
\medskip

From this kernel lower bound estimate, it is direct to see that
for each $j$, $K_j(g_1, g_2)$ satisfies \eqref{e-assump cz ker low bdd weak}. Hence,  Theorem \ref{thm3} holds for $\X_j (-\Delta)^{-\frac{1}{2}}$ ($1 \leq j \leq  n$).
\end{proof}

\subsection{Riesz Transform Associated with the Bessel Operator on $\R_+$}

Consider $\R_+=(0,\infty)$. For $\lambda> -{1\over2}$, the Bessel operator $\Delta_\lambda$ on $\mathbb R_+$ (\cite{MS}) is defined by
\begin{align*}
\Delta_\lambda = -{d^2\over dx^2} -{2\lambda\over x} {d\over dx}.
\end{align*}
It is a formally self-adjoint operator in $L^2(\mathbb R_+, dm_\lz)$, where $d m_\lz(x)=x^{2\lambda}dx$.
For any $x\in \mathbb{R}_+$ and $r>0$, let $I(x, r)=(x-r, x+r)\cap \mathbb{R}_+$. Moreover, we assume that $r\le x$
without loss of generality. Observe that for any $x\in \mathbb{R}_+$ and $r\in(0, x]$,
$m_\lambda(I(x,r))\sim x^{2\lambda} r.$
Thus, $(\mathbb R_+, |\cdot|, dm_\lambda)$ is a space of homogeneous type.

The Bessel Riesz transform is defined as
$ R_\lambda = {d\over dx} (\Delta_\lambda)^{-{1\over2}}.$
In \cite{MS}, Muckenhoupt--Stein introduced and obtained the $L^p(\mathbb R_+, dm_\lambda)$-boundedness of
$R_\lambda$ for $\lambda\in(0, \infty)$. Under this condition, the unweighted version of commutator theorem for $R_\lambda$
was obtained in \cite{DLWY} via weak factorisation.
However, the two weight commutator and high order commutator are unknown, and the case when $\lambda\in(-1/2,0)$ is totally unknown.
Here we will establish the two weight commutator and high order commutator for $R_\lambda$ for all $\lambda\in(-1/2,\infty)$.

\begin{prop}
Theorem \ref{thm3} holds for the Bessel Riesz transform $R_\lambda$   with the underlying setting $(\mathbb R_+, |\cdot|, dm_\lambda)$.
\end{prop}

\begin{proof}

In \cite{BHNV10}, Betancor et al. further considered $R_\lambda$
for the range $\lambda\in(-1/2, \infty)$. They showed that for $f\in C^\infty_c(\mathbb R_+)$ and $x\in(0, \infty)$,
$$R_\lambda f(x)={\rm p. v.}\int_0^\infty R_\lambda(x,y)f(y)dm_\lambda(y)$$
with the kernel
$$R_\lambda(x,y)=\frac1{\sqrt \pi}\int_0^\infty\frac{\partial}{\partial x}W^\lambda_t(x,y)\frac{dt}{\sqrt t}$$
for $x,y\in(0, \infty)$ with $x\not=y$. Here $W^\lambda_t (x,y)$ is the heat kernel associated to $\Delta_\lambda$
\begin{align}\label{e-Bessel heat ker}
W^\lambda_t(x,y)=\frac{(xy)^{-\lambda+1/2}}{2t}e^{-(x^2+y^2)/4t}
I_{\lambda-1/2}\left(\frac{xy}{2t}\right)
\end{align}
with $I_\nu$ being the modified Bessel function of the first kind and order $\nu>-1$. They also showed that $R_\lambda$
is bounded on the space $L^p(\mathbb R_+, x^\delta dx)$ if and only if $p>1$ and $-1-p<\delta<(2\lambda+1)p-1$. Moreover,
the kernel $R_\lambda(x,y)$ has the following estimates (see \cite[Lemmas 4.3 and 4.4]{BHNV10}):
\begin{itemize}
  \item [(i)] for $x/2<y<2x$ and $x\not=y$,
  \begin{align*}
R_\lambda(x,y)=\frac1{\pi}\frac{(xy)^{-\lambda}}{y-x}+\mathcal O\left(y^{-2\lambda-1}\left(1+\log\frac{xy}{(y-x)^2}\right)\right);
\end{align*}
  \item [(ii)] in the off-diagonal region,
  \begin{align*}
  |R_\lambda(x,y)|\ls \left\{
                        \begin{array}{ll}
                          x^{-2\lambda-1}, & \hbox{$y\le x/2$;} \\
                          xy^{-2\lambda-2}, & \hbox{$2x\le y$.}
                        \end{array}
                      \right.
  \end{align*}
\end{itemize}
From this fact, one deduces that $R_\lambda(x,y)$ satisfies \eqref{size of C-Z-S-I-O} and \eqref{smooth of C-Z-S-I-O} (see \cite[Theorem 2.2]{CS14}). Moreover,
there exist ${K}_1\in(0, 1/2)$  small enough, $K_2>1$ and $C_\lambda>0$ such that
\begin{itemize}
  \item [(i)] for any $x,\,y\in\mathbb{R}_+$ with $0<y/x-1<K_1$,
\begin{equation}\label{e-Bessel Riesz low bdd-diag}
 R_\lambda(x,y)\geq C_{\lambda}{1\over x^\lambda y^\lambda}{1\over y-x};
\end{equation}

  \item [(ii)] for any $x,\,y\in\mathbb{R}_+$ with $0< K_2 x\le y$,
  \begin{equation}\label{e-Bessel Riesz low bdd-offdiag}
  R_\lambda(x,y)\ge C_\lambda xy^{-2\lambda-2}.
  \end{equation}
\end{itemize}
Then an argument involving \eqref{e-Bessel Riesz low bdd-diag} and \eqref{e-Bessel Riesz low bdd-offdiag} shows
that assumption \eqref{e-assump cz ker low bdd} holds (see also \cite[Lemma 2.3]{MWY}). 
In fact, let $I:=I(x_0,r)$ with $x_0\ge r$ and $K_0:=(K_1+K_2+2)/2K_1$. We consider the following two cases.

Case (a): $x_0\le 2K_0r$. In this case, $ m_\lz(I) \sim x_0^{2\lz}r\sim {K_0} x_0^{2\lz+1}$. Let $y_0:=x_0+4K_0^2r$. Then $(2K_0+1) x_0\le y_0\le (4K_0^2+1)x_0$.
 This via \eqref{e-Bessel Riesz low bdd-offdiag} implies that
 $$R_{\lambda}(x_0, y_0)\gs \frac{x_0}{y_0^{2\lambda+2}}\sim \frac1{m_{\lambda}(I)}.$$

Case (b): $x_0>2K_0r$. In this case, $ m_\lz(I) \sim x_0^{2\lz}r$. Let $y_0:=x_0+K_2r$. Then $0<y_0/x_0-1<K_1$  and
$$R_\lambda(x_0, y_0)\gs \frac1{x_0^{2\lambda}(y_0-x_0)}\sim \frac1{m_\lambda(I)},$$
which implies that Theorem \ref{thm3} holds  for $R_\lambda$.
\end{proof}

\subsection{Riesz Transforms Associated with Bessel Operators on $\R^{n+1}_+$}

We now recall the Bessel operator and the Bessel Riesz transform in high dimension from Huber \cite{Hu}.
Consider $\R^{n+1}_+=\R^n\times (0,\infty)$. For $\lambda> -{1\over2}$,
\begin{align}\label{Dlambda}
\Delta_\lambda^{(n+1)} = -{d^2\over dx_1^2}\cdots-{d^2\over dx_n^2} -{d^2\over dx_{n+1}^2} -{2\lambda\over x_{n+1}} {d\over dx_{n+1}}.
\end{align}
The operator $\Delta_\lambda^{(n+1)}$ is symmetric and non-negative in $C^\infty_c(\mathbb R^{n+1}_+)\subset L^2(\mathbb R^{n+1}, d\mu_\lambda)$, 
where 
$$d\mu_\lambda(x):=\prod_{j=1}^n dx_j x_{n+1}^{2\lambda}dx_{n+1}.$$
The $j$-th Riesz transform is defined as
$$ R_{\lambda,j} = {d\over dx_j} (\Delta_\lambda^{(n+1)})^{-{1\over2}},\quad j=1,\ldots,n+1. $$

We point out that there is no known results for the commutator of $R_{\lambda,j}$. Here we provide an intensive study of the kernel of
$R_{\lambda,j} $, especially for the lower bound, and then we obtain the two weight commutator and higher order commutator for $R_{\lambda,j}$.

\begin{prop}\label{prop Bessel h}
Theorem \ref{thm3} holds for the Bessel Riesz transform $R_{\lambda,j} $, $j=1,\ldots,n+1$,   with the underlying setting $(\mathbb R^{n+1}_+, |\cdot|, d\mu_\lambda)$.
\end{prop}

\begin{proof}

To begin with, note that \eqref{Dlambda} can be written as
$ \Delta^{(n+1)}_\lambda = \Delta^{(n)} +  \Delta_\lambda,   $
where $\Delta^{(n)}$ denotes the standard Laplacian on $\R^n$, and $\Delta_\lambda$ denotes the Bessel operator on $\R_+$, which is one-dimensional as shown in Section 7.5. Then it is clear that
$ e^{-t\Delta^{(n+1)}_\lambda} = e^{-t(\Delta^{(n)} +  \Delta_\lambda)} $ and hence the heat kernel
$$ p_{t,\Delta^{(n+1)}_\lambda}(x,y) =p_{t,\Delta^{(n)}_\lambda}(x',y')W^\lambda_t(x_{n+1},y_{n+1})  $$
for $x=(x',x_{n+1})$, $y=(y',y_{n+1}) \in \R^n\times (0,\infty)$, where $W^\lambda_t$ is the heat kernel of $\Delta_\lambda$ as in \eqref{e-Bessel heat ker}.

Then it is direct that for $1\leq j\leq n$:
\begin{align*}
R_{\lambda,j}(x,y)&= c_{n,\lambda}\frac{\partial}{\partial x_j}  \int_0^\infty  p_{t,\Delta^{(n+1)}_\lambda}(x,y) \,\frac{dt}{\sqrt{t}}\\
&= c_{n,\lambda} \frac{\partial}{\partial x_j}  \int_0^\infty  \frac{1}{(4\pi t)^{\frac{n}{
2}}}e^{-\frac{|x'-y'|^2}{4t}} \,  W^\lambda_t(x_{n+1},y_{n+1}) \frac{dt}{\sqrt{t}}
\end{align*}
and
for $j=n+1$,
\begin{align*}
R_{\lambda,n+1}(x,y)&= c_{n,\lambda}\frac{\partial}{\partial x_{n+1}}  \int_0^\infty  p_{t,\Delta^{(n+1)}_\lambda}(x,y) \,\frac{dt}{\sqrt{t}}\\
&= c_{n,\lambda} \frac{\partial}{\partial x_{n+1}}  \int_0^\infty  \frac{1}{(4\pi t)^{\frac{n}{
2}}}e^{-\frac{|x'-y'|^2}{4t}} \,  W^\lambda_t(x_{n+1},y_{n+1}) \frac{dt}{\sqrt{t}}.
\end{align*}
 By \cite[Theorem 2.2]{CS14},  $\{R_{\lambda,j}\}_{j=1}^{n+1}$ are Calder\'on-Zygumnd operators
with kernel satisfying \eqref{size of C-Z-S-I-O}
and \eqref{smooth of C-Z-S-I-O}.

Moreover, we have the following
estimates on $\{R_{\lambda,j}(x,y)\}_{j=1}^{n+1}$.

\begin{lem}\label{p-high Riesz} Let $j\in\{1, \ldots, n\}$. The following statements hold:
\begin{itemize}
  \item [(i)] There exist positive constants $\wz C, \wz c>1$ such that for any $(x, y)$ with $0<x_{n+1}\le y_{n+1}/\wz c$, $R_{\lambda,\,j}(x,y)$ does not change sign and
  $$|R_{\lambda,\,j}(x,y)|\ge \wz C\frac{|x_j-y_j|}{(y_{n+1}^2+|x'-y'|^2)^{\frac{n}2+\lambda+1}}.$$
  \item [(ii)] There exists a positive constant $C_n$ such that for any
  $(x,y)$,
  $$R_{\lambda,\,j}(x,y)= C_n \frac{y_j-x_j}{x_{n+1}^\lambda y_{n+1}^\lambda|x-y|^{n+2}}
      +\mathcal O\left(\frac1{x^{\lambda+1}_{n+1}y^{\lambda+1}_{n+1}}
      \frac{|x_j-y_j|}{|x-y|^{n}}\right).
      $$
\end{itemize}
\end{lem}

\begin{proof}
By \eqref{e-Bessel heat ker} and 
letting $x_{n+1}=z y_{n+1}$ and $u=\frac{2t}{y_{n+1}^2}$, we see that
\begin{align*}
R_{\lambda,j}(x,y)&= c_{n,\lambda}\int_0^\infty \frac1{(4\pi t)^{\frac n2}}
e^{-\frac{|x'-y'|^2}{4t}} \frac{x_j-y_j}{-2t}\frac{(x_{n+1}y_{n+1})^{-\lambda+\frac12}}{2t}
I_{\lambda-\frac12}\left(\frac{x_{n+1}y_{n+1}}{2t}\right)e^{-\frac{x_{n+1}^2+y_{n+1}^2}{4t}} \frac{dt}{\sqrt{t}}  \\
&=c_{n,\lambda}\int_0^\infty \frac1{(4\pi t)^{\frac n2}}
e^{-\frac{|x'-y'|^2}{4t}} \frac{x_j-y_j}{-2t}\left(\frac{y^2_{n+1}z}{2t}\right)^{-\lambda+\frac12}\\
&\quad \times I_{\lambda-\frac12}\left(\frac{y^2_{n+1} z}{2t}\right)(2t)^{-\lambda}e^{-\frac{y_{n+1}^2(1+z^2)}{4t}} \frac{dt}{\sqrt{t}} \\
&=c_{n,\lambda}(y_j-x_j)\int_0^\infty (uy^2_{n+1})^{-\frac n2-\lambda-1}e^{-\frac{|x'-y'|^2}{2uy^2_{n+1}}} e^{-\frac{1+z^2}{2u}} \left(\frac zu\right)^{-\lambda +\frac12}I_{\lambda-\frac12}\left(\frac zu\right)\frac{du}{u}.
\end{align*}
Recall that for any $z>0$ and $\nu>-1$,
\begin{align}\label{e-Bessel func-1}
\lim_{z\to 0^+} z^{-\nu} I_\nu(z)=\frac1{2^\nu\Gamma(\nu+1)},
\end{align}
and for any $z>0$, $\nu>-1$ and $n=0, 1, 2, \ldots$,
\begin{align}\label{e-Bessel func-2}
I_\nu(z)=\frac{e^z}{\sqrt{2\pi z}}\left(\sum_{k=0}^n(-1)^k[\nu, k](2z)^{-k}+\mathcal O(z^{-n-1})\right),
\end{align}
where $[\nu, 0]:=1$ and for $k\in\mathbb N$,
$$[\nu, k]:=\frac{(4\nu^2-1)(4\nu^2-3^2)\cdots(4\nu^2-(2k-1)^2)}{2^{2k}\Gamma(k+1)};$$
see \cite[p.\,109]{BHNV10}.

Then letting $z\to 0$ and applying
the Lebesgue Dominated Convergence Theorem, we have
\begin{align*}
R_{\lambda,j}(x,y)&\to c_{n,\lambda}(x_j-y_j)y_{n+1}^{-n-2\lambda-2}\int_0^\infty u^{-\frac n2-\lambda-2} e^{-\frac1{2u}(\frac{|x'-y'|^2}{y_{n+1}^2}+1)}du\\
&=c\frac{x_j-y_j}{(y^2_{n+1}+|x'-y'|^2)^{\frac n2+1+\lambda}}.
\end{align*}
This shows (i).


Now for $j=1, 2, \ldots, n$, let
$$R_j(x, y):= c_{n, \lambda}\int_0^\infty \frac1{(4\pi t)^\frac n2}e^{-\frac{|x'-y'|^2}{4t}}\frac{x_j-y_j}{-2t} \frac1{x^\lambda_{n+1}y^\lambda_{n+1}}
W_t(x_{n+1}, y_{n+1})\frac{dt}{\sqrt{t}},$$
where
$$W_t(x_{n+1}, y_{n+1}):=\frac1{\sqrt{4\pi t}}e^{-\frac{(x_{n+1}-y_{n+1})^2}{4t}}.$$
Then we see that for any $x, y\in \mathbb R^n$ such that $x_j\not= y_j$,
$$R_j(x, y)= C_n \frac{y_j-x_j}{x_{n+1}^\lambda y_{n+1}^\lambda|x-y|^{n+2}}.$$

On the other hand, by using \eqref{e-Bessel func-2} for $n=0$,
we conclude that
\begin{align*}
&|R_{\lambda,\, j}(x, y)-R_j(x, y)|\\
&= c_{n,\lambda}\bigg| \int_0^\infty \frac1{(4\pi t)^{\frac n2}}
e^{-\frac{|x'-y'|^2}{4t}} \frac{x_j-y_j}{-2t} \\
&\qquad\qquad\times \bigg[\frac{(x_{n+1}y_{n+1})^{-\lambda+\frac12}}{2t}
   {e^{x_{n+1}y_{n+1} \over2t}   \over   \sqrt{ 2\pi\cdot {x_{n+1}y_{n+1} \over2t}  }  } \Big(  1+ C {2t \over x_{n+1}y_{n+1}}  \Big)      e^{-\frac{x_{n+1}^2+y_{n+1}^2}{4t}} \\
&\quad\hskip3cm- \frac1{x^\lambda_{n+1}y^\lambda_{n+1}}
\frac1{\sqrt{4\pi t}}e^{-\frac{(x_{n+1}-y_{n+1})^2}{4t}} \bigg]\frac{dt}{\sqrt{t}}\bigg| \\
&\ls \int_0^\infty \frac1{(4\pi t)^{\frac n2}} e^{-\frac{|x-y|^2}{4t}}\frac{|x_j-y_j|}{x^{\lambda+1}_{n+1}y^{\lambda+1}_{n+1}}\frac{dt}{t}\\
&\ls \frac{|x_j-y_j|}{x^{\lambda+1}_{n+1}y^{\lambda+1}_{n+1}}\int_0^\infty \frac1{t^{\frac{n}2+1}} e^{-\frac{|x-y|^2}{4t}}dt\\
&\ls\frac{|x_j-y_j|}{x^{\lambda+1}_{n+1}y^{\lambda+1}_{n+1}}\frac1{|x-y|^{n}}.
\end{align*}
The proof of Lemma \ref{p-high Riesz} is complete.
\end{proof}

Similarly, for $R_{\lambda,\,n+1}(x,y)$, we have  show the following lemma.

\begin{lem}\label{p-high Riesz-2}
The following statements hold:
\begin{itemize}
 \item [(i)] There exist positive constants $\wz C, \wz c\ge1$ such that for any $(x, y)$ with $0<x_{n+1}\le y_{n+1}/\wz c$ and $\frac{\lambda+\frac12}{\lambda+1}<\frac{y_{n+1}^2}{y_{n+1}^2+|x'-y'|^2}<1$, $R_{\lambda,\,n+1}(x,y)$ does not change sign and
  $$|R_{\lambda,\,n+1}(x,y)|\ge \wz C\frac{x_{n+1}}{(y_{n+1}^2+|x'-y'|^2)^{\frac{n}2+\lambda+1}}.$$
  \item [(ii)] There exists a positive constant $C_n$ such that for any
  $(x,y)$,
  $$R_{\lambda,\,n+1}(x,y)= C_n\frac{y_{n+1}-x_{n+1}}{(x_{n+1}y_{n+1})^\lambda}\frac1{|x-y|^{n+2}}
      +\mathcal O\left(\frac1{x^{\lambda}_{n+1}y^{\lambda+1}_{n+1}}\frac1{|x-y|^n}\right).
      $$
\end{itemize}
\end{lem}
\begin{proof}
 Observe that
\begin{align*}
\frac{\partial}{\partial x_{n+1}} W^\lambda_t(x_{n+1}, y_{n+1}) &=\frac1{(2t)^{\lambda +\frac12}}\left[x_{n+1}\left(\frac{y_{n+1}}{2t}\right)^2 \left(\frac{x_{n+1}y_{n+1}}{2t}\right)^{-\lambda-\frac12}I_{\lambda+\frac12}
\left(\frac{x_{n+1}y_{n+1}}{2t}\right)\right.\\
&\left.\quad-\frac{x_{n+1}}{2t}\left(\frac{x_{n+1}y_{n+1}}{2t}\right)^{-\lambda+\frac12}
I_{\lambda-\frac12}\left(\frac{x_{n+1}y_{n+1}}{2t}\right)\right] e^{-\frac{x^2_{n+1}+y^2_{n+1}}{4t}};
\end{align*}
see, \cite{BHNV10}. Then we have
\begin{align*}
R_{\lambda,\,n+1}(x,y)&=c_{n,\lambda}\int_0^\infty \frac1{(4\pi t)^{\frac n2}}
e^{-\frac{|x'-y'|^2}{4t}}\frac1{(2t)^{\lambda +\frac12}}
x_{n+1}\left(\frac{y_{n+1}}{2t}\right)^2 \left(\frac{x_{n+1}y_{n+1}}{2t}\right)^{-\lambda-\frac12}\\
&\quad\quad\times I_{\lambda+\frac12}
\left(\frac{x_{n+1}y_{n+1}}{2t}\right)
e^{-\frac{x^2_{n+1}+y^2_{n+1}}{4t}}\frac{dt}{\sqrt t}\\
&\quad-c_{n,\lambda}\int_0^\infty \frac1{(4\pi t)^{\frac n2}}
e^{-\frac{|x'-y'|^2}{4t}}\frac1{(2t)^{\lambda +\frac12}}
\frac{x_{n+1}}{2t}\left(\frac{x_{n+1}y_{n+1}}{2t}\right)^{-\lambda+\frac12}\\
&\quad\quad\times I_{\lambda-\frac12}\left(\frac{x_{n+1}y_{n+1}}{2t}\right)
e^{-\frac{x^2_{n+1}+y^2_{n+1}}{4t}}\frac{dt}{\sqrt t}.
\end{align*}

By change of variables, we have
\begin{align*}
R_{\lambda,\,n+1}(x,y)&=c_{n,\lambda}\bigg[x_{n+1}\int_0^\infty \frac1{(4\pi t)^{\frac n2}}
e^{-\frac{|x'-y'|^2}{4t}}\frac1{(2t)^{\lambda +\frac12}}
\left(\frac{y_{n+1}}{2t}\right)^2 \left(\frac{y^2_{n+1}z}{2t}\right)^{-\lambda-\frac12}\\
&\quad\quad\times I_{\lambda+\frac12}\left(\frac{y^2_{n+1}z}{2t}\right)
e^{-\frac{y^2_{n+1}(1+z^2)}{4t}}\frac{dt}{\sqrt t}\\
&\quad -x_{n+1}\int_0^\infty \frac1{(4\pi t)^{\frac n2}}
e^{-\frac{|x'-y'|^2}{4t}}\frac1{(2t)^{\lambda +\frac12}}
\frac{1}{2t}\left(\frac{y_{n+1}^2z}{2t}\right)^{-\lambda+\frac12}\\
&\quad\quad\times I_{\lambda-\frac12}\left(\frac{y_{n+1}^2z}{2t}\right)
e^{-\frac{y^2_{n+1}(1+z^2)}{4t}}\frac{dt}{\sqrt t} \bigg]\\
&= c_{n,\,\lambda}\left[\frac{x_{n+1}}{y_{n+1}^{n+2+2\lambda}}\int_0^\infty \frac1{u^{\frac n2+\lambda+2}}e^{-\frac{|x'-y'|^2}{2u y^2_{n+1}}}\left(\frac zu\right)^{-\lambda -\frac12}I_{\lambda+\frac12}\left(\frac zu\right) e^{-\frac{1+z^2}{2u}}\frac{du}{u}\right.\\
&\quad\quad-\left.\frac{x_{n+1}}{y_{n+1}^{n+2\lambda+2}}\int_0^\infty \frac1{u^{\frac n2+\lambda+1}}e^{-\frac{|x'-y'|^2}{2uy_{n+1}^2}}
\left(\frac zu\right)^{-\lambda+\frac12}I_{\lambda-\frac12}\left(\frac zu\right)
e^{-\frac{1+z^2}{2u}}\frac{du}{u}\right].
\end{align*}
By letting $z\to0$ and applying \eqref{e-Bessel func-1}, we see that
\begin{align*}
R_{\lambda,\,n+1}(x,y)\frac{y_{n+1}^{n+2\lambda+2}}{x_{n+1}}&\to
C_{n,\,\,\lambda}\left[\frac1{2^{\lambda+\frac12}\Gamma(\lambda+\frac32)}\int_0^\infty \frac1{u^{  \frac n2+\lambda+2}}e^{-\frac1{2u}(1+\frac{|x'-y'|^2}{y_{n+1}^2})}\frac{du}{u}
\right.\\
&\quad-\left.\frac1{2^{\lambda-\frac12}\Gamma(\lambda+\frac12)}\int_0^\infty \frac1{u^{\frac n2+\lambda+1}}e^{-\frac1{2u}(1+\frac{|x'-y'|^2}{y_{n+1}^2})}\frac{du}{u}
\right]\\
&=C_{n,\,\,\lambda}2^{\frac{n+3}2}
\frac{y_{n+1}^{n+2\lambda+2}}{(y_{n+1}^2+|x'-y'|^2)^{\frac n2+\lambda+1}}\frac{\Gamma(\lambda+\frac n2-1)}{\Gamma(\lambda+\frac12)}\\
&\times\left(\frac{y_{n+1}^2}{y_{n+1}^2+|x'-y'|^2}\frac{\lambda+\frac n2-1}{\lambda+\frac12}-1\right)\\
&=C_{n,\,\lambda}\frac{y_{n+1}^{n+2\lambda+2}}{(y_{n+1}^2+|x'-y'|^2)^{\frac n2+\lambda+1}}\left(\frac{y_{n+1}^2}{y_{n+1}^2+|x'-y'|^2}\frac{\lambda+\frac n2+1}{\lambda+\frac12}-1\right).
\end{align*}
Since 
$$\frac{y_{n+1}^2}{y_{n+1}^2+|x'-y'|^2}\frac{\lambda+\frac n2+1}{\lambda+\frac12}-1>0,$$
the conclusion (i) holds.

Let
\begin{align*}
{\rm H}(x,y)&:=c_{n,\lambda}\int_0^{\infty}
\frac1{(4\pi t)^{\frac n2}}
e^{-\frac{|x-y|^2}{4t}}\frac1{(2t)^{\lambda +\frac12}}\left(\frac{x_{n+1}y_{n+1}}{2t}\right)^{-\lambda}\\
&\quad\times \left[x_{n+1}\left(\frac{y_{n+1}}{2t}\right)^2\left(\frac{x_{n+1}y_{n+1}}{2t}\right)^{-1}
-\frac{x_{n+1}}{2t}\right]\frac{dt}{\sqrt t}.
\end{align*}
Then observe that
\begin{align*}
&\int_0^{\infty}
\frac1{(4\pi t)^{\frac n2}}
e^{-\frac{|x-y|^2}{4t}}\frac1{(2t)^{\lambda +\frac12}}
\left(\frac{x_{n+1}y_{n+1}}{2t}\right)^{-\lambda} \left[x_{n+1}\left(\frac{y_{n+1}}{2t}\right)^2\left(\frac{x_{n+1}y_{n+1}}{2t}\right)^{-1}
-\frac{x_{n+1}}{2t}\right]\frac{dt}{\sqrt t}\\
&\quad= C_n \frac{y_{n+1}-x_{n+1}}{(x_{n+1}y_{n+1})^\lambda}\int_0^\infty \frac1{t^{\frac n2+2}} e^{-\frac{|x-y|^2}{4t}}dt\\
&\quad= C_n\frac{y_{n+1}-x_{n+1}}{(x_{n+1}y_{n+1})^\lambda}\frac1{|x-y|^{n+2}}.
\end{align*}
By using \eqref{e-Bessel func-2} for $n=0$,
\begin{align*}
\left|R_{\lambda,\,n+1}(x, y)-{\rm H}(x,y)\right|& \ls
\int_0^{\infty}\frac1{t^{\frac n2+\lambda+1}}e^{-\frac{|x-y|^2}{4t}}
\left(\frac{x_{n+1}y_{n+1}}{2t}\right)^{-\lambda}\frac1{y_{n+1}}dt\ls \frac1{x^{\lambda}_{n+1}y^{\lambda+1}_{n+1}}\frac1{|x-y|^n}.
\end{align*}
We then see that (ii) holds.
%
The proof of Lemma \ref{p-high Riesz-2} is complete.
\end{proof}

We now return to the proof of Proposition \ref{prop Bessel h}.
Based on Lemmas \ref{p-high Riesz} and \ref{p-high Riesz-2},
we see that \eqref{e-assump cz ker low bdd} holds. Indeed, 
For any $x:=(x_1, \ldots, x_{n+1})\in \mathbb R^{n+1}_+$ and $r\in(0, \infty)$, let $$Q(x, r):=\{y:=(y_1,\ldots, y_{n+1})\in \mathbb R^{n+1}_+:\, |x_j-y_j|\le r/2,
j\in\{1, \ldots, n+1\}\}.$$
Then we have $\mu_\lambda(Q(x, r))\sim r^{n+1} x_{n+1}^{2\lambda}$. Let
$C_0\gg \tilde c$. For any $x :=(x _1, \ldots, x _{n+1})
\in \mathbb R^{n+1}_+$
and $r\in (0, \infty)$, if $r>\frac{\tilde c-1}{C_0}x _{n+1}$,
take
$y:=(y_1, \ldots, y_{n+1})$ such that $y_i=x_i+C_0r$ for $i=j$ or $i=n+1$ and $y_i:=x_i$ otherwise. Then by Lemma \ref{p-high Riesz}(i),
we see that $y_{n+1}\ge \tilde c x_{n+1}$ and
$$|R_{\lambda,\,j}(x,y)|\gs \frac{C_0r}{[(y_{n+1}^2+(C_0r)^2]^{\frac n2+\lambda+1}}\gs
\frac1{\mu_\lambda(Q(x, r))}.$$
If $r\le\frac{\tilde c-1}{C_0}x _{n+1}$, then there exists $y\in \mathbb R^{n+1}_+$
such that
$|y_{n+1}-x_{n+1}|=|y_j-x_j|\sim |y-x|$ and
$|y_{n+1}-x_{n+1}|\ll x_{n+1}$. Then by
Lemma \ref{p-high Riesz}(ii), we also have
$$|R_{\lambda,\,j}(x, y)|\gs \frac{|y_j-x_j|}{x_{n+1}^\lambda y_{n+1}^\lambda|x-y|^{n+2}}\sim \frac1{\mu_\lambda(Q(x, r))}.$$
Therefore, \eqref{e-assump cz ker low bdd weak} holds for $R_{\lambda,\,j}(x,y)$,
$j\in\{1,\ldots, n\}$. The argument for $R_{\lambda,\,n+1}(x, y)$
is similar and omitted. Then Theorem
 \ref{thm3} holds for the Bessel Riesz transform $R_{\lambda,j} $, $j=1,\ldots,n+1$. 

The proof of Proposition \ref{prop Bessel h} is complete.
\end{proof}

\section{A Digression to Product Setting: Little bmo Space}
\label{s:MainResult 3}

In this section we consider the weighted little bmo space on product spaces of homogeneous type.
To begin with, let $(X_1,d_1,\mu_1)$ and $(X_2,d_2,\mu_2)$ be two copies of spaces of homogeneous type as stated in Section 2, and denote
$\vec{X}:=X_1\times X_2$, $\vec\mu:=\mu_1\times\mu_2$. Moreover, for the points in $\vec X$, we denote $\vec x := (x_1,x_2)\in \vec X$.

Mac\'ias and
Segovia~\cite{MS1} proved the following fundamental result on spaces
of homogeneous type.
Suppose that $(X,d)$ is a space endowed
with a quasi-metric $d$
that may have no regularity.
    Then there exists a quasi-metric $d'$  {that is pointwise equivalent to $d$}
    such that $d(x,y)\sim d'(x,y)$ for all $x,y\in X$
    and there exist constants $\theta\in(0,1)$ and $C > 0$ so that
    $d'$ has the following regularity:
    \begin{eqnarray*} 
        |d'(x,y) - d'(x',y)|
        \le C \, d'(x,x')^\theta \,
            [d'(x,y) + d'(x',y)]^{1 - \theta}.
    \end{eqnarray*}
    for all $x$, $x'$, $y\in X$. Moreover, if the quasi-metric
    balls are defined by this new quasi-metric~$d'$, that is,
    $B'(x,r) := \{y\in X: d'(x,y) < r\}$ for $r > 0$, then
    these balls are open in the topology induced by $d'$. See
    \cite[Theorem 2, p.259]{MS1}.
So, without lost of generality, we assume that in our product setting, the quasi-metrics $d_1$ and $d_2$ have
regularity with constants $\theta_1$ and $\theta_2$, respectively.

We now recall the product $A_p(\vec{X})$ weights on product spaces of homogeneous type.
\begin{defn}
  \label{def:Ap product}
  Let $w(x_1,x_2)$ be a nonnegative locally integrable function
  on~$\vec{X}$. For $1 < p < \infty$, we
  say $w$ is a product $A_p$ \emph{weight}, written as $w\in
  A_p(\vec{X})$, if
  \[
    [w]_{A_p(\vec{X})}
    := \sup_R \left(\intav_R w\right)
    \left(\intav_R
      \left(\dfrac{1}{w}\right)^{1/(p-1)}\right)^{p-1}
    < \infty.
  \]
  Here the supremum is taken over all ``rectangles''~$R:= B_1\times B_2\subset \vec{X}$, where $B_i$ are balls in $X_i$ for $i=1,2$.
  The quantity $[w]_{A_p(\vec{X})}$ is called the \emph{$A_p$~constant
  of~$w$}.
\end{defn}

Next we recall the weighted little bmo space  on product spaces of homogeneous type.
\begin{defn}
  \label{def:bmo}
 For $1 < p < \infty$ and $w\in A_p(\vec{X})$, the weighted little bmo space ${\rm bmo}_w(\vec{X})$ is the space of all locally integrable functions $b$ on $\vec{X}$ such that
 $$ \|b\|_{{\rm bmo}_w(\vec{X})}= \sup_R {1\over w(R)}\int_{R}|b(\vec x) - b_R|d\vec \mu(\vec x)<\infty, $$
where  the supremum is taken over all ``rectangles''~$R= B_1\times B_2\subset \vec{X}$, where $B_i$ are balls in $X_i$ for $i=1,2$.
\end{defn}


Similar to \cite[Section 7.1]{HPW}, we introduce the bi-parameter Journ\'e operator on $\vec{X}$ as follows.  Let $C_0^{\eta_1}(X_1)$, $\eta_1\in(0,\theta_1]$, denote the space of continuous functions $f$ with bounded support such that
$$ \|f\|_{C_0^{\eta_1}(X_1)}:= \sup_{x,y\in X_1,x\not=y} {|f(x)-f(y)|\over d_1(x,y)^{\eta_1}}<\infty. $$
Let $C_0^{\eta_2}(X_2)$, $\eta_2\in(0,\theta_2]$ be defined similarly.

\noindent {\bf I.}  Structural Assumptions: Given $f=f_1\otimes f_2$ and $g=g_1\otimes g_2$, where $f_i,g_i: X_i\to \mathbb C$, $f_i,g_i\in C_0^{\eta_i}(X_i)$ satisfy
${\rm supp}\, f_i\cap {\rm supp}\, g_i=\emptyset$ for $i=1,2$, we  assume the kernel representation
$$ \langle Tf,g \rangle = \int_{\vec X}\int_{\vec X} K(\vec x,\vec y) f(\vec y) g(\vec x) d\vec \mu(\vec y) d\vec \mu(\vec x). $$
The kernel $K: \vec X\times \vec X\backslash\{ (\vec x,\vec y)\in \vec X\times \vec X:\ x_1=y_1, {\rm\ or\ } \ x_2=y_2 \}\to\mathbb C$ is assumed to satisfy:

\noindent 1.  Size condition:
$$  |K(\vec x,\vec y)|\leq C {1\over \mu_1(B(x_1,d_1(x_1,y_1)))  \mu_2(B(x_2,d_2(x_2,y_2)))  }. $$
\noindent 2.  H\"older conditions:

2a. if $d_1(y_1,y'_1)\leq {1\over 2A_0}d_1(x_1,y_1)$ and  $d_2(y_2,y'_2)\leq {1\over 2A_0}d_2(x_2,y_2)$:
\begin{align*}
 &|K(\vec x,\vec y) - K(\vec x, (y_1,y'_2)) -K(\vec x, (y'_1,y_2)) +K(\vec x, \vec y'))    |\\
 &\leq C {d_1(y_1,y'_1)^\delta d_2(y_2,y'_2)^\delta \over \mu_1(B(x_1,d_1(x_1,y_1))) d_1(x_1,y_1)^\delta  \mu_2(B(x_2,d_2(x_2,y_2)))d_2(x_2,y_2)^\delta  }.
\end{align*}

2b. if $d_1(x_1,x'_1)\leq {1\over 2A_0}d_1(x_1,y_1)$ and  $d_2(x_2,x'_2)\leq {1\over 2A_0}d_2(x_2,y_2)$:
\begin{align*}
 &|K(\vec x,\vec y) - K((x_1,x'_2),\vec y) - K((x'_1,x_2),\vec y) +K(\vec x', \vec y))    |\\
 &\leq C {d_1(x_1,x'_1)^\delta d_2(x_2,x'_2)^\delta \over \mu_1(B(x_1,d_1(x_1,y_1))) d_1(x_1,y_1)^\delta  \mu_2(B(x_2,d_2(x_2,y_2)))d_2(x_2,y_2)^\delta  }.
\end{align*}

2c. if $d_1(y_1,y'_1)\leq {1\over 2A_0}d_1(x_1,y_1)$ and  $d_2(x_2,x'_2)\leq {1\over 2A_0}d_2(x_2,y_2)$:
\begin{align*}
 &|K(\vec x,\vec y) - K((x_1,x'_2),\vec y) - K(\vec x, (y'_1,y_2))  +K( (x_1,x'_2), (y'_1,y_2))    |\\
 &\leq C {d_1(y_1,y'_1)^\delta d_2(x_2,x'_2)^\delta \over \mu_1(B(x_1,d_1(x_1,y_1))) d_1(x_1,y_1)^\delta  \mu_2(B(x_2,d_2(x_2,y_2)))d_2(x_2,y_2)^\delta  }.
\end{align*}

2d. if $d_1(x_1,x'_1)\leq {1\over 2A_0}d_1(x_1,y_1)$ and  $d_2(y_2,y'_2)\leq {1\over 2A_0}d_2(x_2,y_2)$:
\begin{align*}
 &|K(\vec x,\vec y)  - K(\vec x, (y_1,y'_2)) - K((x'_1,x_2),\vec y)   +K( (x'_1,x_2), (y_1,y'_2))    |\\
 &\leq C {d_1(x_1,x'_1)^\delta d_2(y_2,y'_2)^\delta \over \mu_1(B(x_1,d_1(x_1,y_1))) d_1(x_1,y_1)^\delta  \mu_2(B(x_2,d_2(x_2,y_2)))d_2(x_2,y_2)^\delta  }.
\end{align*}
\noindent 3. Mixed size and H\"older conditions:

3a. if $d_1(x_1,x'_1)\leq {1\over 2A_0}d_1(x_1,y_1)$:
\begin{align*}
 |K(\vec x,\vec y) - K( (x'_1,x_2), \vec y))    |\leq C {d_1(x_1,x'_1)^\delta  \over \mu_1(B(x_1,d_1(x_1,y_1))) d_1(x_1,y_1)^\delta  \mu_2(B(x_2,d_2(x_2,y_2)))   }.
\end{align*}

3b. if $d_1(y_1,y'_1)\leq {1\over 2A_0}d_1(x_1,y_1)$:
\begin{align*}
 |K(\vec x,\vec y) - K(  \vec x, (y'_1,y_2) ))    |\leq C {d_1(y_1,y'_1)^\delta  \over \mu_1(B(x_1,d_1(x_1,y_1))) d_1(x_1,y_1)^\delta  \mu_2(B(x_2,d_2(x_2,y_2)))   }.
\end{align*}

3c. if $d_2(x_2,x'_2)\leq {1\over 2A_0}d_2(x_2,y_2)$:
\begin{align*}
 |K(\vec x,\vec y) - K( (x_1,x'_2), \vec y))    |\leq C {d_2(x_2,x'_2)^\delta  \over \mu_1(B(x_1,d_1(x_1,y_1)))   \mu_2(B(x_2,d_2(x_2,y_2)))  d_2(x_2,y_2)^\delta }.
\end{align*}

3d. if $d_2(y_2,y'_2)\leq {1\over 2A_0}d_2(x_2,y_2)$:
\begin{align*}
 |K(\vec x,\vec y) - K(  \vec x, (y_1,y'_2) ))    |\leq C {d_2(y_2,y'_2)^\delta  \over \mu_1(B(x_1,d_1(x_1,y_1)))   \mu_2(B(x_2,d_2(x_2,y_2)))  d_2(x_2,y_2)^\delta }.
\end{align*}

\noindent 4. Calder\'on--Zygmund structure in $X_1$ and $X_2$ separately: If $f = f_1\otimes f_2$ and $g = g_1\otimes g_2$ with ${\rm supp} f_1\cap {\rm supp} g_1=\emptyset$, we assume the kernel representation:
$$ \langle Tf,g \rangle=\int_{X_1}\int_{X_1} K_{f_2,g_2}(x_1,y_1)f_1(y_1)g_1(x_1)d\mu_1(x_1)d\mu_1(y_1),  $$
where the kernel $K_{f_2,g_2}:\ X_1\times X_1\backslash \{ (x_1,y_1)\in X_1\times X_1: x_1=y_1 \}$ satisfies the following size condition:
$$ |K_{f_2,g_2}(x_1,y_1)|\leq C(f_2,g_2) {1\over  \mu_1(B(x_1,d_1(x_1,y_1))) }  $$
and H\"older conditions:
$$ |K_{f_2,g_2}(x_1,y_1)-K_{f_2,g_2}(x'_1,y_1)|\leq   {C(f_2,g_2)\ d_1(x_1,x'_1)^\delta \over  \mu_1(B(x_1,d_1(x_1,y_1))) d_1(x_1,y_1)^\delta }, {\rm \ }
   d_1(x_1,x'_1)\leq {1\over 2A_0}d_1(x_1,y_1), $$
$$ |K_{f_2,g_2}(x_1,y_1)-K_{f_2,g_2}(x_1,y'_1)|\leq   {C(f_2,g_2)\ d_1(y_1,y'_1)^\delta \over  \mu_1(B(x_1,d_1(x_1,y_1))) d_1(x_1,y_1)^\delta }, {\rm \ }
   d_1(x_1,x'_1)\leq {1\over 2A_0}d_1(x_1,y_1). $$
We only assume the above representation and a certain control over $C(f_2,g_2)$ on the diagonal, that is:
$$  C(\chi_{Q_2},\chi_{Q_2}) +C(\chi_{Q_2},u_{Q_2}) +C(u_{Q_2},\chi_{Q_2}) \leq C\mu_2(Q_2)   $$
for all cubes $Q_2\subset X_2$ and all ``$Q_2$-adapted zero-mean'' functions $u_{Q_2}$-- that is, ${\rm supp }\,u_{Q_2} \subset Q_2$,
$|u_{Q_2}|\leq1$ and $\int_{X_2} u_{Q_2}(x_2)d\mu_2(x_2)=0$. We assume the symmetrical representation with kernel $K_{f_1,g_1}$
in the case ${\rm supp} f_2\cap {\rm supp} g_2 =\emptyset$.

\noindent {\bf II.} Boundedness and Cancellation Assumptions:

\noindent 1. Assume $T1$, $T^*1$,  $T_11$,  and $T_1^*1$ are in product ${\rm BMO}(\vec X)$, where $T_1$ is the partial adjoint of $T$ defined by $ \langle T_1(f_1\otimes f_2), g_1\otimes g_2 \rangle =  \langle T_1(g_1\otimes f_2), (f_1\otimes g_2) \rangle $.

\noindent 2.  Assume $| \langle T_1(\chi_{Q_1}\otimes \chi_{Q_2}), \chi_{Q_1}\otimes \chi_{Q_2} \rangle |\leq C\mu_1(Q_1)\mu_2(Q_2) $ for all cubes $Q_i\in X_i$ (weak boundedness).

\noindent 3. Diagonal BMO conditions: for all cubes $Q_i\subset X_i$ and all non-zero functions $a_{Q_1}$ and $b_{Q_2}$ that are
$Q_1-$ and $Q_2-$ adapted, respectively, assume:

\noindent $| \langle T_1(a_{Q_1}\otimes \chi_{Q_2}), \chi_{Q_1}\otimes \chi_{Q_2} \rangle |\leq C\mu_1(Q_1)\mu_2(Q_2) $,\ \
$| \langle T_1(\chi_{Q_1}\otimes \chi_{Q_2}), a_{Q_1}\otimes \chi_{Q_2} \rangle |\leq C\mu_1(Q_1)\mu_2(Q_2) $,

\noindent $| \langle T_1(\chi_{Q_1}\otimes b_{Q_2}), \chi_{Q_1}\otimes \chi_{Q_2} \rangle |\leq C\mu_1(Q_1)\mu_2(Q_2) $,\ \
$| \langle T_1(\chi_{Q_1}\otimes \chi_{Q_2}), \chi_{Q_1}\otimes b_{Q_2} \rangle |\leq C\mu_1(Q_1)\mu_2(Q_2) $.

For the upper bound of the commutator of such operators $T$ and $b\in {\rm bmo}_w(\vec{X})$, following the same approach as that in \cite{HPW}, and combining all necessary tools as recalled in Section 2 on spaces of homogeneous type (such as the adjacent dyadic system, Haar basis, et al), we obtain that
\begin{thm}
Let $1 < p < \infty$ and $\lambda_1,\lambda_2\in A_p(\vec{X})$, and define
$\nu= \lambda_1^{1\over p} \lambda_2^{-{1\over p}}$. Let $T$ be a bi-parameter Journ\'e operator on $\vec{X}$ and $b\in {\rm bmo}_\nu(\vec{X})$.  Then we obtain that
$$ \| [b,T]:\ L^p_{\lambda_1}(\vec{X})\to L^p_{\lambda_2}(\vec{X}) \|\lesssim \|b\|_{{\rm bmo}_\nu(\vec{X})}. $$
\end{thm}

We now provide a broader version of the lower bound. Note that in \cite{HPW} the authors only considered the lower bound of commutator with respect to double Riesz transforms, and their proof relies on Fourier transform and hence can not be adapted to spaces of homogeneous type.

We assume that  the bi-parameter Journ\'e operator $T$  satisfies the following ``homogeneous''
condition:

\noindent {\it there exist positive constants $c_0$ and $\overline C$ such that for every $x_1\in X_1$, $x_2\in X_2$ and $r_1,r_2>0$, there exist $y_1\in B_1(x_1, \overline C r_1)\backslash B_1(x_1,r_1)$ and $y_2\in B_2(x_2, \overline C r_2)\backslash B_2(x_2,r_2)$ satisfying
\begin{equation}\label{e-assump cz ker low bdd weak product}
|K(x_1, y_1;x_2,y_2)|\geq \frac1{c_0\mu_1(B_1(x_1,r_1))\mu_2(B_2(x_2,r_2))}.
\end{equation}
}

Then we have the following lower bound.
\begin{thm}\label{thm4}
Let $T$ be a bi-parameter Journ\'e operator on $\vec{X}$ and $T$ satisfies the following ``homogeneous''
condition as above. Let $1 < p < \infty$ and $\lambda_1,\lambda_2\in A_p(\vec{X})$, and define
$\nu:= \lambda_1^{1\over p} \lambda_2^{-{1\over p}}$.
Suppose that $b\in L^1_{loc}(\vec{X})$ and
that $ \| [b,T]:\ L^p_{\lambda_1}(\vec{X})\to L^p_{\lambda_2}(\vec{X}) \|<\infty$.
Then we obtain that
$b\in {\rm bmo}_\nu(\vec{X})$ with
$$ \|b\|_{{\rm bmo}_\nu(\vec{X})}\lesssim\| [b,T]:\ L^p_{\lambda_1}(\vec{X})\to L^p_{\lambda_2}(\vec{X}) \|.$$
\end{thm}

To see this, we first point out that the homogeneous condition \eqref{e-assump cz ker low bdd weak product} implies the following condition:
\ there exist positive constants $3\le A_1\le A_2$ such that for any ball $B_i:=B_i(x_0^{(i)}, r_i)\subset X_i$, there exist balls $\widetilde B_i:=B_i(y_0^{(i)}, r_i)$ 
such that {$A_1 r_i\le d_i(x_0^{(i)}, y_0^{(i)})\le A_2 r_i$}. 
Moreover, for all $(x_1,y_1;x_2,y_2)\in ( B_1\times \widetilde{B}_1) \times ( B_2\times \widetilde{B}_2)$, $K(x_1,y_1;x_2,y_2)$ does not change sign and
\begin{equation*} 
|K(x_1,y_1;x_2,y_2)|\gs \frac1{\mu_1(B_1)}\frac1{\mu_2(B_2)}.
\end{equation*}
If $K(x_1,y_1;x_2,y_2):=K_1(x_1,y_1;x_2,y_2)+ i K_2(x_1,y_1;x_2,y_2)$ is complex-valued, where $i^2=-1$,
then at least one of $K_i$ satisfies the assumption above.

We next consider the median on ``rectangles''~$R= B_1\times B_2\subset \vec{X}$.
  By a median value of a real-valued measurable function $f$ over $R$ we mean a possibly non-unique, real number $\alpha_R(f)$ such that
$\vec\mu(\{(x_1,x_2)\in R: f(x_1,x_2)>\alpha_R(f)\})\leq \frac12\mu_1(B_1)\mu_2(B_2)\,\, \mbox{and}\,\,\vec\mu(\{(x_1,x_2)\in R: f(x_1,x_2)<\alpha_R(f)\})\leq \frac12\mu_1(B_1)\mu_2(B_2). $

Now following the idea in Lemma \ref{l-bmo decomp low bdd},
for the given rectangle $R=B_1\times B_2$,   $\widetilde B_1$ and $\widetilde B_2$, set
\begin{align*}
E_1&:=\{(x_1,x_2)\in B_1\times  B_2: b(x_1,x_2)\geq \alpha_{\widetilde B_1\times \widetilde B_2}(b)\}\\
E_2&:=\{(x_1,x_2)\in  B_1\times  B_2: b(x_1,x_2) \leq \alpha_{\widetilde B_1\times \widetilde B_2}(b)\}
\end{align*}
and
\begin{align*}
F_1&:= \{(y_1,y_2)\in \widetilde B_1\times \widetilde B_2: b(y_1,y_2)\leq \alpha_{\widetilde B_1\times \widetilde B_2}(b)\}\\
F_2&:= \{(y_1,y_2)\in \widetilde B_1\times \widetilde B_2: b(y_1,y_2)\geq \alpha_{\widetilde B_1\times \widetilde B_2}(b)\}.
\end{align*}
Then by the definition of $\alpha_R(f)$, we see that $\vec\mu(F_i)\ge\frac{1}{4}\mu_1(\widetilde B_1)\mu_2(\widetilde B_2)$ for $i=1,2$.
Moreover, for $(x_1,x_2)\times(y_1,y_2)\in (E_1\times F_1)\cup (E_2\times F_2)$,
\begin{align*}
|b(x_1,x_2)-b(y_1,y_2)|
&= \lf|b(x_1,x_2)-\alpha_{\widetilde B_1\times \widetilde B_2}(b)+\alpha_{\widetilde B_1\times \widetilde B_2}(b)-b(y_1,y_2)\r|\\
&=
\lf|b(x_1,x_2)-\alpha_{\widetilde B_1\times \widetilde B_2}(b)\r|+\lf|\alpha_{\widetilde B_1\times \widetilde B_2}(b)-b(y_1,y_2)\r|\\
&\geq \lf|b(x_1,x_2)-\alpha_{\widetilde B_1\times \widetilde B_2}(b)\r|.
\end{align*}

\begin{proof}[Proof of Theorem \ref{thm4}]
For given $b\in L^1_{\rm loc}(\vec X)$ and for any rectangle $R=B_1\times B_2$, let
$$\mathcal O(b; R):=\frac1{\vec\mu(R)}\int_R\lf|b(x_1,x_2)-b_R\r|\,d\mu_1(x_1)d\mu_2(x_2).$$
Under the assumptions of Theorem \ref{thm4},
we will show that for any ball $B$,
\begin{align}\label{e-mean osci weigh upp bdd product}
\mathcal O(b; R)\ls \frac{{\nu}(R)}{\vec\mu(R)}.
\end{align}

Without loss of generality, we assume that $K(x_1,y_1;x_2,y_2)$ is real-valued.
Let $R=B_1\times B_2$ be a rectangle. Then we have two rectangles $ B_1\times  B_2$, $\widetilde B_1\times \widetilde B_2$ and sets $E_i, F_i,\,i=1,2,$ as above.

On the one hand,  we have that for $f_i:=\chi_{F_i}$, $i=1,2$,
\begin{align*}
&\frac1{\vec\mu(R)}\sum_{i=1}^2\int_{ B_1\times  B_2}\lf| [b,T]f_i(x_1,x_2)\right|\,d\mu_1(x_1)d\mu_2(x_2)\\
&\ge\frac1{\vec\mu(R)}\sum_{i=1}^2\int_{E_i}\lf| [b,T]f_i(x_1,x_2)\right|\,d\mu_1(x_1)d\mu_2(x_2)\\
&=\frac1{\vec\mu(R)}\sum_{i=1}^2\int_{E_i}\int_{F_i}|b(x_1,x_2)-b(y_1,y_2)||K(x_1,y_1;x_2,y_2)|\,d\mu_1(y_1)d\mu_2(y_2)\, d\mu_1(x_1)d\mu_2(x_2)\\
&\gs\frac1{\vec\mu(R)}\sum_{i=1}^2\int_{E_i}\int_{F_i}\frac{|b(x_1,x_2)-\alpha_{\widetilde B_1\times \widetilde B_2}(b)|}{\mu_1(B_1)\mu_2(B_2)}\,d\mu_1(y_1)d\mu_2(y_2)\, d\mu_1(x_1)d\mu_2(x_2)\\
&\gs\frac1{\vec\mu(R)}\int_{ B_1\times  B_2}\lf|b(x_1,x_2)-\alpha_{\widetilde B_1\times \widetilde B_2}(b)\r|\, d\mu_1(x_1)d\mu_2(x_2)\\
&\gs \mathcal O(b; R).
\end{align*}

On the other hand, from H\"older's inequality and the boundedness of $[b,T]$, we deduce that
\begin{align*}
&\frac1{\vec\mu(R)}\sum_{i=1}^2\int_{ B_1\times  B_2}\lf| [b,T]f_i(x_1,x_2)\right|\,d\mu_1(x_1)d\mu_2(x_2)\\
&\quad\le \frac1{\vec\mu(R)}\sum_{i=1}^2\lf[\int_{ B_1\times  B_2}\lf|[b,T]f_i(x_1,x_2)\right|^p\lambda_2(x_1,x_2)\,d\mu_1(x_1)d\mu_2(x_2)\r]^{1/p}\\
&\quad\quad\qquad\qquad\qquad\times\lf(\int_{B_1\times  B_2}\lambda_2(x_1,x_2)^{-\frac1{p-1}}d\mu_1(x_1)d\mu_2(x_2)\right)^{1/p'}\\
&\quad\ls \frac1{\vec\mu(R)}\sum_{i=1}^2[\lambda_1(F_i)]^{1/p}\lf(\int_{ B_1\times  B_2} \lambda_2(x_1,x_2)^{-\frac1{p-1}}d\mu_1(x_1)d\mu_2(x_2)\right)^{1/p'}\\
&\quad\ls \frac1{\vec\mu(R)}[\lambda_1(\widetilde B_1\times \widetilde B_2)]^{1/p}\lf(\int_{ B_1\times  B_2} \lambda_2(x_1,x_2)^{-\frac1{p-1}}d\mu_1(x_1)d\mu_2(x_2)\right)^{1/p'}\\
&\quad\ls \frac1{\vec\mu(R)}[\lambda_1(R)]^{1/p}\lf(\int_{R}\lambda_2(x_1,x_2)^{-\frac1{p-1}}d\mu_1(x_1)d\mu_2(x_2)\right)^{1/p'},
\end{align*}
where in the last inequality, we use the facts that 
$K_1 r_{B_1}\le d(x_{B_1}, x_{\widetilde B_1})\le K_2 r_{B_1}$ and $\lambda_1(x_1,x_2)\in A_{p}(\vec X)$.

Combining the two inequalities above and invoking $\lambda_i\in A_p(\vec X)$, we conclude that
\begin{align*}
\mathcal O(b; R)\ls\frac1{\vec\mu(R)}[\lambda_1(R)]^{1/p}\lf(\int_{R}\lambda_2(x_1,x_2)^{-\frac1{p-1}}d\mu_1(x_1)d\mu_2(x_2)\right)^{1/p'}\ls\frac{\nu(R)}{\vec\mu(R)}.
\end{align*}
Thus, \eqref{e-mean osci weigh upp bdd product} holds and hence, the proof of Theorem \ref{thm4} is complete.
\end{proof}

\bigskip
\bigskip

{\bf Acknowledgments:}
X. T. Duong and J. Li are supported by ARC DP 160100153 and Macquarie University Research Seeding Grant.
B. D. Wick's research supported in part by National Science Foundation
DMS grant \#1560995 and \# 1800057. R. M. Gong is supported by NNSF of China (Grant No. 11401120) and the Foundation for Distinguished Young Teachers in Higher Education of Guangdong Province (Grant No.  YQ2015126).  D. Yang is supported by the NNSF of China (Grant No. 11571289 and 11871254).

\medskip


\smallskip

Xuan Thinh Duong, Department of Mathematics, Macquarie University, NSW, 2109, Australia.

\smallskip

{\it E-mail}: \texttt{xuan.duong@mq.edu.au}

\vspace{0.3cm}



Ruming Gong, School of Mathematical Sciences, Guangzhou University, China.

\smallskip

{\it E-mail}: \texttt{gongruming@163.com }

\vspace{0.3cm}


Marie-Jose S. Kuffner, Department of Mathematics, Washington University--St. Louis, St. Louis, MO 63130-4899 USA

\smallskip

{\it E-mail}: \texttt{mariejose@wustl.edu}

\vspace{0.3cm}



Ji Li, Department of Mathematics, Macquarie University, NSW, 2109, Australia.

\smallskip

{\it E-mail}: \texttt{ji.li@mq.edu.au}

\vspace{0.3cm}



Brett D. Wick, Department of Mathematics, Washington University--St. Louis, St. Louis, MO 63130-4899 USA

\smallskip

{\it E-mail}: \texttt{wick@math.wustl.edu}

\vspace{0.3cm}



Dongyong Yang,
 School of Mathematical Sciences,
         Xiamen University,
         Xiamen, China

{\it E-mail}: \texttt{dyyang@xmu.edu.cn}

\end{document}